\documentclass[11pt,reqno]{amsart}

\usepackage{tcolorbox, varwidth} % Boxes
\usepackage{amsthm} % For the proof environment
\usepackage{mathscinet}
\usepackage{latexsym,amsmath,amssymb, amsthm, mathtools}
\usepackage{cases, verbatim, esint, todonotes}

\usepackage{xcolor}
\usepackage{float}
\usepackage[normalem]{ulem}

\usepackage[colorlinks]{hyperref}
\tcbuselibrary{breakable}
\tcbuselibrary{theorems}
\tcbuselibrary{skins}

\usepackage[margin=1.5in]{geometry}

\definecolor{TheoremColor}{RGB}{34,139,34} % Green
\definecolor{DefColor}{RGB}{45, 52, 151} % Blue
\definecolor{CorollaryColor}{RGB}{158, 80, 143} % Purple
\definecolor{ExampleColor}{RGB}{226,135,67} % Orange
\definecolor{ProofColor}{RGB}{0,177,160} % Aqua
\definecolor{RemarkColor}{RGB}{255,0,0} % Red

\numberwithin{equation}{section}

 \newtcbtheorem[number within = section]{theorem}{Theorem}%
{breakable, enhanced,
frame empty,interior empty,colframe=TheoremColor!50!white, colback=white, 
	coltitle=TheoremColor!50!black,fonttitle=\bfseries,colbacktitle=TheoremColor!15!white,
	borderline={0.5mm}{0mm}{TheoremColor!15!white},
	borderline={0.5mm}{0mm}{TheoremColor!50!white,dashed},
	attach boxed title to top center={yshift=-2mm},
	boxed title style={boxrule=0.4pt},varwidth boxed title}{thm}

 \newtcbtheorem[number within = section, use counter from=theorem]{lemma}{Lemma}%
{breakable, enhanced,frame empty,interior empty,colframe=TheoremColor!50!white, colback=white, 
	coltitle=TheoremColor!50!black,fonttitle=\bfseries,colbacktitle=TheoremColor!15!white,
	borderline={0.5mm}{0mm}{TheoremColor!15!white},
	borderline={0.5mm}{0mm}{TheoremColor!50!white,dashed},
	attach boxed title to top center={yshift=-2mm}, 
	boxed title style={boxrule=0.4pt},varwidth boxed title}{lem}

 \newtcbtheorem[number within = section, use counter from=theorem]{proposition}{Proposition}%
{breakable, enhanced,frame empty,interior empty,colframe=TheoremColor!50!white, colback=white, 
	coltitle=TheoremColor!50!black,fonttitle=\bfseries,colbacktitle=TheoremColor!15!white,
	borderline={0.5mm}{0mm}{TheoremColor!15!white},
	borderline={0.5mm}{0mm}{TheoremColor!50!white,dashed},
	attach boxed title to top center={yshift=-2mm},
	boxed title style={boxrule=0.4pt},varwidth boxed title}{prop}

\newtcbtheorem[number within = section, use counter from=theorem]{corollary}{Corollary}%
{breakable, enhanced,frame empty,interior empty,colframe=CorollaryColor!50!white, colback=white, 
	coltitle=CorollaryColor!50!black,fonttitle=\bfseries,colbacktitle=CorollaryColor!15!white,
	borderline={0.5mm}{0mm}{CorollaryColor!15!white},
	borderline={0.5mm}{0mm}{CorollaryColor!50!white,dashed},
	attach boxed title to top center={yshift=-2mm},
	boxed title style={boxrule=0.4pt},varwidth boxed title}{cor}

\newtcbtheorem[number within = section]{definition}{Definition}%
{breakable, enhanced,frame empty,interior empty,colframe=DefColor!50!white, colback=white, 
	coltitle=DefColor!50!black,fonttitle=\bfseries,colbacktitle=DefColor!15!white,
	borderline={0.5mm}{0mm}{DefColor!15!white},
	borderline={0.5mm}{0mm}{DefColor!50!white,dashed},
	attach boxed title to top center={yshift=-2mm},
	boxed title style={boxrule=0.4pt},varwidth boxed title}{def}

\newtcbtheorem[number within = section]{example}{Example}%
{breakable, enhanced,frame empty,interior empty,colframe=ExampleColor!50!white, colback=white, 
	coltitle=ExampleColor!50!black,fonttitle=\bfseries,colbacktitle=ExampleColor!15!white,
	borderline={0.5mm}{0mm}{ExampleColor!15!white},
	borderline={0.5mm}{0mm}{ExampleColor!50!white,dashed},
	attach boxed title to top center={yshift=-2mm},
	boxed title style={boxrule=0.4pt},varwidth boxed title}{defo}

\tcolorboxenvironment{proof}{% `proof' from `amsthm'
	blanker,breakable,left=5mm,
	before skip=14pt,after skip=14pt,
	borderline west={1mm}{0pt}{ProofColor!50!white}}

\newtheoremstyle{remarkstyle}  % <name>
        {4pt}                                               % <space above>
        {6pt}                                               % <space below>
        {\normalfont}                               % <body font>
        {}                                                  % <indent amount}
        {\bfseries\color{teal}}                 % <theorem head font>
        {\normalfont\bfseries:}         % <punctuation after theorem head>
        {.5em}                                          % <space after theorem head>
        {}                                                  % <theorem head spec (can be left empty, meaning `normal')>
\theoremstyle{remarkstyle}
\newtheorem{remark}{Remark}[section]

\newtheoremstyle{problemstyle}  % <name>
        {4pt}                                               % <space above>
        {6pt}                                               % <space below>
        {\normalfont}                               % <body font>
        {}                                                  % <indent amount}
        {\bfseries}                 % <theorem head font>
        {\normalfont\bfseries:}         % <punctuation after theorem head>
        {.5em}                                          % <space after theorem head>
        {}                                                  % <theorem head spec (can be left empty, meaning `normal')>
\theoremstyle{problemstyle}
\newtheorem{problem}{Problem}[section]

\def\A{{\mathcal A}}
\def\R{{\mathbb R}}
\def\X{{\mathbf X}}
\def\Y{{\mathbf Y}}
\def\C{{\mathcal C}}

\newcommand{\pO}{\partial\Omega}
\newcommand{\Oba}{\overline{\Omega}}

\newcommand{\vep}{\varepsilon}
\newcommand{\ol}{\overline}
\newcommand{\ul}{\underline}

\newcommand{\bpm}{\begin{pmatrix}}
\newcommand{\epm}{\end{pmatrix}}
\newcommand{\la}{\langle}
\newcommand{\ra}{\rangle}

\newcommand{\beq}{\begin{equation}}
\newcommand{\eeq}{\end{equation}}
\def\lipl{{\rm Lip}_{loc}}

\def\usc{{\rm USC}}
\def\lsc{{\rm LSC}}

\title[Hamilton-Jacobi equations in metric spaces]{Hamilton-Jacobi equations in metric spaces}
\author[Qing Liu]{Qing Liu}
\address{Qing Liu, Geometric Partial Differential Equations Unit, Okinawa Institute of Science and Technology Graduate University, 1919-1, Tancha Onna-son, Okinawa 904-0495, Japan\\ Email: {\tt qing.liu@oist.jp}}

\author[Xiaodan Zhou]{Xiaodan Zhou}
\address{Xiaodan Zhou, Analysis on Metric Spaces Unit, Okinawa Institute of Science and Technology G
raduate University, 1919-1 Tancha, Onna-son, 
Okinawa 904-0495, Japan\\
Email: {\tt xiaodan.zhou@oist.jp}}

\begin{document}
	\maketitle
 
\begin{abstract}
These are lecture notes for our minicourse at OIST Summer Graduate School ``Analysis and Partial Differential Equations'' on June 12-17, 2023. We give an overview and collect a few important results concerning the well-posedness of Hamilton-Jacobi equations in metric spaces, especially several recently proposed notions of metric viscosity solutions to the eikonal equation. Basic knowledge about metric spaces and a review of viscosity solution theory in the Euclidean spaces are also presented. 
\end{abstract}

\tableofcontents

\section{Introduction}

We are concerned with first order Hamilton-Jacobi equations in metric spaces. 
The Hamilton-Jacobi equations in the Euclidean spaces are widely applied in various fields including geometric optics, optimal control, differential games, computer vision, image processing, etc.  It is well known that the notion of viscosity solutions provides a general framework for the well-posedness of first order fully nonlinear equations; we give a brief review of results related to our lectures in Section \ref{sec:hj} and refer the reader to \cite{CIL, BC} for comprehensive introduction.  

The Hamilton-Jacobi equations in a general metric space $(\X, d)$ have recently attracted great attention, see for example~\cite{AF, GaS, GHN}, in seeking to further develop various fields such as optimal transport \cite{AGS13, Vbook}, mean field games \cite{CaNotes}, topological networks \cite{SchC, IMZ,  ACCT,  IMo1, IMo2} etc.

Typical forms of the equations include
\beq\label{stationary eq}
H(x, u, |\nabla u|)=0 \quad \text{in $\Omega$,}
\end{equation}
and its time-dependent version 
\beq\label{evolution eq}
{\partial_t}u+H(x, t, u, |\nabla u|)=0 \quad \text{in $(0, \infty)\times \Omega$}
\end{equation}
with necessary boundary or initial value conditions. Here 
{$\Omega\subset\X$} is an open set and 
$H: \Omega\times \R\times [0,\infty)\to \R$ is a given continuous function called the Hamiltonian of the Hamilton-Jacobi equation. 
While $\partial_t $ denotes the time differentiation, $|\nabla u|$ stands for a generalized notion of the gradient 
norm of $u$ in metric spaces.
 
In order to discuss various approaches in the same context, we shall pay particular attention to the so-called eikonal equation
\beq\label{eikonal eq}
|\nabla u|(x)=f(x) \quad \text{in $\Omega$.}
\end{equation}
Here $f: \Omega\to (0, \infty)$ is a given continuous function. % satisfying
%\begin{equation}
%\inf_{\Omega} f>0. 
%\end{equation}
The eikonal equation in the Euclidean space has important applications in various fields such as geometric optics, electromagnetic theory and image processing \cite{KKbook, MSbook}. 

We study well-posedness of the associated Dirichlet problem; namely, we investigate uniqueness and existence of solutions to \eqref{eikonal eq} with the Dirichlet boundary condition
\beq\label{bdry cond}
u=g \quad \text{on $\pO$,}
\eeq 
where $g$ is a given bounded continuous function on $\pO$ satisfying an appropriate regularity assumption to be discussed later.

Three notions of solutions in metric spaces will be mentioned. They all have their strength and weakness. 

\begin{itemize}
    \item Curve-based solutions: This notion was proposed by \cite{GHN, Na1}. The key strategy is to consider the composition of functions and curves in $\X$ so that the equation can be converted to a one dimensional problem. This approach have been successfully applied to a wide variety of metric spaces including the Sierpinski gasket \cite{CCM}, but the definition of solutions looks complicated and the Hamiltonian seems limited only to the convex case. 
    \item Slope-based solutions: Appropriate test classes were found to substitute $C^1$ class in the Euclidean spaces so that the standard viscosity solution theory can be extended to metric spaces \cite{AF, GaS2, GaS}. This notion enables us to establish well-posedness of very general first order Hamilton-Jacobi equations. However, we need to include a length structure in $\X$ or change metric to turn $X$ into a length space. 
    \item Monge solutions: It was introduced in \cite{LShZ}, where we develop the idea from \cite{NeSu} to introduce a definition of solutions without even using any test functions. The definition is simple and convenient to use, but it also requires $\X$ to be a length space and it is not clear how to generalize this notion for general nonconvex Hamilton-Jacobi equations as well as the time dependent problems. 
\end{itemize}

It turns out that these approaches are all equivalent for the eikonal equation \eqref{eikonal eq} in a length space $\X$. This is our main goal and we shall discuss it in Section \ref{sec:equivalence}. 

We also add several exercises after each chapter to help go over the topic. Sample solutions to the exercises are available at the end of the notes. 

\subsection*{Acknowledgments}

These notes summarize our presentation for the minicourse at OIST Summer Graduate School ``Analysis and Partial Differential Equations'' held on June 12-17, 2023. We appreciate all the comments and feedback from the conference participants. 

We would particularly like to thank our OIST PhD student, Made Benny Prasetya Wiranata, for his attentive reading and helpful comments that improved the presentation of the notes. 

\clearpage

\section{Metric spaces}

%\section{Metric space}

\subsection{Preliminaries}
	We begin with the definition of a metric space. 
	\begin{definition}{Metric space}{}
	Let $\X$ be a set and let $d:\X\times \X\to \mathbb{R}$ be a function such that
    \begin{itemize}
        \item[(1)] $d(x,y)\ge 0$ for all $x,y\in \X;$
        \item[(2)] $d(x,y)=0$ if and only if $x=y;$
        \item[(3)] $d(x,y)=d(y,x)$ for all $x,y\in \X;$
        \item[(4)] $d(x,z)\le d(x,y)+d(y,z)$ for all $x,y,z\in \X.$
    \end{itemize}
    Then $(\X,d)$ is called a \emph{metric space}. The function $d$ is called a \emph{metric} or a \emph{distance function}. 
	\end{definition}
There are various examples of metric spaces. We give two of them below. 
\begin{example}{Euclidean metric}{}
		Let $\X=\mathbb{R}^n$ and $$d_E(x,y)=\sqrt{(x_1-y_1)^2+\cdots+(x_n-y_n)^2}.$$ Then $(\mathbb{R}^n,d_E)$ is a metric space.
	\end{example}
We call the above metric $d_E$ the \emph{Euclidean metric} in $\mathbb{R}^n$. Given a set $\X$, there are many different ways to define metric on $\X$. The following metric $d_T$ is sometimes called \emph{taxi metric} or \emph{Manhattan metric}. 
 \begin{example}{Taxi metric}{}
		 Let $\X=\mathbb{R}^n$ and $$d_T(x,y)=|x_1-y_1|+\cdots+|x_n-y_n|.$$ Then $(\mathbb{R}^n,d_T)$ is a metric space.
	\end{example}

\begin{definition}{Lipschitz mapping}{}
    Given two metric spaces $(\X,d_\X)$ and $(\Y,d_\Y)$, we say a mapping $f:\X\to \Y$ is a \emph{$L$-Lipschitz mapping} if there exists $L>0$ such that $d_\Y(f(x),f(y))\le Ld_\X(x,y)$ for all $x,y\in \X$. 
\end{definition}

\begin{remark}
    
If $(\Y,d_\Y)=(\mathbb{R}, |\cdot|)$, we also call the Lipschitz mapping $f:(\X,d_\X)\to (\mathbb{R}, |\cdot|)$ a Lipschitz function. 
\end{remark}

\begin{remark}
    If the inverse of a Lipschitz bijection $f:\X\to \Y$  is also a Lipschitz mapping, we say $f$ is a \emph{bi-Lipschitz mapping} between $\X$ and $\Y$ and $\X$ and $\Y$ are bi-Lipschitz equivalent. In particular, we have a constant $L\ge 1$ such that for all $x,y\in \X$
\[
L^{-1}d_\X(x,y)\le d_\Y(f(x),f(y))\le Ld_\X(x,y),
\]
if the bijection $f$ is a bi-Lipschitz mapping. %Roughly speaking, if two metric spaces $(\X,d_\X)$ and $(\Y,d_\Y)$ are bi-Lipschitz equivalent, these two spaces ``look the same" from metric point of view. 
\end{remark}

\subsection{Curves in metric spaces}

Let $(\X, d)$ be a metric space. 
A curve in $\X$ is a continuous mapping $\gamma:[a,b]\to \X$. The \emph{length} of a curve is defined as
\[
\ell(\gamma)=\sup\biggl\{\sum_{i=0}^{n-1} d(\gamma(t_i),\gamma(t_{i+1}))\biggr\},
\]
where the supremum is taken over all partitions $a=t_0<t_1<\cdots <t_n=b.$

We say a curve is a \textit{rectifiable} curve if $\ell(\gamma)<\infty.$ It is easy to see that $d(x,y)\le \ell(\gamma)$ for any curve $\gamma$ connecting $x,y$.

\begin{definition}{Rectifiably connected, length, geodesic spaces}{}
\begin{itemize}
    \item[(1)] We say a metric space is a \emph{rectifiably connected space} if for any $x,y\in \X$, there exists a rectifiable curve connecting $x$ and $y.$
    \item[(2)] We say a metric space is a \emph{length space} if for any $x,y\in \X$, for any $\varepsilon>0$, there exists a curve $\gamma$ connecting $x,y$ such that $\ell(\gamma)\le d(x,y)+\varepsilon$. 
    \item[(3)] We say a metric space is a \emph{geodesic space} if for any $x,y\in \X$, there exists a curve $\gamma$ connecting $x, y$ such that $\ell(\gamma)=d(x, y)$.
\end{itemize}
    
\end{definition}

The length function associated with a rectifiable curve $\gamma:[a,b]\to X$ is defined as $s_{\gamma}(t)=\ell(\gamma|_{[a,t]})$. We leave the proof of the following proposition as an exercise. 

\begin{proposition}{Length of a curve}{}
Let $\gamma:[a,b]\to \X$ be a curve in a metric space. 
\begin{itemize}
    \item[(1)] The length function $s_{\gamma}(t)$ is increasing and continuous. 
    \item[(2)] Assume $f:[c,d]\to [a,b]$ is a continuous, increasing and onto function. Then $\ell(\gamma)=\ell(\gamma\circ f)$.
\end{itemize}
\end{proposition}

The following theorem states that every rectifiable curve admits a nice parametrization by the arc-length.  

\begin{theorem}{Arc-length parametrization}{}
If $\gamma:[a,b]\to \X$ is a rectifiable curve, then there exists a unique curve $\tilde{\gamma}:[0,\ell(\gamma)]\to \X$ such that
\[
\gamma=\tilde{\gamma}\circ s_\gamma.
\]
Moreover, $\ell({\tilde{\gamma}|_{[0,t]}})=t$ for all $t\in [0,\ell(\gamma)]$. We call $\tilde{\gamma}$ an \textit{arc-length parametrization} of $\gamma$.
\end{theorem}

%\begin{remark}
 
%\end{remark}

\begin{remark}
The function $\tilde{\gamma}$ from above is a $1$-Lipschitz mapping. 
\end{remark}

Let us go over the notions of curve speed and length. 

\begin{definition}{Metric derivative}{}
    For a curve $\gamma:[a,b]\to \X$, we define the \emph{metric derivative} (speed) at a point $t\in (a,b)$ as the limit
    \[
    |\gamma'|(t):=\lim_{h\to 0}\frac{d(\gamma(t+h),\gamma(t))}{|h|},
    \]
    whenever the limit exists. 
\end{definition}

\begin{theorem}{Length of a Lipschitz curve}{}
For every Lipschitz curve $\gamma:[a,b]\to \X$ speed exists almost everywhere and 
\[
\ell(\gamma)=\int_a^b |\gamma'|(t)\ dt.
\]
\end{theorem}

\begin{corollary}{Speed of an arc-length parametrized curve}{}
Let $\gamma:[a,b]\to \X$ be a Lipschitz curve and let $\tilde{\gamma}:[0,\ell(\gamma)]\to X$ be the arc-length parametrization of $\gamma$. Then $|\tilde{\gamma}'|(t)=1$ almost everywhere on $[0,\ell(\gamma)]$. 
\end{corollary}

The next definition gives us the integral of a function along a rectifiable curve. 
\begin{definition}{Path integral along an arc-length parametrized curve}{}
    Let $g:\X\to [0,\infty)$ be a Borel measurable function and let $\gamma:[a,b]\to \X$ be a rectifiable curve. We define
    \[
    \int_\gamma g :=\int_0^{\ell(
    \gamma)} g(\tilde{\gamma}(t))\ dt ,
    \]
    where $\tilde{\gamma}:[0,\ell(\gamma)]\to \X$ is the arc-length parametrization of $\gamma.$ 
\end{definition}

\begin{corollary}{Path integral along a Lipschitz curve}{}
Let $\gamma:[a,b]\to \X$ be a Lipschitz curve and let $g:\X\to [0,\infty)$ be a Borel measurable function. Then
\[
\int_\gamma g=\int_a^b g(\gamma(t))|\gamma'|(t)\ dt.
\]
\end{corollary}

\subsection{Ekeland variational principle}

Let us go over the so-called \emph{Ekeland variational principle}, which is a very important and useful tool in finding the almost optimal solution to some optimization problem. See also \cite[Theorem 1.1]{Ek1}, \cite[Theorem 1]{Ek2}.

    A metric space $(\X,d)$ is called a \emph{complete} metric space if every Cauchy sequence converges to a point in $\X$. 
	\begin{theorem}{Ekeland variational principle}{}\label{Eke}
		Let $(\X,d)$ be a complete metric space and let $f:\X\to \mathbb{R}\cup \{\infty\}$ be a lower semicontinuous function. Assume $f\not\equiv \infty$ and is bounded from below. For any $\varepsilon>0$ and any $x_0\in \X$ satisfying $f(x_0)<\infty$, there exists $x_\varepsilon$ such that
\[
f(x_\varepsilon)\le f(x_0)-\varepsilon d(x_0, x_\varepsilon),
\]
and 
\[
f(x_\varepsilon)<f(x)+\varepsilon d(x,x_\varepsilon)\quad\quad \text{for all $x\in \X\setminus\{x_\varepsilon\}$}.
\]
	\end{theorem}
	\begin{proof}
		Define a set value function $S:\X\to 2^\X$ as 
\[
S(x)=\{y\in \X: f(y)+\varepsilon d(x,y)\le f(x)\}.
\]

We can get the following properties of $S$: 
\begin{itemize}
    \item[(1)] For any $x\in \X$, $S(x)$ is nonempty and if $y\in S(x)$ and $z\in S(y)$, then $z\in S(x)$;
    \item[(2)] Fix $x$, the function $g: \X\to \mathbb{R}$ defined as
    $g(y)=f(y)+\varepsilon d(x,y)$ is a lower semicontinuous function. Hence $S(x)=\{y\in \X: g(y)\le f(x)\}$ is a closed set.
\end{itemize}

It suffices to show that there exists a point $x_\varepsilon\in S(x_0)$ and $S(x_\varepsilon)=\{x_\varepsilon\}$. 

Let $\alpha_1=\inf\{f(x): x\in S(x_0)\}$. It is clear that $\alpha_1\in \R$, since $f$ is bounded below and $f\not\equiv \infty$. There exists a point $x_1\in S(x_0)$ such that
\[
\alpha_1\le f(x_1)\le\alpha_1+\frac{\varepsilon}{2^{n+1}}.
\]
Define a sequence inductively in the following way: 
\begin{itemize}
    \item[(i)] Let $\alpha_{n+1}=\inf\{f(x): x\in S(x_n)\} $;
    \item[(ii)] Let $x_{n+1}\in S(x_n)$ be such that
    \[
   \alpha_{n+1}\le f(x_{n+1})\le \alpha_{n+1}+\frac{\varepsilon}{2^{n}}.
    \]
\end{itemize}

It is easy to observe that $\alpha_n\le \alpha_{n+1}$ as $S(x_{n+1})\subset S(x_n).$ Furthermore, for any $z\in S(x_n)$,
\begin{equation}\label{contraction}
    \varepsilon d(x_n, z)\le f(x_n)-f(z)\le f(x_n)-\alpha_{n}\le \varepsilon/2^{n-1}.
\end{equation}
Let $z=x_{n+1}$, we can conclude that $\{x_n\}$ is a Cauchy sequence and hence converges to a point $x_\varepsilon\in \X.$ From the construction, we have $x_\varepsilon\in S(x_n)$ for all $n\in \mathbb{Z}$ and $n\ge 0$. The same holds if $x\in S(x_\varepsilon)$. Apply \eqref{contraction}, we obtain that $d(x,x_n)\le \frac{1}{2^n}$ for any $n\in \mathbb{N}$. Let $n\to \infty$, we obtain that $x=x_\varepsilon.$ The proof is complete. 
	\end{proof}
	By replacing $f$ by $-f$, we can obtain the following version of Ekeland's variational principle, which will be used later. 
	\begin{corollary}{Another version of Ekeland variational principle}{ekeland}
		Let $(\X,d)$ be a complete metric space and let $f:\X\to \mathbb{R}\cup \{\infty\}$ be an upper semicontinuous function. Assume $f\not\equiv -\infty$ and is bounded from above. Let $\delta>0$ be given and a point $x_0\in X$ such that 
\[
f(x_0)\ge \sup_\X f-\delta.
\]
For any $\lambda>0$, there exists a point $x_\lambda\in \X$ such that
\begin{itemize}
    \item[(1)] $f(x_\lambda)\ge f(x_0)$;
    \item[(2)] $d(x_\lambda,x_0)\le \lambda$;
    \item[(3)] For all $y\neq x_\lambda$, $
    f(x_\lambda)> f(y)-\frac{\delta}{\lambda}d(y,x_\lambda). $
\end{itemize}
	\end{corollary}
	\begin{remark}\label{rmk ekeland}
In our later application, we will usually take $\delta=\varepsilon^2$ and $\lambda=\varepsilon$. In other words, we have $x_\vep\in \X$ satisfying $d(x_\vep, x_0)\leq \vep$ and $f-\vep d(\cdot, x_\vep)$ attains a strict maximum at $x_\vep$.
\end{remark}

Actually, if a lower semicontinuous function $f:\X\to \mathbb{R}\cup\{\infty\}$ which is bounded from below and $f\not\equiv \infty$ satisfies properties listed in Theorem \ref{Eke}, the space $X$ must be complete. See Exercise \ref{Eketocomplete} below. In other words, the Ekeland variational principle provides a characterization of complete metric spaces \cite{Su81}.

\subsection{Exercises}
\begin{problem}
    Show $(\mathbb{R}^n,d_E)$ and $(\mathbb{R}^n, d_T)$ are bi-Lipschitz equivalent. 
\end{problem}

\begin{problem}

Let $(\X,d)$ be a rectifiably connected metric space and $f: \X\to [\alpha, \infty)$ be a function with $\alpha>0$. Let $\Gamma_f(x,y)$ be a collection of curves in $\X$ such that
\[
\Gamma_f(x, y)=\left\{\gamma:[0, \ell]\to \X: \int_\gamma f\, ds<\infty,\ \gamma(0)=x, \gamma(\ell)=y \text{ and } |\gamma'|(s)=1 \text{ for a.e. }s\right\}.
\]
Assume $\Gamma_{f}(x,y)\neq \emptyset$ for all $x,y\in \X$. Define 
\[
L_f(x,y)=\inf_{\gamma\in \Gamma_f(x,y)} \int_\gamma f\ ds.
\]
\begin{itemize}
    \item[(1)] Prove $(\X,L_f)$ is a metric space. 
    \item[(2)] If $(\X,d)$ is complete, then $(\X,L_f)$ is complete.    
\end{itemize}
\end{problem}  

\begin{comment}
\begin{problem}
 Let $\X=[-1,1]$ and let $f$ be a function defined as
\[
f(x)=\frac{1}{\sqrt{|x|}}, \quad x\in [-1, 1]\setminus\{0\}.
\]
\begin{itemize}
 \item[(1)] Show $(X,L_f)$ is bounded but is not bi-Lipschitz equivalent with $(X,d_E)=(X, |\cdot|)$.
\item[(2)] Show for $L_f$ defined in (3), for all $x,y\in X$, $d_{E}(x,x_n)\to 0$ if and only if $L_f(x,x_n)\to 0$ for all $x,x_n\in X$.
\end{itemize}
\end{problem}

\begin{problem}

In $\mathbb{R}^2$, let $e_0$ be the closed line segment between points $(0, 0)$ and $(1, 0)$, and $e_j$ be the line 
segments connecting $P_j=(1/j, 0)$ and $Q_j=(1/j, 1/j)$ for $j=1, 2, \ldots$ We take 
$\X=\bigcup_{j=0, 1, \ldots} e_j$ with $d$ being the intrinsic metric of this graph. Show there exists $x_n, x\in X$ such that $d(x_n,x)\to 0$ and $L_f(x_n,x)>0$. 

\end{problem}
\end{comment}

%\begin{itemize}
 %   \item[(1)] Show $(X,L_f)$ is complete. 
  %  \item[(2)] Show there exists $x_n, x\in X$ such that $d(x_n,x)\to 0$ and $L_f(x_n,x)>0$. 
%\end{itemize}

\begin{problem}
Assume that $\gamma:[a,b]\to \X$ is a rectifiable curve and $s_\gamma(t):[a,b]\to [0,\ell(\gamma)]$ is the length function. Show $s_\gamma$ is continuous.  
\end{problem}

\begin{comment}

\begin{problem}
 Assume $\gamma:[a,b]\to \X$ be a mapping (not necessarily continuous). Define the length of $\gamma$ as that of a curve. If $\ell(\gamma)<\infty$, prove $\gamma$ has at most countably many points of discontinuity.
\end{problem}

\end{comment}

\begin{problem}\label{Eketocomplete}
Let $(\X,d)$ be a metric space and let $f:\X\to \mathbb{R}\cup \{\infty\}$ be a lower semicontinuous function. Assume $f\not\equiv \infty$ and is bounded from below. Then $(\X,d)$ is complete if and only if for any $\varepsilon>0$, there exists a point $x_\varepsilon\in \X$ such that
\begin{itemize}
    \item[(1)] $f(x_\varepsilon)\le \inf_\X f+\varepsilon$;
    \item[(2)] For all $y\neq x_\varepsilon$, $
    f(x_\varepsilon)< f(y)+\varepsilon d(y,x_\varepsilon). $
\end{itemize}
\end{problem}

\begin{comment}

Given an interval $I=[a,b]\subset\R$ and a continuous function (curve) $\gamma:I\to\Omega$, we define the length 
of $\gamma$ by
\[
\ell(\gamma):=\sup_{a=t_0<t_1<\cdots<t_k=b} \sum_{j=0}^{k-1}d(\gamma(t_j),\gamma(t_{j+1})).
\]
Note that if $\ell(\gamma)<\infty$, then the real-valued function $s_\gamma:[a,b]\to[0,\ell(\gamma)]$ defined by
\[
s_\gamma(t):=\ell(\gamma\vert_{[a,t]})
\]
is a monotone increasing function, and hence it is differentiable at almost every $t\in[a,b]$. We denote
$s_\gamma'$ by $|\gamma^\prime|$, which can be equivalently defined by 
\[
|\gamma'|(t)=\lim_{\tau\to 0} {d(\gamma(t+\tau), \gamma(t))\over |\tau|}
\]
for almost every $t\in (a, b)$. Note that if $\gamma$ is an absolutely continuous curve
(and so $s_\gamma$ is absolutely continuous real-valued function), then
\[
s_\gamma(t)=\int_a^t|\gamma^\prime|(\tau)\, d\tau \quad\text{for all $t\in [a,b]$.}
\]
\end{comment}

\clearpage

\section{Hamilton-Jacobi equations in Euclidean spaces}\label{sec:hj}

Let us briefly go over the standard theory of viscosity solution theory in the Euclidean space. One can find more details in \cite{CIL, BC}. Another nice reference for beginners is \cite{KoBook}. We shall focus on the stationary equation 
\beq\label{stationary eq1}
H(x, u(x), \nabla u(x))=0 \quad \text{in $\Omega$,}
\eeq
where $\Omega\subset \R^n$ is assumed to be a bounded domain. 
The definition and many properties of viscosity solutions to \eqref{evolution eq} are analogous. We omit the details here. 

\subsection{Optimal control problem}\label{sec:control}

We begin with a simple example from optimal control theory, which is one of the motivations for studying Hamilton-Jacobi equations and viscosity solution theory. 

Suppose that $\Omega\subset \R^n$ is a bounded domain. We aim to move a point from $x\in \Oba$ to the boundary $\partial \Omega$ with speed limit $1$. What is the minimal arrival time?
%Fix $x\in \Oba$ and consider rectifiable curves $\gamma$ connecting $x$ to the boundary $\partial \Omega$. We can parametrize such a curve $\gamma$ by its arc-length; that is, $\gamma: [0, s_\gamma]\to \Oba$ satisfies $\gamma(0)=x$, $\gamma(s_\gamma)=y\in \partial \Omega$, $\gamma(t)\in \Omega$ for all $t\in (0, s_\gamma)$ and $|\gamma'|(t)=1$ for almost every $t\in (0, s_\gamma)$. Let $C_x$ denotes the set of all such curves.  Our goal is to minimize the cost 
%\[
%J(x; \gamma)=\int_{\gamma} f\, ds+g(y).
%\]
To answer this, we denote the trajectory of motion by $y(t)$ for $t\geq 0$. Then we have 
\[
y'(t)=\alpha(t), \quad y(0)=x,
\]
where $\alpha\in L^\infty(0, \infty; \R^n)$ and $|\alpha|\leq 1$ holds almost everywhere. We call $\alpha$ a control function, which takes value in the closed unit ball $B_1(0)\subset\R^n$ centered at $0$. We call the set the control set associated to $\alpha$. 

We define the first arrival time for a given control by
\[
t_x(\alpha):=\inf\{t\geq 0: y(t)\in \partial \Omega\}.
\]
The so-called \textit{value function} is then defined as 
\[
u(x)=\inf_{\alpha} t_x(\alpha), \quad x\in \Oba.
\]
%In the special case that $f\equiv 1$ in $\Omega$, $g\equiv 0$ on $\partial \Omega$, it becomes an optimal exit time problem and 
It is easily seen that $u(x)=d(x, \partial \Omega)$ for $x\in \Oba$. Here, $d(x, \partial \Omega)$ denotes the distance from $x$ to $\partial \Omega$, namely, $d(x, \partial \Omega)=\inf_{y\in \partial \Omega} |x-y|$.
Furthermore, %for such choices of $f$ and $g$, 
if $n=1$ and $\Omega=(-1, 1)$, we have
$u(x)=\min\{1+x, 1-x\}=1-|x|$ for $-1\le x\le 1$.

In general, it is not difficulty to verify that the value function is $u(x)=d(x, \partial \Omega)$ for all $x\in \Oba$ and it is attained at a minimizing control $\alpha\equiv a_x:={(z-x)/|z-x|}$, where $z\in \partial \Omega$ is a point satisfying \begin{equation}\label{bdry minimizer}
|z-x|=d(x, \partial\Omega).
\end{equation}
It is also clear that, for any $x\in \Omega$,
\begin{equation}\label{dpp intro}
u(x)=\inf_{y\in B_r(x)}\{u(y)+|y-x|\},\quad \text{for all $r>0$ small, }
\end{equation}
where $B_r(x)$ denotes the ball centered at $x$ with radius $r$. 
It is a straightforward property of the distance function $d(x, \partial\Omega)$. In fact, it is immediate to get $\leq$  by the triangle inequality, and for $\geq$, we see that 
$u(x)=u(y)+|y-x|$ holds if we take $y=x+ra_x/2\in B_r(x)$. Our proof above also yields that \eqref{dpp intro} holds with $B_r(x)$ replaced by $B_r(x)\setminus\{x\}$.

When considering more general optimal control problems, for which the value function may not be explicitly computed, we can still obtain an equality similar to \eqref{dpp intro}. This relation is called \textit{Dynamic Programming Principle (DPP)} and plays a key role in our lectures.  

The DPP enables us to derive a PDE associated to this optimal control problem, at least formally. Assume that $u$ is differentiable at $x\in \Omega$. Then by Taylor expansion on \eqref{dpp intro} with punctured ball $B_r(x)\setminus \{x\}$, at $x\in \Omega$ we have 
\[
u(x)=\inf_{y\in {B_r(x)\setminus\{x\}}}\{u(x)+\la\nabla u(x), y-x\ra+|y-x|\}+o(|y-x|)
\]
for $r>0$ small. Dividing both sides by $|y-x|$ and sending $r\to 0$, we are led to 
\[
\sup_{w\in \partial B_1(0)}\la\nabla u(x), w\ra=1,
\]
which yields $|\nabla u(x)|=1$. This leads us to the eikonal equation
\begin{equation}\label{eikonal1}
|\nabla u(x)|=f(x) \quad \text{in $\Omega$}
\end{equation}
with $f\equiv 1$ in $\Omega$. 
%which finds applications not only in optimal control problems but also in geometric optics, image processing, etc. 
Note that despite slight difference in appearance, \eqref{eikonal1} is actually the same as \eqref{eikonal eq}; we find the notation $|\nabla u|(x)$ makes better sense in metric spaces. We assume that $f>0$ in $\Omega$.

Our derivation above relies on the differentiability of solutions. However, in general one cannot expect existence of differentiable solutions to the eikonal equation with even homogeneous Dirichlet boundary condition. Consider again the case $n=1$ and $\Omega=(-1, 1)$. Any differentiable function $u$ satisfying $|u'(x)|=1$ for every $x$ must be linear with constant slope $1$ or $-1$. It thus cannot satisfy the boundary condition $u(\pm1)=0$ at the same time. Hence, a notion of weak solutions is needed to solve the PDE properly. It turns out that the \textit{viscosity solution theory} provides a successful framework.

\subsection{Definition of viscosity solutions}

Below let us go over the viscosity solution theory for the first order Hamilton-Jacobi equations. Denote by $\usc(\Omega)$ and $\lsc(\Omega)$ respectively the classes of upper semicontinuous and lower semicontinuous functions in $\Omega$.

\begin{definition}{Viscosity solutions}{def vis}
    A function $u\in \usc(\Omega)$ is called a viscosity subsolution of \eqref{stationary eq1} if 
\beq\label{hj-sub-intro}
H(x, u(x), \nabla\psi(x))\leq 0
\eeq
holds for any $x\in \Omega$ and $\psi\in C^1(\Omega)$ such that $u-\psi$ attains a local maximum in $\Omega$, 
 A function $u\in \lsc(\Omega)$ is called a viscosity supersolution of \eqref{stationary eq1} if 
\beq\label{hj-super-intro}
H(x, u(x), \nabla\psi(x))\geq 0
\eeq
holds for any $x\in \Omega$ and $\psi\in C^1(\Omega)$ such that $u-\psi$ attains a local minimum in $\Omega$. We say that $u\in C(\Omega)$ is a viscosity solution of \eqref{stationary eq1} if it is both a subsolution and a supersolution of \eqref{stationary eq1}.
We sometimes also say that $u$ satisfies $H(x, u, \nabla u)=0$ (resp., $\leq$, $\geq$) in the viscosity sense if it is a viscosity solution (resp., subsolution, supersolution).
\end{definition}
The function $\psi$ is usually called a test function. 
Adding appropriate quadratic functions to such test function, one can replace the local maximum/minimum in the definition above by local strict maximum/minimum. Moreover, it is also possible to restrict the test functions $\psi$ to the class $C^\infty(\Omega)$.  

Given the boundary condition \eqref{bdry cond}, we say that $u$ is a viscosity solution (resp., subsolution, supersolution) of the associated Dirichlet problem if it is a viscosity solution of \eqref{stationary eq1} and $u=g$ (resp., $u\leq g$, $u\geq g$) holds on $\partial \Omega$.
%Below we shall suppress the word ``viscosity'' for simplicity, as this is the only kind of solutions we will consider throughout the notes. 

Such type of weak solutions indeed generalizes the notion of classical solutions.

\begin{lemma}{Consistency with classical solutions}{lem:property1}
Let $u\in C^1(\Omega)$. Then $u$ is a viscosity subsolution (resp., supersolution) of \eqref{stationary eq1} if and only if  \eqref{hj-sub-intro} (resp., \eqref{hj-super-intro}) holds at all $x\in \Omega$ in the classical sense.  
\end{lemma}

In the special case with $H(x, r, p)=|p|-1$, which corresponds to the eikonal equation \eqref{eikonal1} with $f\equiv 1$, we can now justify that $u(x)=d(x, \partial \Omega)$ is a viscosity solution satisfying the boundary condition $u=0$ on $\partial \Omega$.

\begin{proposition}{Viscosity solution to eikonal equation}{eikonal sol-intro}
Let $\Omega$ be a bounded domain in $\R^n$. Then, $u(x)=d(x, \partial \Omega)$ ($x\in \Omega$) is a viscosity solution of \eqref{eikonal1} with $f\equiv 1$ in $\Omega$. 
\end{proposition}
\begin{proof}
    It is obvious that $u$ is a continuous function in $\Omega$. Let us verify the viscosity subsolution property. Our proof follows our previous formal derivation of eikonal equation but with use of test functions. Fix $x\in \Omega$. Suppose that there is $\psi\in C^1(\Omega)$ such that $u-\psi$ attains a local maximum at $x$. Since \eqref{dpp intro} holds, we have 
    \[
    \psi(x)\leq \inf_{y\in B_r(x)} \{\psi(y)+|y-x|\}
    \]
    for $r>0$ small. Then by Taylor expansion, we have
    \[
    0\leq \inf_{y\in B_r(x)} \{\la \nabla \psi(x), y-x\ra+|y-x|\}+o(|y-x|).
    \]
    Dividing the inequality by $|y-x|$ and letting $r\to 0$, we deduce 
    \[
    |\nabla \psi(x)|\leq 1.
    \]
    Therefore $u$ is a subsolution of \eqref{eikonal1} with $f\equiv 1$. One can similarly show the supersolution property.
\end{proof}

%\begin{remark}{Viscosity solutions versus weak solutions}
As there are no classical solutions to the eikonal equation \eqref{eikonal1}, it is tempting to consider the so-called (Lipschitz) \textit{weak solutions}. A Lipschitz function $u$ is said to be a Lipschitz weak solution of \eqref{eikonal1} if it satisfies \eqref{eikonal1} almost everywhere in $\Omega$. However, it is not difficult to see that there are infinitely many such kind of weak solutions even with boundary data assigned; see Figure \ref{fig:weak-sol}.

The notion of viscosity solutions selects among all of the weak solutions the most physical one. The superiority of viscosity solution to the eikonal equation can be evidenced by the following aspects:
\begin{enumerate}
    \item It is closely connected to applications in optimal control theory, as we have seen in Section \ref{sec:control}.
    \item It is the most regular candidate solution: it is the only \textit{semiconcave} weak solution. A function $f\in C(\Omega)$ is said to be semiconcave if $f(x)-C|x|^2$ is concave for some $C\in \R$.
    \item It can be obtained by the \textit{vanishing viscosity method}. A good exercise related to this would be Problem \ref{prob:vanishing viscosity}. 
\end{enumerate}

\begin{figure}[H]
    \centering
    \includegraphics[width=0.4\textwidth]{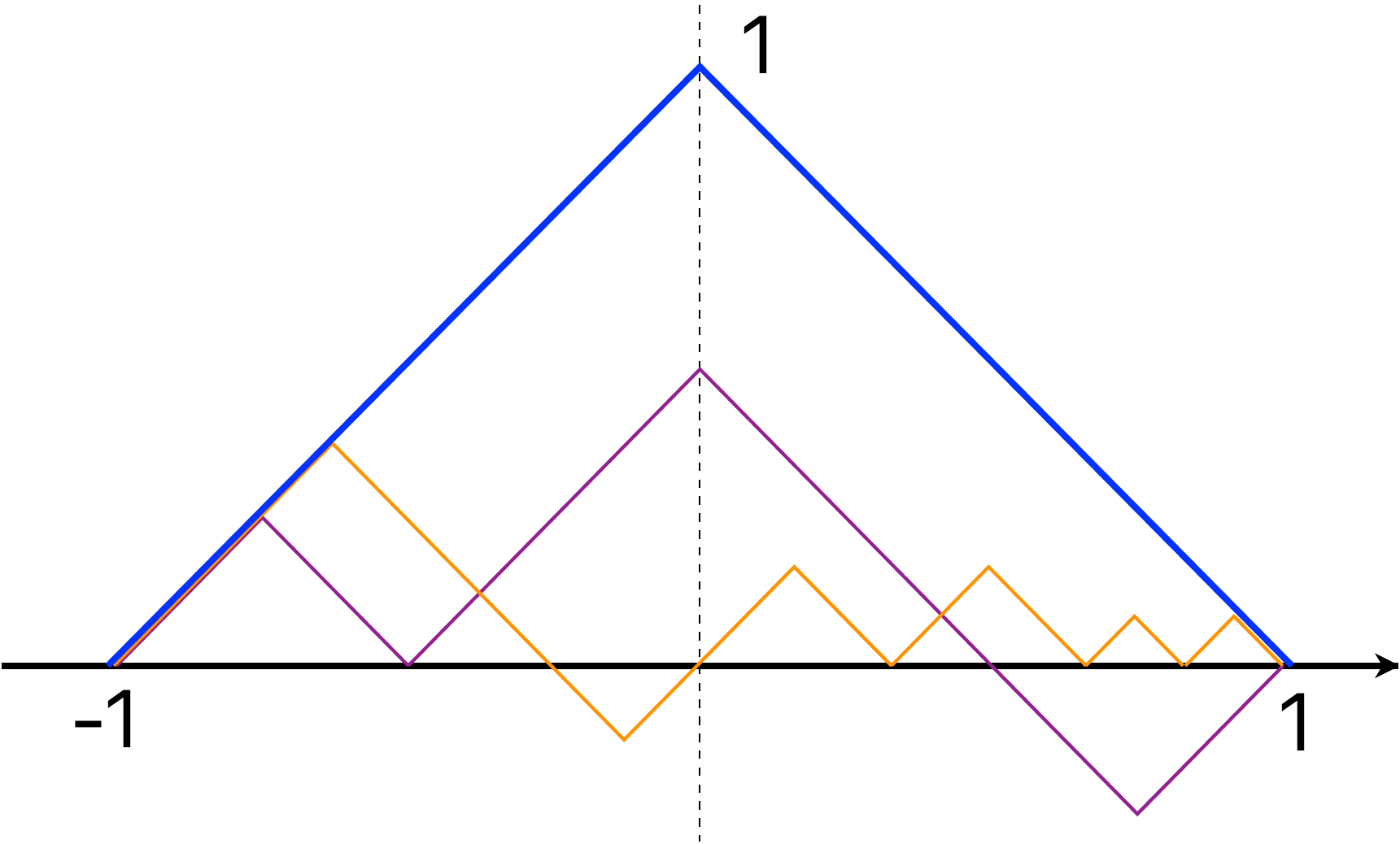}
    \caption{Weak solutions to $|u'|=1$ in $(-1, 1)$ with homogeneous boundary data}
    \label{fig:weak-sol}
\end{figure}

An interesting remark is that $H(x, u(x), \nabla u(x))=0$ and $0=H(x, u(x), \nabla u(x))$ are not the same equations in the viscosity solution theory. This is because viscosity solutions are defined through inequalities. Another possible explanation pertaining to the vanishing viscosity approximation is about the sign of viscosity term we  include; adding $\vep \Delta u$ or $-\vep \Delta u$ in the equation could make a world of difference to the limit of its solution as $\vep\to 0$.

%\end{remark}

\subsection{Comparison principle}

\emph{Comparison principle} undoubtedly plays a pivotal role in the viscosity solution theory. It not only implies uniqueness of solutions but also can facilitate us to show existence, regularity and other properties. % In this section, we present two comparison theorems, one for the eikonal equation \eqref{eikonal1}, the other for more general Hamilton-Jacobi equation as in \eqref{stationary eq1} but with $H(x, r, p)$ strictly monotone in $r$.

Let us begin with a comparison principle for \eqref{stationary eq1}. 
\begin{theorem}{Comparison principle for HJ equation}{comparison-intro2}
    Let $\Omega$ be a bounded domain in $\R^n$. Assume that $H: \Omega\times \R\times \R^n\to \R$ satisfies the following conditions:
    \begin{enumerate}
    \item[(i)] $r\mapsto H(x, r, p)$ is strictly increasing in the sense that there exists $\lambda>0$ such that for all $x\in \Omega, r, s\in \R$ and $p\in \R^n$, we have
    \[
    H(x, r, p)-H(x, s, p)\geq \lambda(r-s) \quad \text{if $r>s$.}
    \]
    \item[(ii)] $x\mapsto H(x, r, p)$ is uniformly continuous in the sense that there exists a modulus of continuity $\omega: [0, \infty)\to [0, \infty)$ (continuous, increasing with $\omega(0)=0$) such that for all $x, y\in \Omega, r\in \R$ and $p\in \R^n$, we have 
    \[
    |H(x, r, p)-H(y, r, p)|\leq \omega((1+|p|)|x-y|).
    \]
    \end{enumerate}
    Suppose that $u\in USC(\Oba)$ and $v\in LSC(\Oba)$ are respectively a viscosity subsolution and a viscosity supersolution of \eqref{stationary eq1}. If $u\leq v$ on $\partial \Omega$, then $u\leq v$ in $\Oba$. 
\end{theorem}
\begin{proof}
%The proof is similar to that of Theorem \ref{theo:comparison-intro1}. 
Suppose by contradiction that $\sigma:=\max_{\Oba}u-v>0$. 
In what follows, we use the so-called \textit{doubling variables} technique. 
Consider 
\[
\Phi_\vep(x, y)=u(x)- v(y)-{1\over \vep}|x-y|^2,\quad  x, y\in \Oba
\]
for $\vep>0$ small. We can find a maximizer $(x_\vep, y_\vep)$ of $\Phi_\vep$ and therefore
\[
u(x_\vep)-v(y_\vep)-{1\over \vep}|x_\vep-y_\vep|^2\geq \sigma.
\]
This yields
\[
|x_\vep-y_\vep|^2\leq \vep\left(u(x_\vep)- v(y_\vep)-\sigma\right).
\]
It follows again that $x_\vep, y_\vep$ converges to some $x_0\in \Oba$ along a subsequence. Moreover, thanks to the semicontinuity of $u$ and $v$, we further obtain 
\[
\limsup_{\vep\to 0}{1\over \vep}|x_\vep-y_\vep|^2\leq u(x_0)-v(x_0)-\sigma\leq 0;
\]
in other words, 
\begin{equation}\label{doubing var1}
{1\over \vep}|x_\vep-y_\vep|^2\to 0, \quad \text{as $\vep\to 0$.}
\end{equation}
Due to the relation $u\leq v$ on $\partial \Omega$, we see that $x_0\in \Omega$ and thus $x_\vep, y_\vep\in \Omega$. 

Set
\[
\begin{aligned}
&\psi_1(x)= v(y_\vep)+{1\over \vep}|x-y_\vep|^2,\\
&\psi_2(y)=u(x_\vep)-{1\over \vep}|x_\vep-y|^2.
\end{aligned}
\]
We see that $u(x)-\psi_1(x)$ attains a local maximum at $x=x_\vep$ and $v-\psi_2$ attains a local minimum at $y=y_\vep$. It follows from \ref{def:def vis} that 
\begin{equation}\label{doubling var2}
H(x_\vep, u(x_\vep), \nabla \psi_1(x_\vep))\leq 0,
\end{equation}
\begin{equation}\label{doubling var3}
H(y_\vep, v(y_\vep), \nabla\psi_2(y_\vep))\geq 0.    
\end{equation}
Note that \eqref{doubling var5} still holds in this case. 

Since $u(x_\vep)-v(y_\vep)\geq \sigma$, by the assumption (i), we deduce from \eqref{doubling var2} that
\begin{equation}\label{doubling var4}
H(x_\vep, v(y_\vep), \nabla \psi_1(x_\vep))\leq H(x_\vep, u(x_\vep), \nabla \psi_1(x_\vep))-\lambda(u(x_\vep)-v(y_\vep))\leq -\lambda\sigma.
\end{equation}
Combining \eqref{doubling var3} and \eqref{doubling var4}, and applying the assumption (ii), we then have
\[
\lambda\sigma \leq H(y_\vep, v(y_\vep), \nabla \psi_2(y_\vep))-H(x_\vep, v(y_\vep), \nabla \psi_1(x_\vep))\le \omega\left(\left(1+{2\over \vep}|x_\vep-y_\vep|\right)|x_\vep-y_\vep|\right).
\]
Letting $\vep\to 0$ and utilizing \eqref{doubing var1}, we end up with $\lambda\sigma\leq 0$, which is a contradiction. 
\end{proof}

The comparison principle above immediately implies the uniqueness of viscosity solutions to the Dirichlet problem for \eqref{stationary eq1}. 

We next present a comparison principle for \eqref{eikonal1}. One possible approach to use the \emph{Kru\v{z}kov transform} 
\[
U(x)=-e^{-u(x)} \quad x\in \Omega
\]
to reduce \eqref{eikonal1} to a Hamilton-Jacobi equation
\[
|\nabla U|+f(x)U=0
\]
We leave the verification of this fact to the reader; see Problem \ref{prob:kruzkov}.
If $\inf_{\Omega}f>0$ and $f$ is continuous in $\Oba$, then the equation is equivalent to 
\[
{1\over f(x)}|\nabla U|+U=0,
\]
where the Hamiltonian satisfies the assumptions (i) and (ii) of Theorem \ref{thm:comparison-intro2}. We can apply the theorem to obtain a comparison result. We below provide an alternative method, based on the homogeneity of the left hand side of \eqref{eikonal1}, under a weaker positivity condition on $f$.

\begin{theorem}{Comparison principle for eikonal equation}{comparison-intro1}
    Let $\Omega$ be a bounded domain in $\R^n$ and $f$ be a positive continuous function in $\Omega$.  Suppose that $u\in USC(\Oba)$ and $v\in LSC(\Oba)$ are respectively a viscosity subsolution and a viscosity supersolution of \eqref{eikonal1}. If $u\leq v$ on $\partial \Omega$, then $u\leq v$ in $\Oba$. 
\end{theorem}
\begin{proof}
 Note first that the upper semicontinuity of $u$ and the lower semicontinuity of $v$ respectively imply the  boundedness of $u$ from above and $v$ from below in $\Oba$. %We may assume that $v>0$ in $\Oba$, since it is not difficult to verify that, for any $C\in \R$, $u+C$ and $v+C$ are still subsolution and supersolutions respectively. 
Assume by contradiction that $\max_{\Oba} (u-v) >0$. Then fixing $c\in (0, 1)$ close to $1$, we have 
\[
\beta:=\max_{\Oba}(cu-v)>0
\]
and by the condition that $u\leq v$ on $\partial\Omega$, 
\begin{equation}\label{cp boundary eq1}
cu-v\leq c(u-v)+(c-1)v\leq \beta/2\quad \text{on $\partial \Omega$.}
\end{equation}
%Let us prove that $u\leq cv$ for any fixed $c>1$. 
Below we again double variables. Let $\vep>0$ small and consider 
\[
\Phi_\vep(x, y)=cu(x)- v(y)-{1\over \vep}|x-y|^2,\quad  x, y\in \Oba. 
\]
Then there exists $(x_\vep, y_\vep)\in \Oba\times \Oba$ such that $\Phi_\vep$ attains a maximum at $(x_\vep, y_\vep)$. It follows that 
\[
cu(x_\vep)- v(y_\vep)-{1\over \vep}|x_\vep-y_\vep|^2\geq \beta,
\]
which implies 
\[
|x_\vep-y_\vep|^2\leq \vep\left(cu(x_\vep)-v(y_\vep)\right).
\]
Due to the upper boundedness of $u$ and $-c v$, we can get $x_0\in \Oba$ such that $x_\vep, y_\vep\to x_0$ as $\vep\to 0$ via a subsequence. By the upper semicontinuity of $u$ and lower semicontinuity of $v$, we obtain $cu(x_0)-v(x_0)\geq \beta>0$.

In view of \eqref{cp boundary eq1}, we thus have $x_0\in \Omega$, which yields $x_\vep, y_\vep\in \Omega$ as well when $\vep>0$ is sufficiently small. We next adopt the definition of sub- and supersolutions. Setting
\[
\begin{aligned}
&\psi_1(x)=v(y_\vep)+{1\over \vep}|x-y_\vep|^2,\\
&\psi_2(y)=cu(x_\vep)-{1\over \vep}|x_\vep-y|^2,
\end{aligned}
\]
we see from the maximality of $\Phi_\vep$ at $(x_\vep, y_\vep)$ that $u-\psi_1/c$ attains a maximum at $x=x_\vep$ and $v-\psi_2$ attains a minimum at $y=y_\vep$. Then, by Definition \ref{def:def vis}, we have 
\begin{equation}\label{doubling var}
|\nabla \psi_1(x_\vep)|\leq cf(x_\vep),\quad |\nabla \psi_2(y_\vep)|\geq  f(y_\vep).    
\end{equation}
Since 
\begin{equation}\label{doubling var5}
\nabla \psi_1(x_\vep)=\nabla \psi_2(y_\vep)={2\over \vep}(x_\vep-y_\vep),
\end{equation}
we are led to $cf(x_\vep)\geq  f(y_\vep)$. Letting $\vep\to 0$, we end up with $c f(x_0)\geq f(x_0)$, which is an obvious contradiction to the positivity of $f$ in $\Omega$. 
\end{proof}
\begin{remark}\label{rmk:doubling-eikonal}
    Our comparison argument (doubling variable technique) above can be applied to handle the following result: if $u$ and $v$ respectively satisfy $|\nabla u|\leq f_1$ and $|\nabla u|\geq f_2$ in a bounded domain $\Omega$ in the viscosity sense with $f_1, f_2\in C(\Omega)$ fulfilling $0<f_1<f_2$ in $\Omega$, then $\max_{\Oba}(u-v)\leq \max_{\partial \Omega}(u-v)$ holds.
\end{remark}

As an immediate result of Theorem \ref{thm:comparison-intro1} together with Proposition \ref{prop:eikonal sol-intro}, we obtain the following. 

\begin{corollary}{Unique viscosity solution to eikonal equation}{}
Let $\Omega$ be a bounded domain in $\R^n$. Then, $u(x)=d(x, \partial \Omega)$ ($x\in \Omega$) is the unique viscosity solution of \eqref{eikonal1} satisfying \eqref{bdry cond} with $f\equiv 1$ in $\Omega$ and $g\equiv 0$ on $\partial \Omega$. 
\end{corollary}

\subsection{Optimal control interpretation}\label{sec:control eucl}

Once the comparison principle is established, there are multiple ways to show the existence of viscosity solutions, including the vanishing viscosity approach and Perron's method. We shall continue our discussion in Section \ref{sec:control} to build viscosity solutions by using optimal control formulations. 

Let $\Omega$ be a domain in $\R^n$ and $A$ be a closed bounded set in $\R^m$. Consider a \textit{control system} governed by the \textit{state equation}
\begin{equation}\label{state eq}
\begin{aligned}
&y'(t)=q(y(t), \alpha(t)), \quad t>0\\
&y(0)=x,
\end{aligned}    
\end{equation}
where $x\in \Oba$ represents the initial state, $y(t)$ denotes the state of system at time $t$, $\alpha$ is a control from the set $\A:=\{\alpha: [0, \infty)\to A\ \text{measurable}\}$,  the dynamics $q: \Oba\times A\to \R^n$ is a given bounded continuous function that satisfies the Lipschitz continuity in the first variable: there exists $L>0$ such that
\[
|q(x, a)-q(y, a)|\leq L|x-y| \quad \text{for all $x, y\in \Oba$ and $a\in A$.}
\]
For any fixed $\alpha\in \A$, the standard theory ordinary differential equations ensures the existence and uniqueness of a solution $y=y(\cdot; \alpha):[0, t_x(\alpha)]\to \Oba$ of \eqref{state eq} with $y(t_x(\alpha))\in \partial\Omega$ and $y(t)\in \Omega$ for all $t\in (0, t_x(\alpha))$. Here $t_x(\alpha)$ denotes the first exit time, i.e., 
\[
t_x(\alpha)=\min\{t\geq 0: y(t)\in \partial \Omega\}.
\]
Let us introduce the running cost function $h: \Oba\times \A\to \R$, which is assumed to be continuous and uniformly continuous in the first variable: there exists a modulus of continuity $\omega_\ell$ such that 
\[
|h(x, a)-h(y, a)|\leq \omega_\ell(|x-y|)\quad \text{for all $x, y\in \Oba$ and $a\in A$.}
\]
Let $g: \partial \Omega\to \R$ be a continuous function representing the terminal cost. 
We define the cost functional as 
\begin{equation}\label{cost eucl}
J(x; \alpha)=\int_0^{t_x(\alpha)} e^{-\lambda t}h(y(t), \alpha(t))\, dt+e^{-\lambda t_x(\alpha)}g(y(t_x(\alpha))),
\end{equation}
where $\lambda\geq 0$ denotes a discount rate. The value function is defined by
\begin{equation}\label{general value}
u(x)=\inf_{\alpha\in \A} J(x; \alpha), \quad x\in \Oba.
\end{equation}
The following DPP holds in this general case.
\begin{theorem}{Dynamical programming principle}{general dpp}
Under the assumptions on $q, h$ above, at any $x\in \Omega$, the value function $u$ defined by \eqref{general value} satisfies 
\begin{equation}\label{general dpp}
    u(x)=\inf_{\alpha\in \A} \left\{e^{-\lambda \tau}u(y(\tau))+\int_0^{\tau} e^{-\lambda t}h(y(t; \alpha), \alpha(t))\, dt\right\}
\end{equation}
for any $\tau>0$ small. 
\end{theorem}
This result is based on the following separation of cost:
\[
\begin{aligned}
J(x; \alpha)& =\int_0^{\tau}+\int_{\tau}^{t_x(\alpha)} e^{-\lambda t}h(y(t), \alpha(t))\, dt+e^{-\lambda t_x(\alpha)}g(y(t_x(\alpha)))\\
&=\int_{\tau}^{t_x(\alpha)} e^{-\lambda t}h(y(t), \alpha(t))\, dt+e^{-\lambda\tau-\lambda (t_x(\alpha)-\tau)}g(y(t_x(\alpha)))+\int_0^{\tau}e^{-\lambda t}h(y(t), \alpha(t))\, dt\\
&=e^{-\lambda \tau}\left(\int_0^{t_{\tilde{x}}(\alpha)} e^{-\lambda t}h(\tilde{y}(t), \alpha(t))\, dt+e^{-\lambda t_{\tilde{x}}(\alpha)} g(\tilde{y}(t_{\tilde{x}}(\alpha)))\right)+ \int_0^{\tau}e^{-\lambda t}h(y(t), \alpha(t))\, dt,
\end{aligned}
\]
where $\tilde{x}=y(\tau)$ and $\tilde{y}$ denotes the ODE solution with the same $\alpha$ but initial value $\tilde{y}(0)=\tilde{x}$; it is clear that $t_{\tilde{x}(\alpha)}=t_x(\alpha)-\tau$. Optimizing the cost over $\alpha\in \A$ first for the time period $[\tau, t_{{x}}(\alpha)]$ and then for $[0, \tau]$, we obtain \eqref{general dpp} at least formally. 

We omit the rigorous proof for this general control setting but focus on the case when $A=\ol{B_1(0)}\subset \R^n$, $\lambda=0$, $q(x, a)=a$ and $h(x, a)=f(x)$ with a given positive continuous function $f:\Oba \to (0, \infty)$.
Note that in this special case, the cost functional in \eqref{cost eucl} reduces to 
\[
J(x; \alpha)=\int_0^{t_x(\alpha)} f(y(t))\, dt+g(y(t_x(\alpha))),
\]
where $y'(t)=\alpha(t)\in \overline{B_1(0)}$ for almost every $t\in (0, t_x(\alpha))$.
Since $f>0$ in $\Oba$ and thus
\[
\int_0^{t_x(\alpha)} f(y(t))\, dt\geq \int_0^{t_x(\alpha)} f(y(t))|y'|(t)\, dt,
\]
taking the arc-length parametrization $\ol{y}(s)=y(t)$ with $s=\int_0^t |y'|(t)\, dt$ decreases the total cost, i.e., 
\begin{equation}\label{curve change}
\int_0^{t_x(\alpha)} f(y(t))\, dt+g(y(t_x(\alpha)))\geq \int_0^{s_x(\alpha)} f(\ol{y}(s))\, ds+g(\ol{y}(s_x(\alpha))),
\end{equation}
where $\ol{y}(s_x(\alpha))=y(t_x(\alpha))$. Certainly $|\ol{y}'|(s)=1$ holds for almost every $s\in (0, s_x(\alpha))$. In other words, we can express %$u(x)$ can be understood as the infimum of 
\begin{equation}\label{value eucl}
u(x)=\inf_{\gamma\in C_x}\left\{\int_0^{t_x} f(\gamma(s))\, ds+g(\gamma(t_x))\right\}, \quad x\in \Oba, 
\end{equation}
%over all arc-length parametrized curves $\gamma\in C_x$, 
where $C_x$ denotes the collection of curves in $\Omega$ connecting $x\in \Oba$ to $\partial \Omega$, that is, 
\begin{equation}\label{curve set eucl}
\begin{aligned}
C_x:=\bigg\{  \gamma: & [0, t_x]\to \Oba \text{ Lipschitz with $t_x<\infty$, $\gamma(0)=x$, } \\ %$\gamma(t_x)\in \partial \Omega$,
&  \text{$|\gamma'|(t)=1$ for almost every $t\in (0, t_x)$}\bigg\},
\end{aligned}
\end{equation}
where  $t_x=\inf\{t\geq 0: \gamma(t)\in \partial\Omega\}$.
Note that $t_x$ denotes the first arrival time at $\partial \Omega$ along $\gamma$ when starting from $x\in \Oba$; in particular $t_x=0$ for any $x\in \partial\Omega$. We should have written $t_x[\gamma]$ to indicate its dependence on $\gamma$, but we suppress the bracket for simplicity of notation. 

Let us use this alternative definition of value function to prove DPP in this case. 

\begin{theorem}{Dynamical programming principle for eikonal equation}{dpp eikonal eucl}
Let $\Omega$ be a bounded domain in $\R^n$ and $f\in C(\Oba)$ with $f\geq 0$ in $\Oba$. %. Let $A=\ol{B_1(0)}\subset \R^n$, $\lambda=0$, $q(x, a)=a$ and $h(x, a)=f(x)$ for $x\in \Oba$ and $a\in A$ with a given $f\in C(\Oba)$ satisfying $\inf_{\Oba} f>0$. 
Let $u: \Omega\to \R$ be defined by \eqref{value eucl}. Then, at any $x\in \Omega$, 
\begin{equation}\label{dpp eikonal eucl eq}
    u(x)=\inf_{\gamma\in C_x}\left\{u(\gamma(\tau))+\int_{0}^\tau f(\gamma(s))\, ds\right\}
\end{equation}
holds for any $\tau>0$ small.
\end{theorem}
\begin{proof}
For $x\in \Omega$, we can take $\tau>0$ small such that $\gamma(t)\in \Omega$ for all $\gamma\in C_x$ and $t\in (0, \tau)$. 
%$\vep>0$ there exists 
For every $\gamma\in C_x$, we take a portion $\gamma^\tau(s)=\gamma(s+\tau)$ ($s\in [0, t_x^\tau]$) of the curve $\gamma$ such that $\gamma^\tau\in C_{\gamma(\tau)}$, $t_x^\tau=t_x-\tau$, and
\begin{equation}\label{dpp eucl eq1}
\int_0^{t_x} f(\gamma(s))\, ds+g(\gamma(t_x))=\int_0^\tau f(\gamma(s))\, ds+ \int_0^{t_x^\tau} f(\gamma^\tau(s))\, ds+g(\gamma^\tau(t_x^\tau)).
 \end{equation}

Keeping the value of $\gamma$ in $[0, \tau]$ and taking the infimum of the right hand side over curves in $C_{\gamma(\tau)}$, we obtain
\[
\int_0^{t_x} f(\gamma(s))\, ds+g(\gamma(t_x))\geq \int_0^\tau f(\gamma(s))\, ds+ u(\gamma(\tau)).
\]
Taking the infimum over $\gamma\in C_x$ next of both sides, we have 
\[
u(x)\geq \inf_{\gamma\in C_x}\left\{\int_0^\tau f(\gamma(s))\, ds+ u(\gamma(\tau))\right\}.
\]
On the other hand, for every $\gamma\in C_x$ and $\vep>0$, take $\gamma^\tau\in C_{\gamma(\tau)}$ such that 
\[
\int_0^{t_x^\tau} f(\gamma^\tau(s))\, ds+g(\gamma^\tau(t_x^\tau))\leq u(\gamma(\tau))+\vep.
\]
Then we immediately obtain from \eqref{dpp eucl eq1} that 
\[
u(x)\leq \int_0^\tau f(\gamma(s))\, ds+ \int_0^{t_x^\tau} f(\gamma^\tau(s))\, ds+g(\gamma^\tau(t_x^\tau))\leq \int_0^\tau f(\gamma(s))\, ds+u(\gamma(\tau))+\vep.
\]
Since $\gamma\in C_x$ is arbitrary, letting $\vep\to 0$ yields
\[
u(x)\leq \inf_{\gamma\in C_x}\left\{\int_0^\tau f(\gamma(s))\, ds+ u(\gamma(\tau))\right\}.
\]
Our proof is now complete.
\end{proof}
\begin{remark}
One can actually take $A=\partial B_1(0)\subset \R^n$ in the general optimal control setting to define the value function $u$, which satisfies \eqref{dpp eikonal eucl eq}.
\end{remark}

Let us show the connection between this control problem and the eikonal equation \eqref{eikonal1}.

%\clearpage

\begin{theorem}{Optimal control interpretation for eikonal equation}{}
Suppose that the same assumptions in Theorem \ref{thm:dpp eikonal eucl} hold. Then the value function $u$ is a viscosity solution of \eqref{eikonal1}. In addition, if $\inf_{\Omega} f>0$ and $g\in C(\partial \Omega)$ satisfies
\begin{equation}\label{compatibility eucl}
\begin{aligned}
g(x)\leq g(y)+\int_{\gamma} f\, ds \quad &\text{for any arc-length parametrized $\gamma$}\\
& \hspace{3cm} \text{in $\Oba$ joining $x, y\in \partial \Omega$,}
\end{aligned}
\end{equation}
then $u\in C(\Oba)$ is the unique viscosity solution of \eqref{eikonal1} satisfying \eqref{bdry cond}.
\end{theorem}
\begin{proof}
We first show that $u$ is continuous in $\Omega$. For any $x\in \Omega$ and $\vep>0$, by \eqref{value eucl} there exists $\gamma\in C_x$ such that 
\[
u(x)\geq \int_0^{t_x} f(\gamma(s))\, ds+g(\gamma(t_x))-\vep.
\]
Then for any $y\in B_r(x)$ with $r>0$ small, we can build $\tilde{\gamma}\in C_y$ by concatenating $\gamma$ and the line segment $l_{xy}$ between $x, y$ and therefore
\[
u(y)\leq \int_{\tilde{\gamma}} f\, ds+g(\tilde{\gamma}(t_x))\leq u(x)+\int_{l_{xy}} f\, ds+\vep.
\]
Since $f>0$ is continuous, sending $r\to 0$ and $\vep\to 0$ we have $\limsup_{y\to x} u(y)\leq u(x)$. %Due to the arbitrariness of $\vep>0$, we are led to $\limsup_{y\to x} u(y)\leq u(x)$. 

On the other hand, for $y\in B_r(x)\setminus \{x\}$ with $r>0$ small, by \eqref{dpp eikonal eucl eq}, we can choose an arc-length parametrized curve $\gamma\in C_x$ first joining $x, y$ by the line segment $l_{xy}$, that is, $\gamma(t)=x+t(y-x)/|y-x|$ for any $0\leq t\leq |y-x|$.  This implies
\[
u(x)\leq u(y)+\int_{l_{xy}} f\, ds.
\]
It follows that $u(x)\leq \liminf_{y\to x} u(y)$ again due to the continuity of $f$. We thus obtain the continuity of $u$ in $\Omega$.

Let us next show that $u$ is a viscosity solution. Suppose that there exists $\psi\in C^1(\Omega)$ and $x\in \Omega$ such that $u-\psi$ attains a local maximum in $\Omega$. Then, by \eqref{dpp eikonal eucl eq}, for any $\tau>0$ we have
\[
\psi(x)\leq \inf_{\gamma\in C_x}\left\{\psi(\gamma(\tau))+\int_{0}^\tau f(\gamma(s))\, ds\right\}.
\]
% we deduce 
%\[
%0\leq \inf_{\gamma\in C_x}\left\{\la\nabla \psi(x), \gamma(\tau)-x\ra +\int_{0}^\tau %f(\gamma(s))\, ds\right\}+o(\tau).
%\]
For any $a\in \partial B_1(0)$, we can choose $\gamma\in C_x$ such that $\gamma(t)=x+at$ for all $t\in [0, \tau]$. By Taylor expansion, it follows that 
\[
0\leq \inf_{a\in \partial B_1(0)} \left\{\la\nabla \psi(x), a \tau\ra+ \int_{0}^\tau f(x+as)\, ds\right\}+o(\tau).
\]
Dividing the inequality by $\tau$ and then letting $\tau\to 0$, we deduce that 
\[
-\min_{a\in \partial B_1(0)}\{\la \nabla \psi(x), a\ra\}\leq f(x),
\]
which is equivalent to $|\nabla \psi(x)|\leq f(x)$. 

For the supersolution property, we have $\psi\in C^1(\Omega)$ and $x\in \Omega$ such that $u-\psi$ attains a local minimum in $\Omega$. By the DPP \eqref{dpp eikonal eucl eq}, for any $\tau>0$ and $\vep>0$, there exists $\gamma_\vep\in C_x$ such that 
\[
u(x)\geq \left\{u(\gamma_\vep(\tau))+\int_{0}^\tau f(\gamma_\vep(s))\, ds\right\}-\vep \tau.
\]
%\[
%\psi(x)\geq \inf_{\gamma\in C_x}\left\{\psi(\gamma(\tau))+\int_{0}^\tau f(\gamma(s))\, ds\right\}
%\]
%for any $\tau>0$ small. For each $\vep>0$, 
The local minimality of $u-\psi$ at $x$ thus implies that 
\[
\psi(x)\geq \left\{\psi(\gamma_\vep(\tau))+\int_{0}^\tau f(\gamma_\vep(s))\, ds\right\}-\vep \tau
\]
for $\tau>0$ sufficiently small. Adopting Taylor expansion, we obtain 
\[
0\geq \left\{\la\nabla \psi(x), \gamma_\vep(\tau)-x\ra +\int_{0}^\tau f(\gamma_\vep(s))\, ds\right\}-\vep \tau +o(\tau).
\]
We can take a subsequence so that as $\tau\to 0$, $(\gamma_\vep(\tau)-x)/\tau\to a$ for some $a\in \ol{B_1(0)}$.

Dividing it by $\tau>0$ and passing to the limit along this subsequence, we use the continuity of $f$ to get
\[
0\geq \la \nabla \psi(x), a\ra+ f(x)-\vep,
\]
which yields $|\nabla \psi(x)|\geq f(x)-\vep$. Due to the arbitrariness of $\vep>0$, we obtain $|\nabla \psi(x)|\geq f(x)$ as desired.

Let us finally show that $u\in C(\Oba)$ and it is the unique viscosity solution of the Dirichlet problem. It is clear from the definition of $u$ in \eqref{value eucl} that $u= g$ on $\partial \Omega$. %By \eqref{compatibility eucl}, we also have $g\leq u$ on $\partial \Omega$. We thus get $u=g$ on $\partial \Omega$. 
%Below let us verify $u\in \partial \Omega$. 
For any $x\in \Omega$, one can find a line segment $l_{xy}$ connecting $x$ to $y\in \partial \Omega$ with length $|x-y|=d(x, \partial \Omega)$. By \eqref{value eucl}, we thus have 
\[
u(x)\leq |x-y| \sup_{\Oba} f +g(y)=d(x, \partial \Omega)\sup_{\Oba} f +g(y).
\]
Suppose that $x\to x_0\in \partial \Omega$. Due to the fact that $d(x, \partial \Omega)\leq |x-x_0|$, we have
\[
|y-x_0|\leq d(x, \partial \Omega)+|x-x_0|\leq 2|x-x_0|.
\]
Noticing that
\[
u(x)-g(x_0)\leq |x-x_0|\sup_{\Oba} f +|g(y)-g(x_0)|,
\]
by the continuity of $g$ on $\partial\Omega$, we obtain 
\[
\limsup_{x\to x_0} u(x)\leq g(x_0)= u(x_0).
\]

On the other hand, for every $\vep>0$, there exists $\gamma_\vep\in C_x$ with $\gamma_\vep(t_x)\in \partial \Omega$ such that 
\begin{equation}\label{compatibility eucl eq2}
u(x)\geq \int_0^{t_x} f(\gamma_\vep(s))\, ds+g(\gamma_\vep(t_x))-\vep.
\end{equation}
For $y\in \partial \Omega$ with $|y-x|=d(x, \partial \Omega)$, we can concatenate $l_{xy}$ and $\gamma_\vep$ to build a curve $\tilde{\gamma}_\vep\in C_y$ joining $y$ and $\gamma_\vep(t_x)$. It follows from \eqref{compatibility eucl} that 
\begin{equation}\label{compatibility eucl eq3}
g(y)\leq g(\gamma_\vep(t_x))+\int_{\tilde{\gamma}_\vep} f\, ds\leq g(\gamma_\vep(t_x)) +\int_0^{t_x} f(\gamma_\vep(s))\, ds+d(x, \partial \Omega)\sup_{\Oba} f.
\end{equation}
Combining \eqref{compatibility eucl eq2} and \eqref{compatibility eucl eq3}, we are led to 
\[
u(x_0)=g(x_0)\leq |g(x_0)-g(y)|+u(x)+d(x, \partial \Omega)\sup_{\Oba} f+\vep.
\]
Letting $\vep\to 0$ and then $x\to x_0$, we apply the continuity of $g$ again to deduce that 
$g(x_0)\leq \liminf_{x\to x_0} u(x)$. We have shown that $u(x)\to g(x_0)=u(x_0)$ as $x\to x_0\in \partial \Omega$ and thus completed the proof of continuity of $u$ in $\Oba$. We finally apply the comparison result, Theorem \ref{thm:comparison-intro1}, to conclude that $u$ is the only viscosity solution to \eqref{eikonal1} with the Dirichlet boundary condition \eqref{bdry cond}.
\end{proof}

The compatibility condition \eqref{compatibility eucl} is needed for the existence of viscosity solutions in $C(\Oba)$ satisfying the boundary data, as revealed by the following example. 
\begin{example}{}{}
Let $\Omega=(-1, 1)\subset \R$, $f\equiv 1$ in $\Omega$, $g(-1)=0$ and $g(1)=3$. Then it is clear that \eqref{compatibility eucl} fails to hold, since its right hand side at $x=1$ is equal to $\min\{g(1), g(-1)+\int_{-1}^1 \, ds\}=2$ but $g(1)=3$. %In this case, by our definition $u(1)= 2$ and it does not meet the boundary condition. 
If we take $u(1)=3$ to maintain the boundary value, then $u$ fails to be continuous in $\Oba$.
\end{example}

For the general setting of optimal control problem introduced at the beginning of Section \ref{sec:control eucl}, under the continuity of the value function $u$, one can use the associated DPP \eqref{general dpp} in Theorem \ref{thm:general dpp} to prove that $u$ is a viscosity solution to 
\begin{equation}\label{control hj eq}
\lambda u(x)+\sup_{a\in A} \{-\la q(x, a), \nabla u(x)\ra-h(x, a)\}=0\quad \text{in $\Omega$.}
\end{equation}
Prove this in Problem \ref{prob:general HJ}. Note that in general the continuity of $u$ in $\Oba$ may not hold and $u$ may not satisfy the boundary condition \eqref{bdry cond}. These cause difficulty to establish the uniqueness of viscosity solutions in this case. We refer to \cite[Chapter IV]{BC} for detailed discussions on the continuity and uniqueness of solutions under certain controllability conditions on the control system.

\subsection{Exercises}
\begin{problem}\label{prob:vanishing viscosity}
    Find the unique classical solution $u_\vep$ to the viscous eikonal equation 
    \[
    |u'|^2-\vep u''=1 \quad \text{in $(-1, 1)\subset \R$}
    \]
    with $\vep>0$ satisfying $u(\pm 1)=0$. Show that $u_\vep(x)\to 1-|x|=d(x, \{\pm1\})$ uniformly in $[-1, 1]$ as $\vep\to 0$. Use the property that $u_\vep(x)$ is an even function in $(-1, 1)$.
\end{problem}

\begin{problem}\label{prob:kruzkov}
    Let $\Omega$ be a domain in $\R^n$ and $f\in C(\Omega)$ satisfy $f>0$ in $\Omega$. Show that $u$ is a viscosity solution of $|\nabla u(x)|=f(x)$ in $\Omega$ if and only if $U(x)=-e^{-u(x)}$ is a viscosity solution of $|\nabla U|+f(x)U=0$ in $\Omega$.
\end{problem}

\begin{problem}
    Let $\Omega$ be a domain in $\R^n$. Assume that $u$ and $v$ are both viscosity subsolutions of $H(x, u(x), \nabla u(x))=0$ in $\Omega$. Show that $w(x)=\max\{u(x), v(x)\}$ is also a viscosity subsolution. Give an example to show that in general $\min\{u, v\}$ is not necessarily a viscosity subsolution. 
\end{problem}

\begin{problem}\label{prob:general HJ}
Suppose that $u\in C(\Omega)$ satisfies DPP in \eqref{general dpp}. Show that $u$ is a viscosity solution of \eqref{control hj eq}.
\end{problem}

\clearpage

\section{Metric viscosity solutions}

We discuss several different types of notions of viscosity solutions in metric spaces. Let $(\X, d)$ be a metric space and $\Omega\subsetneq \X$ is a bounded domain. Our main focus is the eikonal equation \eqref{eikonal eq}, where $f$ is a given positive continuous function in $\Omega$.

\subsection{Curve-based solutions}

For any interval $I\subset \R$, %we say an absolutely continuous curve $\gamma: I\to \X$ is %{admissible} if 
%\[
%|\gamma'|\leq 1\quad  \text{a.e. in $I$.}
%\]
%Note that these are $1$-Lipschitz curves.
let $\A(I, S)$ denote an \emph{admissible curve} class, which is the set of all arc-length parametrized curves in a set $S\subset \X$
defined on $I$; without loss of generality,
we only consider intervals $I$ for which $0\in I$. For any $x\in \X$, 
we write $\gamma\in \A_x(I, S)$ if  
$\gamma(0)=x$. We also need to define the \emph{exit time} and \emph{entrance time} of a curve $\gamma$ from a domain $\Omega$:
\begin{equation}\label{exit/entrance}
\begin{aligned}
T^+_\Omega[\gamma]&:=\inf\{t\in I: t\geq 0, \gamma(t)\notin \Omega\};\\
T^-_\Omega[\gamma]&:=\sup\{t\in I: t\leq 0, \gamma(t)\notin \Omega\}.
\end{aligned}
\end{equation}
Note that the original admissible class used in \cite{GHN} is the set of absolutely continuous curves that have speed below 1. Since any absolutely continuous curve is locally rectifiable, meaning that it has finite length in every bounded parameter interval, and one can always reparametrize a rectifiable curve by its arc-length (cf. \cite[Theorem 3.2]{Haj1}), hereafter we do not distinguish 
the difference between an absolutely continuous curve and a locally rectifiable (or Lipschitz) curve. 
Moreover, it turns out that changing the curve speed to be $1$ almost everywhere does not affect the notion of solutions and related discussions later; the reason is somewhat similar to our explanation for \eqref{curve change} in the Euclidean case. In the sequel, we thus simply adopt the class of arc-length parametrized Lipschitz curves for our convenience of presentation.  

The following definition is from \cite[Definition 2.1 ]{GHN} but with our modified curve class.
\begin{definition}{Curve-based solutions}{defi c}
A function $u\in \usc(\Omega)$ is called a \emph{curve-based viscosity subsolution} or 
\emph{c-subsolution} of \eqref{eikonal eq} if for any $x\in\Omega$ and $\gamma\in \A_x(\R, \Omega)$, we have 
\beq\label{eq c-sub}
\left|\phi'(0)\right|\leq f(x)
\eeq
whenever $\phi\in C^1(\R)$  
such that $t\mapsto u(\gamma(t))-\phi(t)$ attains a local maximum at $t=0$.

A function $u\in \lsc(\Omega)$ is called a 
\emph{curve-based viscosity supersolution} or  
\emph{c-supersolution} of \eqref{eikonal eq} if 
for any $\vep>0$ and $x\in \Omega$, there exists $\gamma\in \A_x(\R, \X)$ and $w\in \lsc(T^-, T^+)$ with 
$-\infty<T^{\pm}=T^{\pm}_\Omega[\gamma]<\infty$ ($(\gamma, w)$ is called an \textit{$\vep$-pair}) such that 
\begin{equation}\label{wapprox}
w(0)=u(x), \quad w\geq u\circ\gamma-\vep,
\end{equation}
and 
\begin{equation}\label{sgapprox}
\left|\phi'(t_0)\right|\geq f(\gamma(t_0))-\vep
\end{equation}
whenever  $\phi\in C^1(\R)$ such that $t\mapsto w(t)-\phi(t)$ attains a local minimum at $t=t_0\in (T^-, T^+)$. 
A function $u\in C(\Omega)$ is said to be a 
\emph{curve-based viscosity solution} or 
\emph{c-solution} if it is both a c-subsolution and a 
c-supersolution of \eqref{eikonal eq}. 
\end{definition}

In the definition of supersolutions, in general we cannot merely replace $w$ with
$u\circ\gamma$. Suppose that $\X$ is not a geodesic space but a length space. When $f\equiv 1$, as 
we expect that the distance function $u=d(\cdot, x_0)$ is still a solution for any $x_0\in \X$, the 
supersolution property for $u\circ\gamma$ without approximation would imply that $\gamma$ is a geodesic, 
which is a contradiction.

The regularity of $u$ above can be  relaxed, since we only need its semicontinuity along each curve $\gamma$.  
In fact, one can require a c-subsolution (resp., c-supersolution) to be merely arcwise upper  (resp., lower) 
semicontinuous; consult \cite{GHN} for details. %However, in order to obtain our main results, we need to impose the conventional continuity of $u$ rather than the arcwise continuity.

%The notions of c-subsolutions and c-supersolutions with respect to the metric $d$ and the intrinsic metric $\tilde{d}$ given in \eqref{int metric} are equivalent.  Note that $\gamma$ is a curve with respect to $d$ if and only if it is a curve with
%respect to $\tilde{d}$ because of our assumption~\eqref{eq ast}. It can also be seen that the speed $|\gamma'|$ of the curve remains the same in both metrics; see Lemma~\ref{lem length}.  Hence, the class of admissible curves in Definition \ref{defi c} does not depend on the choice between $d$ and $\tilde{d}$. 

Although there seems to be no requirement on the metric space in the definition above,  
it is implicitly assumed in the definition of the c-supersolution that each point $x\in \Omega$ can be connected to the boundary $\partial\Omega$ by a curve of finite length. 

The definition of c-supersolutions is actually equivalent to super-optimality conditions as below. 
Hereafeter, we denote $C_x=\A_x([0, \infty), \Oba  )\text{ with $T_\Omega^+[\gamma]<\infty$}$. (Note that this notion of $C_x$ is consistent with that in \eqref{curve set eucl} in the Euclidean case.)

\begin{proposition}{Super-optimality characterization of c-supersolutions}{c-super char}
Assume that $f\in C(\Omega)$ satisfies $\inf_{\Omega} f>0$.
%\begin{equation}\label{f-lower}
%\sigma:=\inf_{\Omega} f>0.
%\end{equation}
Let $u\in \lsc(\Omega)$.  Then, the following hold:
\begin{enumerate}
  \item If for any $x\in \Omega$,
\begin{equation}\label{c-super char}
u(x)\geq \inf_{\gamma\in C_x}\left\{ u(\gamma(t))+\int_0^t f(\gamma(s))\, ds\right\},
\end{equation}
then $u$ is a c-supersolution of \eqref{eikonal eq}. 
\item If $u$ is a c-supersolution of \eqref{eikonal eq}, then for any $\vep>0$, there exists $\gamma\in C_x$ such that
\begin{equation}\label{c-super char2}
u(\gamma(t))\leq u(x)-\int_0^t f(\gamma(s))\, ds+\vep (t+1).
\end{equation}
holds for all $t\in (0, T_\Omega^+[\gamma])$.
\end{enumerate}
\end{proposition}

\begin{remark}
   If $u$ is further assumed to be bounded from below, then in the statement (2), $T_\Omega^+[\gamma]$ is bounded uniformly with respect to $\gamma\in C_x$ and $\vep\in (0, \inf_{\Omega}f/2)$. Then we see that $u$ is a c-supersolution of \eqref{eikonal eq} if and only if $u$ satisfies \eqref{c-super char}.
\end{remark}

\begin{proof}[Proof of Proposition \ref{prop:c-super char}]
Let us prove (1). By \eqref{c-super char}, for any $\vep>0$, there exists $\gamma\in C_x$ ($\gamma\in \A_x([0, \infty), \X)$ with $T:=T^+_\Omega[\gamma]<\infty$) such that 
\[
%\inf_{\gamma\in C_x}\left\{ u(\gamma(t))+\int_0^t f(\gamma(s))\, ds\right\} 
u(x)\geq u(\gamma(t))+\int_0^t f(\gamma(s))\, ds-\vep.
\]
This $\gamma$ can be modified to be a curve $\tilde{\gamma}$ in $\A_x(\R, \X)$ by taking $\tilde{\gamma}(t)=\gamma(-t)$ for $t<0$. Define $w:(-T, T)\to \R$ by
\[
w(t)=u(x)-\int_0^{|t|} f(\gamma(s))\, ds.
\]
Then it follows that $w(t)\geq u(\tilde{\gamma}(t))-\vep$. Note that along this curve, $w$ is just a viscosity supersolution of $|w'(t)|=f(\tilde{\gamma}(t))$ in $(-T, T)$. It satisfies \eqref{sgapprox} for any test function  $\phi$. 
\medskip

Let us show (2). By the definition of c-supersolutions, for every $\vep<\inf_{\Omega}f$, we can find an $\vep$-pair $(\gamma, w)$ with $\gamma\in \A_{x}(\R, \X)$ with $T_\pm=T^{\pm}_\Omega[\gamma]<\infty$. Since $w(t)$ satisfies the viscosity supersolution $|w'(t)|\geq f(\gamma(t))-\vep>0$ in $(-T, T)$, $w$ cannot have local minimizers. We thus can assume that $w(t)$ is nonincreasing in $(0, T)$, for otherwise we can change the orientation of the curve.  

We can further show that $t\mapsto w(t)+\int_0^t f(\gamma(s))\, ds-\vep t$ is nonincreasing in $(0, T)$; see Problem \ref{prob:monotone}. It follows that, for all $t\in (0, T)$,
\[
w(t)\leq w(0)-\int_0^t f(\gamma(s))\, ds+\vep t.
\]
which yields \eqref{c-super char2}, as $w(0)=u(x)$ and $w(t)\geq u(\gamma(t))-\vep$. 
\end{proof}

There is also a similar but simpler characterization for c-subsolutions shown in \cite[Proposition 2.6]{GHN}. 

\begin{proposition}{Sub-optimality characterization of c-subsolutions}{c-sub}
Assume that $f\in C(\Omega)$ with $f\ge 0$ in $\Omega$. 
Let $u\in \usc(\Omega)$. Then $u$ is a c-subsolution of \eqref{eikonal eq} if and only if 
\beq\label{prop c-sub eq}
u(\gamma(t_2))-u(\gamma(t_1))\leq \int_{t_1}^{t_2} f(\gamma(s))\, ds
\eeq
for all $\gamma\in \mathcal{A}(\mathbb{R}, \Omega)$ and $t_1, t_2\in \mathbb{R}$ with $t_1<t_2$.
\end{proposition}
\begin{proof}
    The proof is based on a direct application of the equivalence between viscosity inequalities and monotonicity of single variable functions, as stated in Problem \ref{prob:monotone0}. The definition of c-subsolutions is equivalent to saying that $u\circ \gamma$ is a subsolution of $|u'|=f$ in $\R$ for any $\gamma\in \A(\R, \Omega)$. It follows that $u'\leq f\circ\gamma $ holds in $\R$ in the viscosity sense. By the result in Problem \ref{prob:monotone0}, we obtain \eqref{prop c-sub eq}.
\end{proof}

%\begin{remark}
    %It is worth remarking that the proof above shows a slightly stronger property of c-supersolution than \eqref{c-super char}. If $u$ is a c-supersolution, then for any $\vep>0$, there exists $\gamma\in C_x$ such that \eqref{c-super char2} holds for all $t\in (0, T)$.
    %\end{remark}

Uniqueness of c-solutions of \eqref{eikonal eq} with boundary data \eqref{bdry cond} is shown 
by proving the following comparison principle, which is taken from \cite[Theorem 3.1]{GHN}. 
\begin{theorem}{Comparison principle for c-solutions}{comparison c-sol}
Assume that $f\in C(\Omega)$ satisfies $\inf_{\Omega} f>0$. %\eqref{f-lower}.
Let $u\in \usc(\Oba)$ and $v\in \lsc(\Oba)$ be respectively a c-subsolution and a c-superslution of \eqref{eikonal eq}. If $u\leq v$ on $\partial \Omega$, then $u\leq v$ in $\Oba$.  
\end{theorem}
\begin{proof}
Assume by contradiction that $\max_{\Oba} u-v>0$. Then by the semicontinuity of $u$ and $v$, for $0<c<1$ close to $1$, we have 
\[
\beta:=\max_{\Oba} (cu-v)>0,
\]
and, by the condition $u\leq v$ on $\partial \Omega$, 
\[
cu-v\leq c(u-v)+(c-1)v\leq \beta/2 \quad \text{on $\partial \Omega$.}
\]
The above consequence is the same as \eqref{cp boundary eq1} in the Euclidean proof. We thus have $x_0\in \Omega$ such that $cu(x_0)-v(x_0)> \beta/2$.
%Using the characterization of c-supersolutions in Proposition \ref{prop:c-super char}, 
Let us fix $0<\vep<(1-c)\nu$, where $\nu=\inf_{\Omega} f>0$. % is the lower bound of $f$ as in \eqref{f-lower}. 

By definition of c-supersolutions, we can find an $\vep$-pair $(\gamma_\vep, w)$ satisfying $\gamma_\vep\in C_{x_0}$ with $T:=T_\Omega^+[\gamma_\vep]<\infty$ and $|w'|\geq f\circ \gamma_\vep-\vep$ in the viscosity sense. 

We extend $\gamma_\vep$ to a curve $\tilde{\gamma}_\vep$ in $\A_x(\R, \X)$ by taking $\tilde{\gamma}_\vep(t)=\gamma_\vep(-t)$ for all $t\in (-T, 0)$. By definition of c-subsolutions, we see that $\eta=u\circ \tilde{\gamma}_\vep$ is a viscosity subsolution of $|w'|\leq f\circ \tilde{\gamma}_\vep$ in $(-T, T)$. Thus $c\eta=cu\circ\tilde{\gamma}_\vep$ satisfies $|w'|\leq cf\circ \tilde{\gamma}_\vep$ in the viscosity sense. Since our choice of $\vep\in (0, (1-c)\nu)$ implies $f-\vep>cf$, by the Euclidean comparison argument as pointed out in Remark \ref{rmk:doubling-eikonal}, we have 
\[
\max_{[-T, T]}(c\eta-w)\leq \max_{\{\pm T\}} (c\eta-w)\leq \beta/2.
\]
However, noticing that $x_0\in \tilde{\gamma}_\vep$ and $cu(x_0)-v(x_0)>\beta/2$, we derive a contradiction. 
\end{proof}

The existence of solutions in $C(\Oba)$, 
on the other hand, is based on the following optimal control formula: 
\beq\label{eq optimal control}
u(x)=\inf_{\gamma\in C_x} \left\{ \int_{0}^{T_\Omega^+[\gamma]} f(\gamma(s))\, ds+g\left(\gamma(T_\Omega^+[\gamma])\right)\right\},
\eeq
where we recall that 
%\begin{equation}\label{c-existence eq1}
\[
C_x=\left\{\gamma\in \mathcal{A}_x([0, \infty), \Oba) \, :\, T_\Omega^+[\gamma]\in [0, \infty)\right\},\quad {x\in \Oba}.
\]
%\end{equation}
It can be shown, by combining Proposition \ref{prop:c-super char}(1) and Proposition \ref{prop:c-sub},  that $u$ is a c-solution of \eqref{eikonal eq} and \eqref{bdry cond} provided that there are plenty of curves in $\Oba$ and $g$ satisfies a compatibility condition in the form of 
\begin{equation}\label{bdry compatibility}
    g(x)\leq g(\gamma(T))+ \int_0^{T} f(\gamma(s))\, ds %+g(\gamma(T_\Omega^+[\gamma]))\right\}\quad
    \quad \text{for all $\gamma\in \A_x([0, \infty), \Oba)$ with $T<\infty$, $\gamma(T)\in \partial \Omega$}.
\end{equation} 
This condition amounts to saying that $g\leq u$ on $\partial \Omega$ for $u$ defined in \eqref{eq optimal control}. The original existence result with relaxed regularity (arcwise continuity) of c-solutions is due to \cite[Theorem~4.2 and Theorem~4.5]{GHN}. 

Concerning the abundance condition of curves in $\Oba$, we take the \emph{intrinsic metric} 
\beq\label{int metric}
\tilde{d}(x, y)=\inf\{\ell(\gamma): \gamma\text{ is a rectifiable curve connecting }x\text{ and }y\}, \ x, y\in \X,
\eeq
and assume that 
\begin{equation}\label{eq ast} 
\tilde{d}\to 0\quad \text{as} \quad d\to 0.
\end{equation}

%\begin{equation}\label{continuity length}
%\tilde{d}    \inf\left\{\ell(\gamma):\ \gamma\ \text{is a rectifiable curve in $\Oba$ joining $x, y\in\Oba$} \right\}\to 0 \quad \text{as $d(x, y)\to 0$}.
%\end{equation}
This assumption is somewhat natural for us to build a connection between the metric $d$ and curves in the space. It has further applications in Section \ref{sec:induced length}. 

It is clear that \eqref{eq ast} also implies $C_x\neq \emptyset$,
%and
%\begin{equation}\label{c-existence eq2}
%\inf_{\gamma\in C_x}\ell(\gamma)\to 0\quad \text{as $d(x, \partial \Omega)\to 0$,}
%\end{equation}
which ensures that $u$ in \eqref{eq optimal control} is well-defined.

\begin{theorem}{Existence of c-solutions}{c-existence}
Let $(\X, d)$ be a complete metric space and $\Omega\subsetneq \X$ be a bounded domain. Assume that \eqref{eq ast} holds. Assume that $f\in C(\Oba)$ is bounded and satisfies $f\geq 0$ in $\Omega$.  Then $u$ defined by \eqref{eq optimal control} is uniformly continuous in $\Omega$ and is a c-solution of \eqref{eikonal eq}. 
Assume in addition that $g\in C(\partial \Omega)$ satisfies \eqref{bdry compatibility} and $f$ satisfies $\inf_\Omega f>0$, then $u$ is the unique c-solution of \eqref{eikonal eq} \eqref{bdry cond}.  
\end{theorem}
\begin{proof}
  \ul{Step 1}. Let us show $u$ is uniformly continuous in $\Omega$. 
   %By \eqref{eq ast}, there exists a modulus of continuity $\omega_d$ such that for any $x, y\in \Oba$ with $x\neq y$,  
   By definition of \eqref{int metric}, we can find an arc-length parametrized curve $\gamma_{xy}$ joining $x, y$ satisfying $\ell(\gamma_{xy})\leq 2\tilde{d}(x, y)$. If $x\in \Omega$, then in view of \eqref{eq optimal control}, we can take $\gamma\in C_x$ such that 
    \[
    u(x)\geq \int_{0}^{T_\Omega^+[\gamma]} f(\gamma(s))\, ds+g\left(\gamma(T_\Omega^+[\gamma])\right)-\tilde{d}(x, y).
    \]
    For $y\in \Omega$ near $x$, concatenating $\gamma_{xy}$ and $\gamma$, we have a curve in $C_y$ and thus 
    \begin{equation}\label{c-existence eq3}
    \begin{aligned}
    u(y)&\leq \int_{\gamma_{xy}} f\, ds+\int_{0}^{T_\Omega^+[\gamma]} f(\gamma(s))\, ds+g\left(\gamma(T_\Omega^+[\gamma])\right) \\
    &\leq u(x)+\tilde{d}(x, y)+ \ell(\gamma_{xy}) \sup_{\Oba} f\leq u(x)+ \left(2\sup_{\Oba} f+1\right)\tilde{d}(x, y).
    \end{aligned}
    \end{equation}
A symmetric argument yields that 
\[
u(x)-u(y)\leq \left(2\sup_{\Oba} f+1\right)\tilde{d}(x, y)\quad \text{for all $x, y\in \Omega$.}
\]
Since \eqref{eq ast} holds, $u$ is uniformly continuous with respect to $d$ in $\Omega$.

\ul{Step 2}. The verification of c-supersolution and c-subsolution properties is based on Proposition \ref{prop:c-super char} and Proposition \ref{prop:c-sub}. We omit the details here. 

\ul{Step 3}. Let us show that $u\in C(\Oba)$ under the condition that $g\in C(\partial\Omega)$ satisfies \eqref{bdry compatibility}. Note that $u=g$ holds on $\partial \Omega$ by definition. For any $x\in \Omega$, there exist $y\in \partial \Omega$ and a curve $\gamma$ in $\Oba$ joining $x, y$ with 
\beq\label{c-existence eq6}
\tilde{d}(x, \partial \Omega)\leq \tilde{d}(x, y)\leq \ell(\gamma)\leq 2\tilde{d}(x, \partial \Omega), 
\eeq
where $\tilde{d}(x, \partial \Omega)=\inf_{y\in \partial \Omega} \tilde{d}(x, y)$.
Adopting  \eqref{eq optimal control} we easily obtain 
\begin{equation}\label{c-existence eq4}
u(x)\leq u(y)+\int_{\gamma} f\, ds\leq g(y)+2\tilde{d}(x, \partial \Omega)\sup_{\Oba} f. 
\end{equation}
Suppose that $x\to x_0\in \partial \Omega$. Since
\[
\tilde{d}(x_0, y)\leq \tilde{d}(x, x_0)+\tilde{d}(x, y)\leq 3\tilde{d}(x, x_0),
\]
by \eqref{c-existence eq4} we have 
\beq\label{c-existence eq7}
u(x)\leq g(x_0)+\sup_{\tilde{d}(x_0, y)\leq 3\tilde{d}(x, x_0)}|g(y)-g(x_0)|+2\tilde{d}(x, \partial \Omega)\sup_{\Oba} f.
\eeq
By the continuity of $g$ and \eqref{eq ast}, it follows that 
\[
\limsup_{x\to x_0} u(x)\leq g(x_0)=u(x_0).
\]
On the other hand, for every $\vep>0$, there exists $\gamma_1\in C_x$ with $T_1^+:=T_\Omega^+[\gamma_1]$, $\gamma_1(T_1^+)\in \partial \Omega$ such that 
\begin{equation}\label{compatibility metric eq2}
u(x)\geq \int_0^{T_1^+} f(\gamma_1(s))\, ds+g(\gamma_1(T_1^+))-\vep.
\end{equation}
Let $\gamma_2$ be the curve joining $x\in \Omega$ and $y\in \partial \Omega$ satisfying \eqref{c-existence eq6}. 
%For any curve $\gamma_2\in C_x$ with $T_2^+:=T_\Omega^+[\gamma_2]$ and $y=\gamma_2(T_2^+)\in \partial \Omega$ with $T_2^+\leq 2\tilde{d}(x, \partial \Omega)$, 
We can concatenate $\gamma_2$ and $\gamma_1$ to build a curve $\tilde{\gamma}\in \A_{x}[0, \infty), \Oba)$ joining $y$ and $\gamma_1(T_1^+)$. It follows from \eqref{bdry compatibility} and \eqref{c-existence eq6} that 
\begin{equation}\label{compatibility metric eq3}
\begin{aligned}
g(y)\leq g(\gamma_1(T_1^+))+\int_{\tilde{\gamma}} f\, ds&\leq g(\gamma_1(T_1^+)) +\int_0^{T_1^+} f(\gamma_1(s))\, ds+\int_\gamma f\, ds\\
&\leq g(\gamma_1(T_1^+)) +\int_0^{T_1^+} f(\gamma_1(s))\, ds + 2\tilde{d}(x, \partial \Omega)\sup_{\Oba} f.
\end{aligned}
\end{equation}
Combining \eqref{compatibility metric eq2} and \eqref{compatibility metric eq3}, we are led to 
\[
g(y)\leq u(x)+2\tilde{d}(x, \partial \Omega)\sup_{\Oba} f+\vep.
\]
Letting $\vep\to 0$, we have
\beq\label{c-existence eq5}
g(y)\leq u(x)+2\tilde{d}(x, \partial \Omega)\sup_{\Oba} f.
\eeq
Sending $d(x, x_0)\to 0$, we apply the continuity of $g$ and \eqref{eq ast} again to deduce that 
$g(x_0)\leq \liminf_{x\to x_0} u(x)$. Therefore $u\in C(\Oba)$ and $u=g$ on $\partial \Omega$. 
The uniqueness of c-solutions follows from Theorem \ref{thm:comparison c-sol}.
%for any curve joining $x, y$. Taking infimum of the right hand side over all such curves yields $u(x)\leq u(y)+ \omega_d(d(x, y))$.      
%
%Note that \eqref{c-existence eq3} also yields $\liminf_{\Oba\ni x\to y} u(x)\geq u(y)$ for all $y\in \partial \Omega$. It then follows from \eqref{bdry compatibility} that $\liminf_{\Oba\ni x\to y} u(x)\geq g(y)$.  Since by definition $u\leq g$ on $\partial \Omega$, we are led to $\limsup_{\Oba\ni x\to y} u(x)\geq g(y)$.  
\end{proof}

Note in the definition of c-supersolution, for each $(x,\vep)$ the conditions \eqref{wapprox} 
and \eqref{sgapprox} for $(\gamma, w)$ are satisfied for all $t\in (T^-, T^+)$. We localize this definition as follows. 
Recall the notions of $T^{\pm}_\Omega[\gamma]$ for open sets $\Omega\subset\X$ from~\eqref{exit/entrance}.

\begin{definition}{Local curve-based solutions}{local c}
A function $u\in \lsc(\Omega)$ is called a \emph{local c-supersolution} if for each $x\in\Omega$ there exists $r>0$ with $B_r(x)\subset\Omega$, and
for each $\vep>0$ we can find a curve $\gamma_\vep\in \A_x(\R, \X)$ with $\gamma_\vep(0)=x$ and a function 
$w\in \lsc(t_r^-, t_r^+)$ with $t^-_r:=T^-_{B_r^d(x)}[\gamma_\vep]$ and $t^+_r:=T^+_{B_r(x)}[\gamma_\vep]$ such that
\[
w(0)=u(x), \quad w(t)\geq u\circ\gamma_\vep(t)-\vep \ \text{ for all } t\in (t_r^-,t_r^+),
\]
and 
\[
\left|\phi'(t_0)\right|\geq f(\gamma_\vep(t_0))-\vep 
\]
whenever  $\phi\in C^1(\R)$ such that $w(t)-\phi(t)$ attains a minimum at $t=t_0\in (t_r^-, t_r^+)$. 

We call a function $u\in C(\Omega)$ a \emph{local c-solution} if it is both a c-subsolution and a 
local c-supersolution of \eqref{eikonal eq}. 
\end{definition}

%In the definition of local c-supersolutions given above,
%the ball $B_r^d(x)$ is taken with respect to $d$. If $(\X, d)$ is a complete rectifiably connected 
%metric space such that the intrinsic metric $\tilde{d}$ defined in \eqref{int metric} satisfies the consistency 
%condition \eqref{eq ast}, 
%then it is equivalent to use metric balls $B_r(x)$
%with respect to $\tilde{d}$. 

%This definition is studied mainly in Section~4,
%where we assume $(\X,d)$ to be a length space. For length spaces balls with respect to $d$ and balls with respect to $\tilde{d}$
%are the same, that is, $B_r^d(x)=B_r(x)$.

A c-supersolution (resp., c-solution) is clearly a local c-supersolution (resp., local c-solution), but it is not clear to us whether 
the reverse is also true in general. The notion of c-subsolutions is already a localized one, and we therefore do not have to define
``local c-subsolutions'' separately.

\begin{remark}\label{local c-super prop}  
If $u$ is a local c-supersolution instead, then for each $x\in \Omega$ there is a sufficiently small 
$r>0$ such that for each $\vep>0$ we can find a choice $\gamma_\vep\in\A_{x}([0,\infty), B_r(x))$
such that for all $0\le t\le T^+_{B_r(x)}[\gamma_\vep]$, \eqref{c-super char2} holds. 
This is seen by directly adapting the proof of Proposition \ref{prop:c-super char} (\cite[Proposition~2.8]{GHN}).
\end{remark}

\subsection{Slope-based solutions}
We next discuss the definition proposed in \cite{GaS2}, which relies more on the property of 
geodesic or length metric. It is clear that if $(\X, d)$ is a length space, then $\tilde{d}=d$ in  $\X\times \X$ holds for $\tilde{d}$ defined in \eqref{int metric}.

Let us introduce several function classes first. For an open subset $\Omega$ of a complete length space $(\X, d)$,  we say a function $u:\Omega\to \R$ is locally Lipschitz if for any $x\in \Omega$, there exists $r>0$ such that $u$ is Lipschitz in $B_r(x)$. We denote by $\text{Lip}_{loc}(\Omega)$ the set of locally Lipschitz continuous functions in $\Omega$. 
For $u\in \text{Lip}_{loc}(\Omega)$ and for $x\in\Omega$, we define the \emph{local slope} of $u$ to be
\beq\label{slope}
|\nabla u|(x):=\limsup_{y\to x}\frac{|u(y)-u(x)|}{d(x,y)}.
\eeq
Let 
\beq\label{test class}
\begin{aligned}
\overline{\mathcal{C}}(\Omega) &:= \{ u \in \text{Lip}_{loc}(\Omega) \, :\,  \text{$|\nabla^+ u| 
= |\nabla u|$ and $|\nabla u|$ is continuous in $\Omega$} \}, \\
\underline{\mathcal{C}}(\Omega) &:= \{ u \in \text{Lip}_{loc}(\Omega) \, :\,  \text{$|\nabla^- u| 
= |\nabla u|$ and $|\nabla u|$ is continuous in $\Omega$} \}, \\
\end{aligned}
\eeq
where, for each $x\in \X$, 
\beq\label{semi slope}
|\nabla^\pm u|(x) := \limsup_{y \to x} \frac{[u(y)-u(x)]_\pm}{d(x, y)}
\eeq
with $[a]_+:=\max\{a, 0\}$ and $[a]_-:=-\min\{a, 0\}$ for any $a\in \R$. In this work we call 
$|\nabla^+  u|$ and $|\nabla^- u|$ the (local) \emph{super-slope} and \emph{sub-slope} of $u$ respectively; they 
are also named super- and sub-gradient norms in the literature (cf. \cite{LoVi}).

Concerning the test class in a general geodesic or length space $(\X, d)$, the following result is known. More details can be found in \cite[Lemma 7.2]{GaS2} and \cite[Lemma 2.3]{GaS}. 

\begin{lemma}{Distance-type test functions}{dist-test}
Let $\Omega$ be a domain in a length space $(\X, d)$. For $x_0\in \X$, the function $\varphi(x)= h(d(x, x_0))$ (resp., $\varphi(x)=- h(d(x, x_0))$) belongs to the class $\ul{\mathcal{C}}(\Omega)$ (resp., $\ol{\mathcal{C}}(\Omega)$) provided that $h\in C^1([0, \infty))$ 
satisfies $h'(0)=0$ and $h'\geq 0$ in $(0, \infty)$ (resp., $h'\leq 0$ in $(0, \infty)$). Moreover, \begin{equation}\label{dist-test eq2}
|\nabla \varphi|(x)=|\nabla^- \varphi|(x)=h'(d(x, x_0)).
\end{equation}
holds for all $x\in \Omega$.
\end{lemma}
\begin{proof}
We only prove the part with $\varphi(x)= h(d(x, x_0))$. The other function can be treated in a symmetric manner. For any fixed $x\in \Omega$, we have 
\[
\begin{aligned}
\varphi(x)-\varphi(y)&=h(d(x, x_0))-h(d(y, x_0))\\
&=h'(d(x, x_0))(d(x, x_0)-d(y, x_0))+o(|d(x, x_0)-d(y, x_0)|)
\end{aligned}
\]
for all $y\in \Omega$ near $x$. By the triangle inequality, we have $d(x, x_0)-d(y, x_0)\leq d(x, y)$. Since $h'(d(x, x_0))\geq 0$ by our assumption, it follows that 
\begin{equation}\label{dist-test eq1}
|\nabla \varphi|(x)\leq h'(d(x, x_0)) \limsup_{y\to x}  \frac{d(x, x_0)-d(y, x_0)}{d(x, y)}\leq h'(d(x, x_0)).
\end{equation}
In the case that $x=x_0$, we have $|\nabla \varphi|(x_0)\leq 0$ and thus $|\nabla \varphi|(x_0)=|\nabla^-\varphi|(x_0)=0$. If $x\neq x_0$, then due to the length structure of the space, for any $\vep>0$, we can find a curve $\gamma\in \X$ joining $x, x_0$ such that $\ell(\gamma)\leq d(x, x_0)+\vep^2$. We take points $y_\vep$ on $\gamma$ such that $d(x, y_\vep)=\vep$ and thus 
\[
d(y_\vep, x_0)\leq \ell(\gamma\vert_{x_0y_\vep})=\ell(\gamma)-\ell(\gamma\vert_{xy_\vep})\leq \ell(\gamma)-\vep,
%d(x, y_\vep)\leq \ell(\gamma\vert_{xy_\vep})=\ell(\gamma)-\ell(\gamma\vert_{x_0y_\vep})\leq (1+\vep)d(x, x_0)-d(x_0, y), 
\]
where $\gamma\vert_{xy_\vep}$ and $\gamma\vert_{x_0y_\vep}$ denote respectively the points of $\gamma$ between $x, y_\vep$ and between $x_0, y_\vep$. It follows that
\[
d(x, x_0)-d(y_\vep, x_0)\geq \ell(\gamma)-\vep^\vep- d(y_\vep, x_0)\geq \vep-\vep^2
\]
We thus have 
\[
\frac{\max\{\varphi(x)-\varphi(y_\vep), 0\}}{d(x, y_\vep)}\geq h'(d(x, x_0)) \frac{\vep-\vep^2}{\vep}
\]
Sending $\vep\to 0$ yields $|\nabla^- \varphi|(x)\geq h'(d(x, x_0))$. Combining it with \eqref{dist-test eq1}, we see that \eqref{dist-test eq2} holds for all $x\in \Omega$. Hence, $\varphi\in \ul{\C}(\Omega)$. 
\end{proof}

\begin{remark}\label{rmk:s-test}
    As a consequence of Lemma \ref{lem:dist-test}, we see that $\varphi(x)= kd(x, x_0)^2+C$ is in the test class $\ul{\C}(\Omega)$ for any $k\geq 0$ and $C\in \R$. In particular, by \eqref{dist-test eq2} we have 
    \[
    |\nabla \varphi|(x)=|\nabla^- \varphi|(x)=2kd(x, x_0) \quad\text{for all $\Omega$.}
    \]
One can show that, if $k<0$, $x\mapsto kd(x, x_0)^2+C$ is in $\ol{\C}$ but in general it is not in $\ul{\C}(\Omega)$. See the example below.
\end{remark}

\begin{example}{}{}
    Let $\X$ be a triangular region in $\R^2$ as follows:
\[
\X=\{(x_1, x_2)\in \R^2: x_1\geq 0, x_2\geq 0, x_1+x_2\leq 1\}.
\]
Under the Euclidean metric $|\cdot |$, it is not difficult to see that $\X$ is a geodesic space.
Let $x_0=(1, 0)$ and $\Omega=\X \setminus \{x_0\}$; see Figure \ref{fig:test-example}.

\begin{figure}[H]
    \centering
    \includegraphics[width=0.5\textwidth]{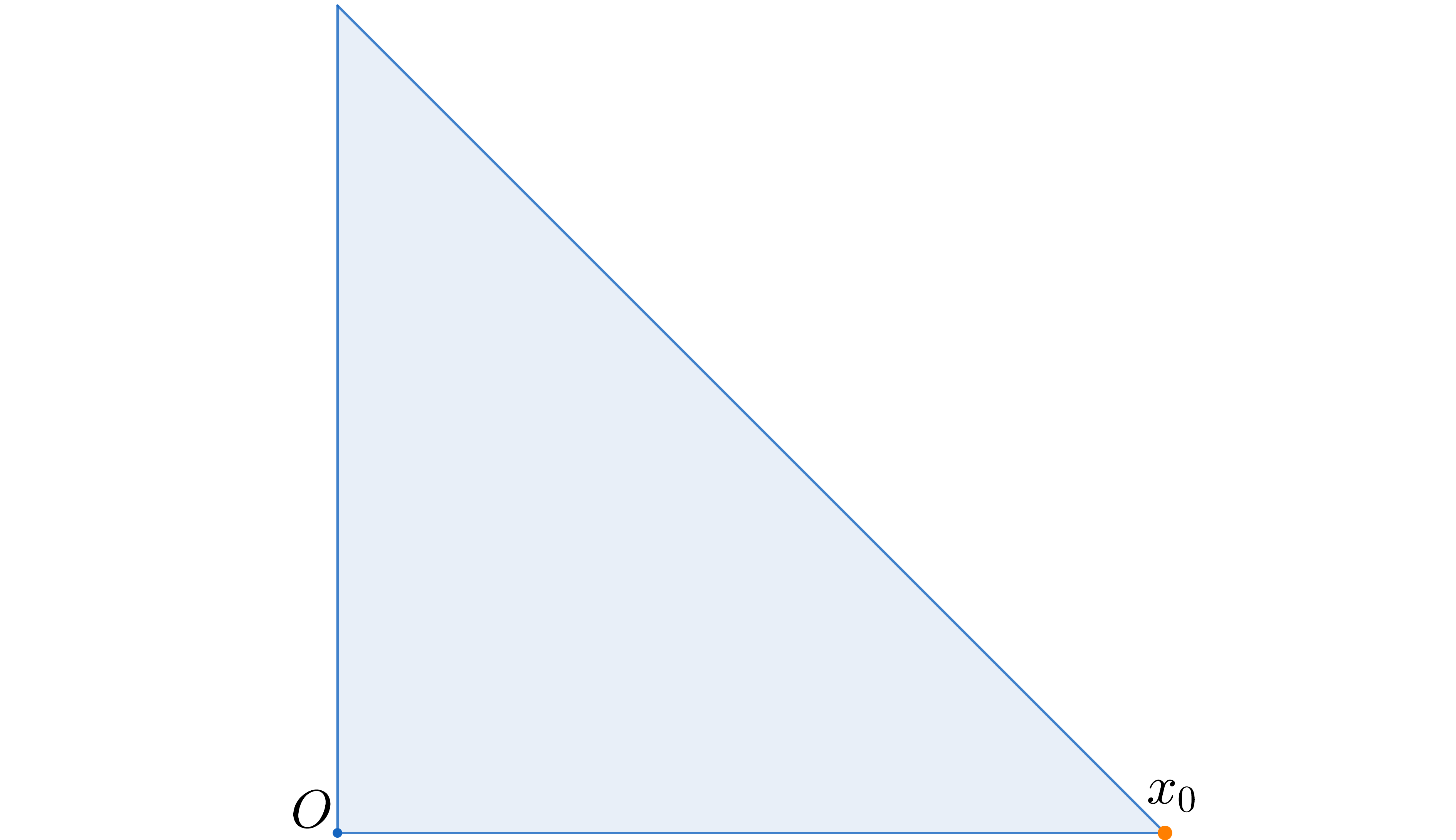}
    \caption{A domain $\Omega$ where $-d(\cdot, x_0)^2\notin \ul{\C}(\Omega)$}
    \label{fig:test-example}
\end{figure}

Then $\Omega$ is a domain in $\X$. We take $\varphi(x)=-|x-x_0|^2$ for $x\in \Omega$. Note that at the origin $O=(0, 0)\in \Omega$, 
\[
|\nabla^+ \varphi|(O)=\lim_{\vep\to 0+}{1\over \vep}(\varphi(\vep, 0)-\varphi(0, 0))=\lim_{\vep\to 0+}{1\over \vep}(1-(1-\vep)^2)=2
\]
while
\[
|\nabla^- \varphi|(O)=\lim_{\vep\to 0+}{1\over \vep}(\varphi(0, 0)-\varphi(0, \vep))=\lim_{\vep\to 0+}{1\over \vep}(1+\vep^2-1)=0.
\]
Hence, we get $\varphi\notin \ul{\C}(\Omega)$. 
\end{example}

Now we recall from \cite{GaS}  the definition of s-solutions of a general class of Hamilton-Jacobi equations. Following \cite{GaS}, we first extend the definition of $H(x, r, p)$ for $p<0$ by setting
\beq\label{h-extension}
H(x, r, p)=H(x, r, 0)\quad \text{for $(x, r, p)\in \Omega\times \R\times (-\infty, 0)$.}
\eeq

The following definition is from Definition 2.6 in \cite{GaS}. For any $(x, r, p)\in \Omega\times \R\times \R$ and $a\geq 0$, denote
\[
H_a(x, r, p) = \inf_{|q-p| \le a}H(x, r,  q), \quad H^a(x, r, p) = \sup_{|q-p| \le a}H(x, r, q). 
\] 

\begin{definition}{Slope-based solutions}{defi s}
A function $u\in \usc(\Omega)$ is called a
\emph{slope-based subsolution} or 
 \emph{s-subsolution} of \eqref{stationary eq} if
\beq\label{s-sub eq}
H_{|\nabla \psi_2|^*(x)}(x, u(x),  |\nabla \psi_1|(x)) \le 0
\end{equation}
holds for any $\psi_1 \in \underline{\mathcal{C}}(\Omega)$ (resp., $\psi_1 \in \overline{\mathcal{C}}(\Omega)$) 
and $\psi_2 \in \text{Lip}_{loc}(\Omega)$  such that $u-\psi_1-\psi_2$ attains a local maximum 
at a point $x \in \Omega$, where $|\nabla \psi_2|^*(x) = \limsup_{y \to x}|\nabla \psi_2|(y)$. 

A function $u\in \lsc(\Omega)$ in an open set $\Omega\subset \X$ is called a
\emph{slope-based supersolution} or \emph{s-supersolution} of \eqref{stationary eq} if
\beq\label{s-sup eq}
H^{|\nabla \psi_2|^*(x)}(x, u(x), |\nabla \psi_1|(x)) \ge 0
\eeq
holds for any $\psi_1 \in \overline{\mathcal{C}}(\Omega)$ 
and $\psi_2 \in \text{Lip}_{loc}(\Omega)$  such that $u-\psi_1-\psi_2$ attains a local minimum  
at $x \in \Omega$.
We say that $u\in C(\Omega)$ is an \emph{slope-based solution} or \emph{s-solution} of \eqref{stationary eq} if it is both an s-subsolution and an s-supersolution of \eqref{stationary eq}. 
\end{definition}

In the case of \eqref{eikonal eq}, we can define subsolutions (resp.,  supersolutions) by 
replacing \eqref{s-sub eq} (resp., \eqref{s-sup eq}) with 
\beq\label{s-sub eikonal}
|\nabla\psi_1|(x)\leq f(x)+|\nabla \psi_2|^\ast(x)
\eeq
\beq\label{s-super eikonal}
\left(\text{resp.,} \quad |\nabla\psi_1|(x)\geq f(x)-|\nabla \psi_2|^\ast(x) \right).
\eeq

When $\X=\R^N$, it is not difficult to see that $C^1(\Omega)\subset \ol{\mathcal{C}}(\Omega)\cap \ul{\mathcal{C}}(\Omega)$ 
for any open set $\Omega\subset \X$. Hence,  s-subsolutions, s-supersolutions and s-solutions of \eqref{stationary eq} 
in this case reduce to  conventional viscosity subsolutions, supersolution and solutions respectively. 

Comparison principles for s-solutions are given in \cite{GaS}. %\cite[Theorem 5.3]{GaS} for eikonal equations and 
%in for more general Hamilton-Jacobi equations. 
The following result is a simplified version of \cite[Theorem 5.3]{GaS} in the casse of \eqref{eikonal eq} in a bounded domain. 

\begin{theorem}{Comparison principle for s-solutions}{comparison s-sol}
Let $(\X, d)$ be a complete length space and $\Omega\subsetneq \X$ be a bounded domain.  Assume that $f\in C(\Omega)$ satisfies $\inf_{\Omega} f>0$. %\eqref{f-lower}. 
Assume that $f$ is uniformly continuous in $\Omega$; namely, there is a modulus of continuity $\omega_f$ such that 
\begin{equation}\label{f-uc}
|f(x)-f(y)|\leq \omega_f(d(x, y))\quad \text{for all $x, y\in \Omega$.}
\end{equation}
Let $u\in \usc(\Oba)$ and $v\in \lsc(\Oba)$ be respectively a bounded s-subsolution and a bounded s-supersolution of \eqref{eikonal eq}. If 
\begin{equation}\label{s-comparison bdry}
\sup \{u(x)-v(y): x, y\in \Oba,\ d(x, \partial \Omega)+d(y, \partial \Omega)+d(x,y)\leq \sigma\}\to 0\quad \text{as $\sigma\to 0$,}
\end{equation}
then $u\leq v$ in $\Oba$. 
\end{theorem}
\begin{proof}
Assume by contradiction that $\sup_{\Oba} (u-v)>0$. Thanks to the boundedness of $u$ and $v$ and the comparison on the boundary, we can find $0<c<1$ and $\beta>0$ such that $\sup_{\Oba}(cu-v)>\beta$ and 
\begin{equation}\label{s-comparison eq1}
\sup \{cu(x)-v(y): x, y\in \Oba,\ d(x, \partial \Omega)+d(y, \partial \Omega)+d(x,y)\leq \sigma\}<{\beta\over 2}
\end{equation}
for $\sigma>0$ sufficiently small. 

We apply the doubling variable technique by considering, for $\vep>0$,  
\[
\Phi_\vep(x,y)=cu(x)-v(y)-{1\over \vep}d(x, y)^2, \quad \text{$x, y\in \Oba$.}
\]
Since $\sup_{\Oba\times \Oba}\Phi_\vep> \beta$, one can let $\vep>0$ small and take $(\hat{x}, \hat{y})\in \Oba\times \Oba$ such that 
\[
\Phi_\vep(\hat{x}, \hat{y})>\sup_{\Oba\times \Oba}\Phi_\vep-\vep^2>\beta. 
\]
We thus have
\[
{1\over \vep}d(\hat{x}, \hat{y})^2\leq cu(\hat{x})-v(\hat{y})\leq M, 
\]
where $M$ denotes the bound of $|u|+|v|$ in $\Omega$. By taking $\vep>0$ small, we can get 
\beq\label{s-comparison eq4}
d(\hat{x}, \hat{y})\leq \sqrt{M\vep}<{\sigma\over 4}.
\eeq
This implies that 
\begin{equation}\label{s-comparison eq2}
d(\hat{x}, \partial \Omega)>{\sigma\over 4}, \quad d(\hat{y}, \partial \Omega)>{\sigma\over 4}.
\end{equation}
Indeed, if either of these fails to hold, then we have 
\[
d(\hat{x}, \partial \Omega)+d(\hat{y}, \partial \Omega)\leq 2\min\{d(\hat{x}, \partial \Omega), d(\hat{y}, \partial \Omega)\}+d(\hat{x}, \hat{y})\leq {3\over 4}\sigma,
\]
and therefore
\[
d(\hat{x}, \partial \Omega)+d(\hat{y}, \partial \Omega)+d(\hat{x}, \hat{y})<\sigma. 
\]
It then follows from \eqref{s-comparison eq1} that $\Phi_\vep(\hat{x}, \hat{y})<\beta/2$, which is a contradiction. We thus complete the proof of \eqref{s-comparison eq2}. 

Letting $\vep>0$ further small if necessary to guarantee $\vep<\sigma/4$,  we apply Ekeland's variational principle (Corollary \ref{cor:ekeland} with Remark \ref{rmk ekeland}) to get $x_\vep, y_\vep\in\Omega$ such that 
\beq\label{s-comparison eq3}
d(x_\vep, \hat{x})+d(y_\vep, \hat{y})\leq\vep
\eeq
and
\[
\Phi_\vep(x_\vep, y_\vep)\geq \Phi_\vep(x, y)-\vep d(x, x_\vep)-\vep d(y, y_\vep)
\]
for all $(x, y)$ near $(x_\vep, y_\vep)$. In other words, 
\[
(x, y)\mapsto cu(x)-v(y)-{1\over \vep}d(x, y)^2-\vep d(x, x_\vep)-\vep d(y, y_\vep)
\]
attains a local maximum at $(x_\vep, y_\vep)\in \Omega\times \Omega$. 

Let 
\[
\psi_1^u(x)={1\over c\vep} d(x, y_\vep)^2+{1\over c}v(y_\vep), \quad \psi_2^u(x)={\vep\over c}d(x, x_\vep),
\]
\[
\psi_1^v(y)=-{1\over \vep} d(x_\vep, y)^2+cu(x_\vep), \quad \psi_2^v(y)=-\vep d(y, y_\vep).
\]
Then, by Remark \ref{rmk:s-test} we see that $u-\psi_1^u-\psi_2^u$ attains a local maximum at $x_\vep$ with $\psi_1^u\in \underline{\mathcal{C}}(\Omega)$ and $\psi_2^u \in \lipl(\Omega)$ and $v-\psi_1^v-\psi_2^v$ attains a local minimum at $y_\vep$ with $\psi_1^v\in \overline{\mathcal{C}}(\Omega)$ and $\psi_2^v \in \lipl(\Omega)$.

By the definition of s-subsolutions and s-supersolutions of \eqref{eikonal eq}, we get
\[
{2\over c\vep} d(x_\vep, y_\vep)\leq f(x_\vep)+{\vep\over c}, \quad
{2\over \vep} d(x_\vep, y_\vep)\geq f(y_\vep)+\vep.
\]
It follows that $f(y_\vep)\leq cf(x_\vep)+2\vep$. By the uniform continuity assumption on $f$ as in \eqref{f-uc}, we thus have
\begin{equation}\label{s-comparison eq5}
(1-c)f(x_\vep)\leq f(x_\vep)-f(y_\vep)+2\vep\leq \omega_f(d(x_\vep, y_\vep))+2\vep.
\end{equation}
Since 
\[
d(x_\vep, y_\vep)\leq d(x_\vep, \hat{x})+d(y_\vep, \hat{y})+d(\hat{x}, \hat{y})\leq \vep+ \sqrt{M\vep}
\]
holds due to \eqref{s-comparison eq4} and \eqref{s-comparison eq3}, we can further let $\vep>0$ small to make the right hand side of \eqref{s-comparison eq5} arbitrarily small. We reach a contradiction by noticing that the left hand side of \eqref{s-comparison eq5} is bounded from below by $(1-c)\inf_{\Omega}f>0$. 
\end{proof}

Compared to the comparison principle in the Euclidean space, Theorem \ref{thm:comparison-intro1}, the statement and proof of Theorem \ref{thm:comparison s-sol} are more involved. The main issue lies at the possible loss of compactness of $\Oba$ in general metric spaces including infinitely dimensional spaces. This is why we assume the uniform continuity of $f$ in addition to its continuity. Since in the proof we are not able to take a subsequential limit of local maximizers $(x_\vep, y_\vep)$ again due to the possible lack of compactness of $\Oba$, we impose \eqref{s-comparison bdry} rather than simply assuming that $u\leq v$ on $\partial \Omega$. Moreover, we need to include $\psi_2\in \lipl(\Omega)$ in the test class to handle those extra terms terms generated by Ekeland's variational principle, which is only used for us to find local maximizers of a perturbed functional in an infinitely dimensional space.

It is possible to prove a comparison result similar to Theorem \ref{thm:comparison s-sol} for more general Hamilton-Jacobi equation including \eqref{stationary eq}. We omit the discussion here but refer to \cite[Theorem~5.1, Theorem~5.2]{GaS} for details. 

\subsection{Eikonal solutions in induced length space}\label{sec:induced length}

Let us compare the notion of c-supersolutions and s-supersolutions. Note that for c-supersolutions we do not need $(\X, d)$ to be a length space. However, if we convert it into a length space by taking the intrinsic metric by \eqref{int metric},  then one can see that the c-supersoluton property implies s-supersolution property. 

\begin{proposition}{Implication of supersolution property}{super}
Let  $(\X, d)$ be a complete rectifiably connected metric space and $\tilde{d}$ be the intrinsic 
metric given by \eqref{int metric}. Assume that \eqref{eq ast} holds. 
Let $\Omega\subsetneq \X$ be an open set. 
Assume that $f\in C(\Omega)$ with $\inf_{\Omega}f>0$. 
If $u$ is a lower semicontinuous c-supersolution of \eqref{eikonal eq} in $(\Omega, d)$, 
then $u$ is an s-supersolution of \eqref{eikonal eq} in $(\Omega, \tilde{d})$. 
\end{proposition}

\begin{proof}
Since $u$ is lower semicontinuous with respect to the metric $d$, it is easily seen that $u$ is also lower 
semicontinuous with respect to $\tilde{d}$, thanks to \eqref{int metric} and \eqref{eq ast}. 

Fix $x_0\in \Omega$ arbitrarily.  Assume that there exist $\psi_1 \in \overline{\mathcal{C}}(\Omega)$ and 
$\psi_2 \in \text{Lip}_{loc}(\Omega)$ such that $u-\psi_1-\psi_2$ attains a local minimum at a point $x_0$. We thus have 
$r_0>0$ such that $B_{2r_0}(x_0)\subset\Omega$ and
\[
u(y)-u(x_0)\geq (\psi_1+\psi_2)(y)-(\psi_1+\psi_2)(x_0)
\]
for all $y\in B_{r_0}(x_0)$.  

Applying Proposition \ref{prop:c-super char}, for any $\vep>0$ satisfying $\sqrt{\vep}<\min\{r,\tilde{d}(x_0, \pO)\}$, 
we can find $\gamma\in  C_{x_0}$ with $T_\Omega^+[\gamma]<\infty$ such that \eqref{c-super char2} 
holds for all $0\leq t\leq T_\Omega^+[\gamma]$. Since $T_\Omega^+[\gamma]\geq \tilde{d}(x_0, \pO)>0$, 
we can take $t=\sqrt{\vep}$ and $x_\vep=\gamma(\sqrt{\vep})$ in \eqref{c-super char2} to get
\[
(\psi_1+\psi_2)(x_\vep)-(\psi_1+\psi_2)(x_0)\leq -\int_0^{\sqrt{\vep}}f(\gamma(s))\, ds+\vep(1+\sqrt{\vep}).
\]
Dividing this relation by $\sqrt{\vep}$, we get
\beq\label{eq lem super1}
{1\over \sqrt{\vep}}\int_0^{\sqrt{\vep}} f(\gamma(s))\, ds+{\psi_2(x_\vep)-\psi_2(x_0)\over \sqrt{\vep}}-\sqrt{\vep}(1+\sqrt{\vep})\le 
{\psi_1(x_0)-\psi_1(x_\vep)\over \sqrt{\vep}}.
\eeq
Noticing that 
\[
\tilde{d}\left(x_\vep, x_0\right)\leq \sqrt{\vep},
\]
we deduce that 
\[
{\psi_1(x_\vep)-\psi_1(x_0)\over \sqrt{\vep}}\le {|\psi_1(x_\vep)-\psi_1(x_0)|\over \sqrt{\vep}}
\le {|\psi_1(x_\vep)-\psi_1(x_0)|\over \tilde{d}\left(x_\vep, x_0\right)}
\]
and
\[
{\psi_2(x_\vep)-\psi_2(x_0)\over \sqrt{\vep}}\ge -{|\psi_2(x_\vep)-\psi_2(x_0)|\over \sqrt{\vep}}
\ge -{|\psi_2(x_\vep)-\psi_2(x_0)|\over \tilde{d}\left(x_\vep, x_0\right)}.
\]
Hence, combining the above inequalities together implies that
\[
{|\psi_1(x_\vep)-\psi_1(x_0)|\over \tilde{d}\left(x_\vep, x_0\right)}
\geq {1\over \sqrt{\vep}}\int_0^{\sqrt{\vep}} f(\gamma(s))\, ds- {|\psi_2(x_\vep)-\psi_2(x_0)|\over \tilde{d}\left(x_\vep, x_0\right)}-
\sqrt{\vep}(1+\sqrt{\vep}).
\]
Letting $\vep\to 0$,  we are led to \eqref{s-super eikonal} with $x=x_0$ as desired.
\end{proof}

\begin{remark}\label{local c to s} 
By Remark \ref{local c-super prop}, it is not difficult to see that if $u$ is only a local c-supersolution, then for any 
$x_0\in \Omega$ the same result as in Proposition \ref{prop:super} holds in $B_r(x_0)$ with $r>0$ small. In fact, 
the proof will be the same except that $\Omega$ should be replaced by $B_r(x_0)$. 
\end{remark}

\begin{proposition}{Implication of subsolution property}{sub}
Let  $(\X, d)$ be a complete rectifiably connected metric space and $\tilde{d}$ be the intrinsic metric 
given by \eqref{int metric}.  Assume that \eqref{eq ast} holds. 
Let $\Omega\subsetneq\X$ be an open set. Assume that $f\in C(\Omega)$ and  $f\geq 0$  in $\Omega$. 
If $u$ is upper semicontinuous in $\Omega$ and is a c-subsolution of \eqref{eikonal eq} in $(\Omega, d)$, then $u$ is an 
s-subsolution of \eqref{eikonal eq} in $(\Omega, \tilde{d})$. 
\end{proposition}

\begin{proof}
Since $(\X, \tilde{d})$ is a length space, our notation $\text{Lip}_{loc}(\Omega)$ now denotes the set of all 
locally Lipschitz functions on $\Omega$ with respect to the
intrinsic metric $\tilde{d}$. Note that if $u$ is upper semicontinuous with respect to the metric $d$, 
then it is upper semicontinous with respect to $\tilde{d}$, since $\tilde{d}\to 0$ if and only if $d\to 0$ 
due to \eqref{int metric} and \eqref{eq ast}.

Fix $x_0\in \Omega$ arbitrarily.  Assume that there exists $\psi_1 \in \underline{\mathcal{C}}(\Omega)$ and 
$\psi_2 \in \text{Lip}_{loc}(\Omega)$ such that $u-\psi_1-\psi_2$ attains a local maximum at a point $x_0$. 
So there is some $r_0>0$ with $B_{2r_0}(x_0)\subset\Omega$ such that 
\[
u(x)-u(x_0)\leq (\psi_1+\psi_2)(x)-(\psi_1+\psi_2)(x_0)
\]
for all $x\in B_{r_0}(x_0)$. 

Moreover, for any fixed $\vep\in (0, 1)$, by the continuity of $f$, we can make $0<r_0$ smaller so that  
\begin{equation}\label{general curve2}
|f(x)-f(x_0)|\leq \vep
\end{equation}
if $x\in B_{r_0}(x_0)$. We fix such $r_0>0$ (and keep in mind that $r_0$ now also depends on $\vep$). 

For any $r\in (0, r_0/2)$ and any $x\in \Omega$ with $0<\tilde{d}(x, x_0)<r$, there exists an arc-length parametrized curve $\gamma$ in $\Omega$ such that $\gamma(0)=x_0$ and $\gamma(t)=x$, where 
\begin{equation}\label{general curve0}
t=\ell(\gamma) \leq \tilde{d}(x, x_0)+\vep\tilde{d}(x, x_0). 
\end{equation}
Applying Proposition \ref{prop:c-sub} to such a curve with $t_1=0, t_2=t$, we get 
\[
u(x_0)\leq \int_0^{t} f(\gamma(s))\, ds+u(x),
\]
and therefore,
\[
(\psi_1+\psi_2)(x_0)-(\psi_1+\psi_2)(x)\leq \int_0^{{t}}f(\gamma(s))\, ds.
\]
Dividing the inequality above by $\tilde{d}(x,x_0)$, we get
\begin{equation}\label{general curve}
{\psi_1(x_0)-\psi_1(x)\over \tilde{d}(x, x_0)}
\leq \frac{1}{\tilde{d}(x, x_0)}\int_0^{{t}}f(\gamma(s))\, ds+{\psi_2(x)-\psi_2(x_0)\over \tilde{d}(x, x_0)}.
\end{equation}
Since $\vep<1$ and $r<r_0/2$, we have
\[
\tilde{d}(\gamma(s), x_0)\leq t\leq r+\vep r<r_0
\]
for all $s\in [0, t]$, by \eqref{general curve2}. Therefore 
\[
f(\gamma(s))\leq f(x_0)+\vep
\]
for all $s\in [0, t]$. Hence \eqref{general curve} yields 
\[
{\psi_1(x_0)-\psi_1(x)\over \tilde{d}(x, x_0)}\leq \frac{t}{\tilde{d}(x, x_0)}(f(x_0)+\vep)+{\psi_2(x)-\psi_2(x_0)\over \tilde{d}(x, x_0)}.
\]
Using \eqref{general curve0} and recalling that the choice of $r_0$ depends on $\vep$, we thus have
\begin{equation}\label{general curve1}
{\psi_1(x_0)-\psi_1(x)\over \tilde{d}(x, x_0)}\leq (1+\vep)(f(x_0)+\vep)+{\psi_2(x)-\psi_2(x_0)\over \tilde{d}(x, x_0)}
\end{equation}
for all $\vep>0$ and all $x\in \Omega$ with $x\in B_{r_0}(x_0)$. 

Since $\psi_1\in \underline{\mathcal{C}}(\Omega)$, 
there exists a sequence of points $x_n\in \Omega$ such that, as $n\to \infty$, we have $x_n\to x_0$ and 
\[
{\psi_1(x_0)-\psi_1(x_n)\over \tilde{d}(x_n, x_0)}\to |\nabla^-\psi_1|(x_0)=|\nabla \psi_1|(x_0).
\]
Adopting \eqref{general curve1} with $x=x_n$ and sending $n\to \infty$ and then $\vep\to 0$, we end up with
the desired inequality \eqref{s-sub eikonal} at $x=x_0$.
\end{proof}

\begin{proposition}{Boundary consistency}{bdry regularity}
Assume that $(\X, d)$ is a complete rectifiably connected metric space and $\tilde{d}$ be the 
induced intrinsic metric given by \eqref{int metric}. Assume that \eqref{eq ast} holds. Let $\Omega\subsetneq \X$ be an open set. 
Suppose that $f\in C(\Oba)$ is bounded and $f\geq 0$ in $\Oba$. Let $u$ be given by \eqref{eq optimal control} with $g\in C(\pO)$ given. 
\begin{enumerate}
\item[(1)] 
If there exists $L>0$ such that  $g$ is $L$-Lipschitz on $\partial\Omega$
with respect to $\tilde{d}$, then 
\beq\label{sol bdry regularity weak}
u(x)- g(y)\leq \tilde{d}(x, y)\max\left\{L,\ \sup_{\Oba} f\right\}\quad \text{for all $x\in \Omega$ and $y\in \pO$.}
\eeq
%\item[(2)] If $g$ satisfies 
%\beq\label{bdry regularity}
%|g(x)-g(y)|\leq  \tilde{d}(x, y)  \inf_{\Oba} f \quad \text{for every $x, y\in \pO$,}
%\eeq
%then 
%\beq\label{sol bdry regularity}
%|u(x)- g(y)|\leq \tilde{d}(x, y)\sup_{\Oba} f \quad \text{for all $x\in \Omega$ and $y\in \pO$.}
%\eeq
\item[(2)] If $g$ is uniformly continuous on $\partial \Omega$ with modulus $\omega_g$ and \eqref{bdry compatibility} holds, then 
\beq\label{sol bdry regularity2}
|u(x)- g(y)|\leq 2\tilde{d}(x, y)\sup_{\Oba} f+\omega_g(2\tilde{d}(x, y)) \quad \text{for all $x\in \Omega$ and $y\in \pO$.}
\eeq
\end{enumerate}
\end{proposition}

\begin{proof}
For simplicity of notation, denote
\[
m:=\inf_{\Oba} f, \quad M:=\sup_{\Oba} f. 
\]
Fix $x\in \Omega$ and $y\in \pO$. Then for any $\vep>0$, there exists an arc-length parametrized curve 
$\gamma\in \A_{x}(\R, \X)$ such that $\gamma(t)=y$ and
\[
\tilde{d}(x, y)\leq t\leq \tilde{d}(x, y)+\vep.
\]
This curve may not stay in $\Omega$, but there exists $z=\gamma(t_1)\in \pO$, where  
\[
t_1:=T^+_\Omega[\gamma]=\inf\{s: \gamma(s)\in \pO\}. 
\]
Since we have
\[
t-t_1\geq \tilde{d}(y, z) \ \text{ and }\  t\leq \tilde{d}(x, z)+\tilde{d}(y, z)+\vep, 
\]
it follows that $t_1\leq \tilde{d}(x, z)+\vep$ and therefore
\beq\label{bdry regu1}
\tilde{d}(x, z)+\tilde{d}(z,y)\leq t \leq \tilde{d}(x, y)+\vep. 
\eeq
Now in view of \eqref{eq optimal control}, we have
\beq\label{bdry regu rev1}
u(x)\leq g(z)+\int_0^{t_1} f(\gamma(s))\, ds\leq g(z)+M(\tilde{d}(x, z)+\vep).
\eeq
Thanks to the $L$-Lipschitz continuity of $g$, we have  
\beq\label{use later1}
g(z)\leq g(y)+L\tilde{d}(y, z). 
\eeq

Applying \eqref{bdry regu1} and \eqref{use later1} in \eqref{bdry regu rev1}, we thus get
\beq\label{use later2}
u(x)\leq g(y)+L \tilde{d}(y, z)+M\tilde{d}(x, z)+M\vep\leq g(y)+\max\{L, M\}\tilde{d}(x, y)+2\max\{L, M\}\vep.
\eeq
Since the above holds for all $\vep>0$,  we have \eqref{sol bdry regularity weak} immediately. 

\begin{comment}
In order to show \eqref{sol bdry regularity}, we need the stronger condition \eqref{bdry regularity}, which 
means that $g$ is $m$-Lipschitz on $\partial\Omega$ with respect to $\tilde{d}$.
By \eqref{eq optimal control}, for any $\vep>0$, there exists $y_\vep\in \pO$ and 
a curve $\gamma_\vep\in \A_x(\R, \Oba)$ such that with $t_\vep>0$ chosen so that $\gamma_\vep(t_\vep)=y_\vep$, we have
\[
u(x)\geq g(y_\vep)+\int_0^{t_\vep} f(\gamma(s))\, ds-\vep\geq g(y_\vep)+m \tilde{d}(x, y_\vep)-\vep.
\]
Using \eqref{bdry regularity},  we have 
\[
u(x)\geq g(y)-m\tilde{d}(y, y_\vep)+m\tilde{d}(x, y_\vep)-\vep\geq g(y)-m\tilde{d}(x, y)-\vep. 
\]
Letting $\vep\to 0$, we obtain 
\[
u(x)-g(y)\geq -m\tilde{d}(x, y)\geq -M\tilde{d}(x, y),
\]
which, combined with \eqref{sol bdry regularity weak} with $L=m$, completes the proof. 
\end{comment}

We can use \eqref{bdry compatibility} to show \eqref{sol bdry regularity2}. The proof is contained in that of Theorem \ref{thm:c-existence}, Step 3. In fact, for any $x\in \Omega$ close to $x_0\in \partial \Omega$, we have \eqref{c-existence eq7} and \eqref{c-existence eq5}. Combining these two inequalities together with the uniform continuity of $g$, we immediately obtain 
\[
|u(x)- g(x_0)|\leq 2\tilde{d}(x, x_0)\sup_{\Oba} f+\omega_g(2\tilde{d}(x, x_0)) \quad \text{for all $x\in \Omega$ and $x_0\in \pO$,}
\]
which is equivalent to \eqref{sol bdry regularity2}.
\end{proof}

The statement (2) complements the up-to-boundary continuity result in \cite[Proposition 3.9]{LShZ}. We here obtain the uniform continuity near the boundary under the more reasonable assumption \eqref{bdry compatibility}.

\subsection{Exercises}

\begin{problem}\label{prob:monotone0}
Let $I\subset \R$ be an open interval and $f\in C(I)$. Show that $u\in \usc(I)$ is a viscosity subsolution of $u'=f$ in $I$ if and only if for any $a\in I$ fixed,  $x\mapsto u(x)-\int_a^x f(y)\, dy$ is nonincreasing in $I$.
\end{problem}

\begin{problem}\label{prob:monotone}
Let $I\subset \R$ be an open interval and $f\in C(I)$ satisfies $f>0$ on $I$. Show that if a nonincreasing function $u\in \lsc(I)$ is a viscosity supersolution of $|u'|=f$ in $I$, then for any $a\in I$ fixed,  $x\mapsto u(x)+\int_a^x f(y)\, dy$ is nonincreasing in $I$. 
\end{problem}

\begin{problem}\label{prob:semiconcave}
Let $I\subset \R$ be an open interval. Let $u\in C(I)$ be a semiconcave function, that is, $u(x)-Kx^2$ is concave in $I$ for some $K\geq 0$. Show that $|\nabla^- u|=|\nabla u|$ holds in $\Omega$.
\end{problem}

%\begin{problem}\label{prob:s-test}
%Find an example to show that in a domain of a length space $(\X, d)$, $\varphi(x)=kd(x, x_0)^2$ with $k<0$ and $x_0\in \X$ may not belong to $\ul{\C}(\Omega)$. 
%\end{problem}
%\noindent Solution: 

\clearpage

\section{Monge solutions}

\subsection{Definition}
We assume that $(\X, d)$ is a complete length space. 

\begin{definition}{Monge solutions}{monge}
A function $u\in \lipl(\Omega)$ is called a \emph{Monge subsolution} (resp., {Monge supersolution}) 
of \eqref{stationary eq} if, at any $x\in \Omega$, 
\[
H\left(x, u(x), |\nabla^- u|(x)\right)\leq 0 \quad \left(\text{resp.}, H\left(x, u(x), |\nabla^- u|(x)\right)\geq 0\right).
\]
A function $u\in \lipl(\Omega)$ is said to be a \emph{Monge solution} if $u$ is both a Monge 
subsolution and a Monge supersolution, i.e., $u$ satisfies
\beq\label{stationary monge}
H\left(x, u(x), |\nabla^- u|(x)\right)=0
\eeq
at any $x\in \Omega$.  
\end{definition}

In the case of \eqref{eikonal eq}, the definition of Monge subsolutions (resp.,  supersolutions) reduces to 
\begin{equation}\label{eq:Monge-supsubsol-Eik}
|\nabla^- u|(x)\leq f(x)\quad (\text{resp., } |\nabla^- u|(x)\geq f(x))
\end{equation}
for all $x\in \Omega$.  Then $u$ is a Monge solution of \eqref{eikonal eq} if 
\begin{equation}\label{eq:Monge-sol-Eik}
|\nabla^- u|= f\quad \text{in $\Omega$}. 
\end{equation}

The notion of Monge solutions of \eqref{eikonal eq} in the Euclidean space is studied in \cite{NeSu}, where the 
right hand side $f$ is allowed to be more generally lower semicontinuous in $\overline{\Omega}$. 
Such a notion, still in the Euclidean space, was later generalized in \cite{BrDa} to handle general Hamilton-Jacobi 
equations with discontinuities. The definitions of Monge solutions in \cite{NeSu, BrDa} require the optical length 
function, which can be regarded as a Lagrangian structure for $H$.
In contrast, our definition of Monge solutions does not rely on optical length functions. See \cite{LShZ2} for further study on discontinuous eikonal equations adopting the notion of Monge solutions. This notion can also be used to study the first eigenvalue problem for infinity Laplacian in general metric spaces \cite{LMit1}. 

\subsection{Comparison principle}

One advantage of using the Monge solutions is that its uniqueness can be easily obtained. In what follows, 
we give a comparison principle for Monge solutions of the eikonal equation. 

\begin{theorem}{Comparison principle for Monge solutions}{comparison monge}
Let $(\X, d)$ be a complete length space and $\Omega\subsetneq \X$ be a bounded open set in $(\X, d)$.  Assume that 
$f\in C(\Omega)$ is bounded and satisfies $\inf_{\Omega}f>0$.  
Let $u \in C(\Oba)\cap \lipl(\Omega)$ be a bounded Monge subsolution and $v \in C(\Oba)\cap \lipl(\Omega)$ be 
a bounded Monge supersolution of \eqref{eikonal eq}. If 
\beq\label{bdry verify monge}
\lim_{\sigma\to 0}\sup\left\{u(x)-v(x): x\in \Oba, \ d(x, \pO)\leq \sigma \right\}\leq 0, 
\eeq
then $u\le v$ in $\Oba$. %Here $d(x, \pO)$ is given by $\inf_{y\in \pO} d(x, y)$.
\end{theorem}

\begin{proof}
Since $u$ and $v$ are assumed to be bounded, we may assume that $u, v\geq 0$ by adding an appropriate positive constant to both of them. It suffices to prove that $c u\le v$ in $\Omega$ for any fixed $c\in (0,1)$. 
Assume by contradiction that 
there exists $c\in (0,1)$ such that $\sup_{\Omega}(c u-v)> \beta$ for some $\beta>0$.  
By \eqref{bdry verify monge}, we may take $\sigma>0$ small such that 
\[
cu(x)-v(x)\leq u(x)-v(x)\leq {\beta\over 2}
\]
for all $x\in \Oba\setminus \Omega_\sigma$, where we denote $\Omega_r=\{x\in\Omega: d(x, \pO)>r\}$ for $r>0$. 
We choose $\vep\in (0, \sigma/2)$ such that
\[
\sup_\Omega (c u-v)>\beta+\vep^2
\]
and
\beq\label{eps small}
\vep<(1-c) \inf_{\Omega_{\sigma/2}} f.
\eeq
We have such an $\vep>0$ because $\inf_{\Omega}f>0$.
Thus there exists $x_0\in \Omega$ such that $c u(x_0)-v(x_0)\geq \sup_\Omega (c u-v)-\vep^2>\beta$ 
and therefore $x_0\in \Omega_\delta$. 

By Ekeland's variational principle in Corollary \ref{cor:ekeland} with Remark \ref{rmk ekeland}, there exists 
$x_\vep\in \ol{B_\vep(x_0)}\subset \Omega_{\sigma/2}$ such that 
\[
c u(x_\vep)-v(x_\vep)\geq c u(x_0)-v(x_0) 
\]
and $x\mapsto c u(x)-v(x)-\vep d(x_\vep, x)$ attains a local maximum in $\Omega$ at $x=x_\vep$.
It follows that
\beq\label{comparison monge1}
v(x_\vep)-v(x)\le c u(x_\vep)-c u(x)+\vep d(x_\vep, x)
\eeq
for all $x\in B_{r}(x_\vep)$ when $r>0$ is small enough. 
Since $v$ is a Monge supersolution of \eqref{eikonal eq} and hence
satisfies~\eqref{eq:Monge-supsubsol-Eik}, there exists 
a sequence $\{y_n\}\subset \Omega$ such that 
\[
\lim_{y_n\to x_\vep}\frac{[v(y_n)-v(x_\vep)]_-}{d(x_\vep, y_n)}\geq f(x_\vep)>0.
\]
Note that here it is crucial to have $f(x_\vep)>0$ so that for large $n\geq 1$ we have
$v(y_n)<v(x_\vep)$. 
Hence, by \eqref{comparison monge1} we have
\[
\begin{aligned}
f(x_\vep)& \leq \lim_{y_n\to x_\vep}\frac{c [u(y_n)-u(x_\vep)]_-}{d(y_n, x_\vep)}+\vep\\
&\leq \limsup_{x\to x_\vep}\frac{c [u(x)-u(x_\vep)]_-}{d(x, x_\vep)}+\vep=c |\nabla^- u|(x_\vep)+\vep
\end{aligned}
\]
Using the fact that $u$ is a Monge subsolution, we get 
\[
f(x_\vep)\leq c f(x_\vep)+\vep,
\]
which contradicts the choice of $\vep>0$ as in \eqref{eps small}. Our proof is thus complete.
\end{proof}

As can be easily observed in the proof above, we do not even need to assume the continuity of $f$. It motivates us to further study the case when $f$ is discontinuous and singular. This will be discussed in our forthcoming paper \cite{LShZ2}.

In order for the comparison argument in Theorem \ref{thm:comparison monge} to work, it is essential to use $|\nabla^-u|$ in the definitions of Monge sub- and supersolutions. One cannot replace $|\nabla^-u|$ by $|\nabla u|$. A simple counterexample with $|\nabla u|$ is as follows: 
\begin{example}{}{}
In $\Omega=(-1,1)\subset\R$, both $u_1(x)=1-|x|$ and $u_2(x)=|x|-1$ are solutions of \eqref{eikonal eq} with $f\equiv 1$ that satisfy the  
boundary condition $u(\pm 1)=0$. But only $u_1$ is the Monge solution. The proof is left as an easy exercise; see Problem \ref{prob:monge-eikonal}.
\end{example}

\subsection{Equivalence for solutions to eikonal equation}\label{sec:equivalence}

We prove that the Monge solution is equivalent to other notions of solutions to the eikonal equation in metric spaces.

\begin{proposition}{Relation between supersolutions}{eikonal super}
Let $(\X,  d)$ be a complete length space and $\Omega$ be an open set in $\X$. 
Assume that $f\in C(\Omega)$ %\sout{is locally uniformly continuous} 
and $f\ge 0$ in $\Omega$.  Let $u\in \lipl(\Omega)$. Then the following results hold. 
\begin{enumerate}
\item[(i)] Any Monge supersolution of \eqref{eikonal eq} is an s-supersolution of \eqref{eikonal eq}. 
\item[(ii)] If $f$ is locally uniformly continuous in $\Omega$, then any s-supersolution of \eqref{eikonal eq} is a Monge supersolution of \eqref{eikonal eq}.
\item[(iii)] If $f>0$ on $\Omega$, then any local c-supersolution of \eqref{eikonal eq} is a Monge supersolution of \eqref{eikonal eq}. 
\end{enumerate}
\end{proposition}
\begin{proof}
(i) %Let us first show the equivalence between a Monge supersolution and an s-supersolution. 
%We begin with the implication ``$\Rightarrow$''. 
Let $u$ be a Monge supersolution. 
Suppose that there exist 
 $\psi_1\in \ol{\mathcal{C}}(\Omega)$ and $\psi_2\in \lipl(\Omega)$ such that $u-\psi_1-\psi_2$ attains a local minimum at $x_0$. 
 If $f(x_0)=0$, then the desired inequality
\[
 |\nabla \psi_1|(x_0)\geq -|\nabla \psi_2|^\ast(x_0)
\]
is trivial. It thus suffices to consider the case $f(x_0)>0$. By the definition of Monge supersolutions, for any $x_0\in \Omega$, 
\[
|\nabla^- u|(x_0)\geq f(x_0)>0. 
\]
Then for any $\delta>0$, for each $\vep>0$ we can find $x_\vep\in \Omega$ with $x_\vep\to x_0$ as $\vep\to 0$ such that
\beq\label{eq weak1}
u(x_\vep)-u(x_0)\leq (-f(x_0)+\delta) d(x_0, x_\vep). 
\eeq
It follows from \eqref{eq weak1} and the maximality of $u-\psi_1-\psi_2$ at $x_0$ that 
\[
\psi_1(x_\vep)-\psi_1(x_0)+\psi_2(x_\vep)-\psi_2(x_0)\leq (-f(x_0)+\delta) d(x_0, x_\vep), 
\]
which implies that 
\[
{|\psi_1(x_\vep)-\psi_1(x_0)|\over  d(x_0, x_\vep)}
\ge \frac{\psi_1(x_0)-\psi_1(x_\vep)}{d(x_0,x_\vep)}
\geq f(x_0)-\delta-{|\psi_2(x_\vep)-\psi_2(x_0)|\over  d(x_0, x_\vep)}.
\]
Letting $\vep\to 0$ and then $\delta\to 0$, we obtain
\[
|\nabla \psi_1|(x_0)\geq f(x_0)-|\nabla \psi_2|^\ast(x_0).
\]
It follows that $u$ is an s-supersolution. 

(ii) We assume that $u$ is an
s-supersolution. Suppose that $u$ is not a Monge supersolution. Then
there exists $x_0\in \Omega$ such that 
\[
\limsup_{x\to x_0}\frac{[u(x_0)-u(x)]_+}{ d(x,x_0)}<f(x_0).
\]
A contradiction is immediately obtained if $f(x_0)=0$. We thus only consider $f(x_0)>0$ below.
Then there exists $\delta>0$ such that for all $x\in B_\delta(x_0)\setminus \{x_0\}$, 
\beq\label{eq monge-s1}
\frac{u(x_0)-u(x)}{ d(x,x_0)}-f(x_0)\leq -2\delta.
\eeq
By the continuity of $f$, for any $0<\vep<\min\{\delta,f(x_0)/2\}$, we can choose $0<r<\delta$ such that 
\beq\label{continuity f}
|f(x)-f(x_0)|\leq \vep \quad \text{for all $x\in B_r(x_0)$.}
\eeq
Observe that $f(x)\ge f(x_0)/2>0$ whenever $x\in B_r(x_0)$.
Let
\[
v(x):=u(x_0)-(f(x_0)-\vep) d(x,x_0)+\delta r.
\]
Then in view of \eqref{eq monge-s1}, we have $v(x)\le u(x)$ for all $x\in \partial B_r(x_0)$.  
Moreover,  we claim that $v$ is an s-subsolution of 
 \[
 |\nabla v|=f \quad \text{in $B_r(x_0)$}.
 \]
Indeed, for any $x\in B_r(x_0)$, if $v-\psi_1-\psi_2$ achieves a maximum at $x$, where 
$\psi_1\in \underline{\mathcal{C}}(\Omega)$ and $\psi_2\in \text{Lip}_{loc}(\Omega)$, then
\[
\psi_1(x)-\psi_1(y)\le v(x)-v(y)+\psi_2(y)-\psi_2(x).
\]
It follows that 
\[
\begin{aligned}
|\nabla \psi_1|(x)=\limsup_{y\to x}\frac{\psi_1(x)-\psi_1(y)}{d(x, y)} &\le \limsup_{y\to x}\frac{v(x)-v(y)}{d(x, y)}+\limsup_{y\to x}\frac{\psi_2(y)-\psi_2(x)}{d(x, y)}\\
&\le f(x_0)-\vep+|\nabla \psi_2|^*(x)\\
&\le f(x)+|\nabla \psi_2|^*(x).
\end{aligned}
\]
The claim has been proved. 
Applying the comparison principle for s-solutions, we have $v\le u$ in $B_r(x_0)$. This contradicts the fact that 
\[
v(x_0)=u(x_0)+\delta r>u(x_0).
\]
Our proof for the equivalence of supersolutions is now complete.

(iii)  It follows immediately from (ii) and Remark \ref{local c to s}.
Note that Remark \ref{local c to s} requires $\inf_{B_r(x)}f>0$ for any $x\in \Omega$ and $r>0$ small, which is implied by the continuity and positivity of $f$ in $\Omega$. A direct alternative proof is as follows. By Remark \ref{local c-super prop},  
%\ref{local c to s}, 
for each $x\in \Omega$ there is a sufficiently small 
$r>0$ such that for each $\vep>0$ we can find an arc-length parametrized curve $\gamma_\vep\in\A_{x}([0,\infty), B_r(x))$ satisfying \eqref{c-super char2} for all $0\le t< T^+_{B_r(x)}[\gamma_\vep]$. 
Note that $T_\Omega^+[\gamma_\vep]\geq \tilde{d}(x, \partial B_r(x))=d(x, \partial B_r(x))$. 
%Note that $T^+_{B_r(x)}[\gamma_\vep]<\infty$ due to the boundedness of $u$ in $B_r(x)$. 
Taking $\gamma_\vep(\sqrt{\vep})$  with $\sqrt{\vep}<r$ on the curve, we have
\[
\limsup_{y\to x} {u(x)-u(y)\over d(x, y)}\geq \limsup_{\vep\to 0} {1\over \sqrt{\vep}} \int_0^{\sqrt{\vep}} \left(f(\gamma_\vep (s))\, ds-{\vep (\sqrt{\vep}+1)\over \sqrt{\vep}}\right)=f(x),
\]
which is precisely the Monge supersolution property at $x\in \Omega$. 
If $u$ is a (global) c-supersolution, then by Proposition \ref{prop:c-super char}, for any $\vep>0$  %satisfying $\sqrt{\vep}<\min\{r,\tilde{d}(x_0, \pO)\}$, 
we can also find $\gamma_\vep\in  C_{x}$ with $T_\Omega^+[\gamma_\vep]<\infty$ such that \eqref{c-super char2} 
holds for all $0\leq t\leq T_\Omega^+[\gamma_\vep]$. We can still take $\gamma_\vep(\sqrt{\vep})$ to obtain the same estimate above. 
\end{proof}

Let us next explore the relation between the three notions of subsolutions.

\begin{proposition}{Relation between c- and Monge subsolutions}{eikonal sub1}
Let $(\X, d)$ be a complete length space and $\Omega$ be an open set in $\X$. Assume 
that $f\in C(\Omega)$ with $f\geq 0$.  Let $u\in C(\Omega)$. Then 
$u$ is a c-subsolution of \eqref{eikonal eq} if and only if it is a Monge subsolution of \eqref{eikonal eq}. 
\end{proposition} 
\begin{proof}
We begin with a proof of the implication ``$\Rightarrow$''. 
%Since $(\X, d)$ is a length space, for any $x, y\in \Omega$ and $\vep>0$, one can take a curve $\gamma: [0, \ell(\gamma)]\to \Omega$ joining $x, y$ with $\ell(\gamma)\leq d(x, y)+\vep$. By definition of c-subsolutions, $u\circ \gamma$ is a subsolution of the eikonal equation in one dimension. We thus have 
%\[
%u(x)-u(y)\leq \int_0^{\ell(\gamma)} f(\gamma(s))\, ds\leq (d+\vep)\sup_{\Omega} f.
%\]
%Letting $\vep\to 0$, 
By Proposition \ref{prop:c-sub}, we obtain the local Lipschitz continuity of c-subsolutions $u$.
%is provided in Lemma~\ref{lem c-lip}, i
It thus suffices to 
verify that $|\nabla^- u|(x_0)\leq f(x_0)$ for every $x_0\in \Omega$. 
To this end, we fix $x_0\in\Omega$ and choose $r>0$ small such that $B_r(x_0)\subset \Omega$.
Since $u$ is a c-subsolution of \eqref{eikonal eq}, 
using \eqref{prop c-sub eq} with $s=0$ and $\gamma(0)=x_0$ we have
\[  
u(x_0)-u(\gamma(t))\leq \int_0^t f(\gamma(s))\, ds
\] 
for any $\gamma\in \mathcal{A}_{x_0}(\R, B_r(x_0))$ and $0\le t\le T^+_{B_r(x_0)}[\gamma]$. By the assumption that $\X$ is a length space, for any $x\in B_r(x_0)$, we can find a curve $\gamma$ in $\X$ joining $x, x_0$ with $d(x, x_0)\leq \ell(\gamma)<r$. % It follows that $\gamma$ stays in $B_r(x_0)$, for otherwise $\ell(\gamma)\geq r$.
Therefore  
\[
{u(x_0)-u(x)\over  d(x, x_0)}\leq {1\over  d(x, x_0)} \int_0^t f(\gamma(s))\, ds \leq \frac{1}{d(x,x_0)} \ell(\gamma) \sup_{B_r(x_0)}f
\]
for any $x\in B_r(x_0)\setminus\{x_0\}$ and any curve $\gamma\subset B_r(x_0)$ joining $x_0$ and $x$. Taking the infimum over all such $\gamma$, we have 
\[
{u(x_0)-u(x)\over  d(x, x_0)}\leq  \sup_{B_r(x_0)}f.
\]
Sending $r\to 0$, by the continuity of $f$ we obtain 
\[
|\nabla^- u|(x_0)\leq f(x_0),
\]
as desired. 

We next prove the reverse implication ``$\Leftarrow$''. We again fix $x_0\in \Omega$ arbitrarily.  We take an 
arbitrary curve $\gamma\in \A_{x_0}(\R, \Omega)$; in particular $\gamma(0)=x_0$. Suppose that there is a function 
$\phi\in C^1(\R)$ such that $t\mapsto u(\gamma(t))-\phi(t)$ attains a local maximum at $t=0$. Then there is some $t_0>0$
such that we have
\[
\phi(t)-\phi(0)\geq u(\gamma(t))-u(\gamma(0))
\]
when $-t_0<t<t_0$ . If $\phi'(0)=0$, then we obtain immediately the desired inequality 
\eqref{eq c-sub}. If $\phi'(0)\neq 0$, then without loss we may assume that 
$\phi'(0)<0$, in which case
$\phi(t)-\phi(0)<0$ for all $t\in (0, t_1)$ for sufficiently small $t_1\in (0, t_0)$. (If 
$\phi'(0)>0$, then we can consider $t\in (-t_1,0)$ instead below.) It follows that 
\[
|\phi'(0)|\leq \limsup_{t\to 0+}{u(\gamma(0))-u(\gamma(t))\over t}\leq |\nabla^- u|(x_0).
\]
Since $u$ is a Monge subsolution, we have $|\nabla^- u|(x_0)\leq f(x_0)$ and thus deduce \eqref{eq c-sub} again. 
\end{proof}

The relation between s-subsolutions and Monge subsolutions is as follows. 
\begin{proposition}{Relation between s- and Monge subsolutions}{eikonal sub2}
Let $(\X, d)$ be a complete length space and $\Omega$ be an open set in $\X$. Assume that 
$f\in C(\Omega)$ %\sout{is locally uniformly continuous} 
and $f\geq 0$ in $\Omega$.
Let $u\in C(\Omega)$. Then the following results hold.
\begin{enumerate}
\item[(i)] If $u$ is a Monge subsolution of \eqref{eikonal eq}, then it is an s-subsolution of \eqref{eikonal eq}. 
\item[(ii)] If $f$ is locally uniformly continuous in $\Omega$ and $u$ is a locally uniformly continuous s-subsolution of \eqref{eikonal eq}, then it is a Monge subsolution of \eqref{eikonal eq}. 
\end{enumerate}
\end{proposition}
\begin{proof}
(i)  It is an immediate consequence of Proposition \ref{prop:sub} and Proposition \ref{prop:eikonal sub1}. In what follows let us also include a direct proof. %, following the argument for Proposition \ref{prop:eikonal super}. 
Let $u$ be a Monge subsolution. Suppose that there exist 
 $\psi_1\in \ul{\mathcal{C}}(\Omega)$ and $\psi_2\in \lipl(\Omega)$ such that $u-\psi_1-\psi_2$ attains a local maximum at $x_0$. Then the Monge subsolution property yields $|\nabla^- u|(x_0)\leq f(x_0)$.
 %If $f(x_0)=0$, then the desired inequality
%\[
% |\nabla \psi_1|(x_0)\geq -|\nabla \psi_2|^\ast(x_0)
%\]
%is trivial. It thus suffices to consider the case $f(x_0)>0$. By the definition of Monge supersolutions, for any $x_0\in \Omega$, 
%\[
%|\nabla^- u|(x_0)\geq f(x_0)>0. 
%\]
Then for any $\delta>0$, % we can find $x_\vep\in \Omega$ with $x_\vep\to x_0$ as $\vep\to 0$ such that
\beq\label{eq weak1 sub}
u(x_0)-u(x)\leq (f(x_0)+\delta) d(x_0, x)\quad \text{for all $x$ near $x_0$.}
\eeq
In view of the maximality of $u-\psi_1-\psi_2$ at $x_0$, we get, for any $x$ close to $x_0$, 
\[
\psi_1(x_0)-\psi_1(x)+\psi_2(x_0)-\psi_2(x)\leq (f(x_0)+\delta) d(x_0, x), 
\]
which implies that 
\[
%{|\psi_1(x)-\psi_1(x_0)|\over  d(x_0, x)}
%\le 
\limsup_{x\to x_0}\frac{\psi_1(x_0)-\psi_1(x)}{d(x_0,x)}
\leq f(x_0)+\delta+\limsup_{x\to x_0}{|\psi_2(x)-\psi_2(x_0)|\over  d(x_0, x)}.
\]
Letting %$x\to x_0$ and then 
$\delta\to 0$ and using Problem \ref{prob:monge-equiv}(1), we obtain
\[
|\nabla \psi_1|(x_0)=|\nabla^-\psi_1|(x_0)\leq f(x_0)+|\nabla \psi_2|^\ast(x_0),
\]
which means that $u$ is an s-subsolution.

(ii) Take $\delta>0$ arbitrarily and $f_\delta=f+\delta$ in $\Omega$.  
Fix $x_0\in \Omega$ and $r>0$ small such that $B_{4r}(x_0)\subset \Omega$ and $u, f$ are both bounded
on $B_{4r}(x_0)$.  For $s\in [0, 4r)$ we set
\[
M_s:=\sup_{B_s(x_0)} f_\delta,
\]
and choose $M>0$ such that
\[
M>\max\left\{M_{2r}, {\sup_{B_{3r}(x_0)} u- u(x_0)\over r}\right\}. 
\]
Define a continuous function $g_r: [0, 3r)\to [0, \infty)$ by
\[
g_r(t):=M_r t+M(t-r)_+
\]
for $t\in [0, 3r)$. Due to the choice of $M$ above, the function defined by
\[
v_r(x)=u(x_0)+g_r(d(x_0, x))
\] 
for $x\in B_{3r}(x_0)$, satisfies $v_r\geq u$ on $\partial B_{2r}(x_0)$.
Moreover, for any $x\in B_{2r}(x_0)$ with $x\neq x_0$ and any $\vep>0$, we can find an arc-length 
parametrized curve $\gamma$ such that $\gamma(0)=x$, $\gamma(t_\vep)=x_0$ 
and $t_\vep-\vep\leq d(x, x_0)\leq t_\vep$. For any $t\in [0, t_\vep]$,  
we have
\[
d(x,\gamma(t))\ge d(x_0,x)-d(x_0,\gamma(t))\ge (t_\vep-\vep)-(t_\vep-t)=t-\vep.
\]
Therefore
\[
t-\vep \leq d(x, \gamma(t))\le \ell(\gamma|_{[0,t]})\leq d(x, x_0)+\vep-d(x_0, \gamma(t)).
\]
Taking $x_\vep=\gamma(\sqrt{\vep})$, we have 
\[
{v_r(x)-v_r(x_\vep)\over d(x, x_\vep)}
 \geq \frac{g_r(d(x_0, x))-g_r(d(x_0, x_\vep))}{d(x, x_0)-d(x_0, x_\vep)+\vep}\geq {1\over 1+2\sqrt{\vep}} 
 \frac{g_r(d(x_0, x))-g_r(d(x_0, x_\vep))}{d(x, x_0)-d(x_0, x_\vep)}
\] 
when $\vep>0$ is sufficiently small. 
Hence if $0<d(x,x_0)\le  r$,  then for sufficiently small $\vep>0$ we have $g_r(d(x_\vep,x_0))= M_rd(x_\vep,x_0)$,
and then by the choice of $M_r$ we have
\[
\limsup_{x_\vep\to x} \frac{g_r(d(x_0, x))-g_r(d(x_0, x_\vep))}{d(x, x_0)-d(x_0, x_\vep)}= M_r\ge f_\delta(x).
\]
If $r< d(x, x_0)<2r$, then for sufficently small $\vep$ we have $d(x_\vep,x_0)>r$ as well, and so we get
\[
\frac{g_r(d(x_0, x))-g_r(d(x_0, x_\vep))}{d(x, x_0)-d(x_0, x_\vep)}
=\frac{(M_r+M)d(x_0,x)-(M_r+M)d(x_0,x_\vep)}{d(x, x_0)-d(x_0, x_\vep)}.
\]
Therefore as $d(x_0,x)<2r$, we have 
\[
\limsup_{x_\vep\to x} \frac{g_r(d(x_0, x))-g_r(d(x_0, x_\vep))}{d(x, x_0)-d(x_0, x_\vep)}\geq M> M_{2r}\ge f_\delta(x).
\]
In either case, we see that $|\nabla^- v_r|(x)\geq f_\delta(x)$; in other words, $v_r$ is a Monge supersolution of 
\[
|\nabla u|=f_\delta \quad \text{in $B_{2r}(x_0)\setminus \{x_0\}$}. 
\]
In view of Proposition \ref{prop:eikonal super}(i), we see that $v_r$ is an s-supersolution of the same equation
(keep in mind also that by its construction, $v_r$ is $2M$-Lipschitz); in particular, we have
\beq\label{bdry verify1}
v_r(x)-u(y)\geq v_r(x)-v_r(y)\geq -2Md(x, y)
\eeq
for all $x\in B_{2r}(x_0)$ and $y\in \partial B_{2r}(x_0)\cup \{x_0\}$. 

On the other hand, $u$ is an s-subsolution and is uniformly continuous in $B_{2r}(x_0)$ with some modulus $\sigma_0$. We have
\[
u(x)-u(y)\leq \sigma_0(d(x, y))
\]
for all $x\in B_{2r}(x_0)$ and $y\in \partial B_{2r}(x_0)\cup \{x_0\}$. Combining this with \eqref{bdry verify1}, 
we have shown that
\[
\sup\{u(x)-v_r(y): d(x, \partial\Omega)+d(y, \partial \Omega)+d(x, y)<\tau\}\le 0\quad \text{as $\tau\to 0$.}
\]
%the condition \eqref{bdry verify0} holds with $\Omega=B_{2r}(x_0)\setminus \{x_0\}$, 
%$v=v_r$, $\zeta=u$ and $\sigma(s)=\max\{Ms, \sigma_0(s)\}$ for $s\geq 0$. 
Since $f_\delta=f+ \delta$ in $\Omega$, the function $u$ must also be an s-subsolution for the eikonal equation 
related to the function $f_\delta$. We thus can use the comparison result, Theorem \ref{thm:comparison s-sol},  
%\cite[Theorem~5.3]{GaS} 
to get $u\leq v_r$ in $B_{2r}(x_0)\setminus \{x_0\}$. Letting $\delta\to 0$, we are led to
\[
u(x)\leq u(x_0)+d(x, x_0)\sup_{B_r(x_0)} f\quad \text{for all $x\in B_r(x_0)$. }
\]
One can use the same argument to show that for all $x, y\in B_{r/4}(x_0)$ (and therefore $d(x, y)\leq r/2$),
\[
u(y)\leq u(x)+d(x, y)\sup_{B_{r/2}(x)} f\leq  u(x)+d(x, y) \sup_{B_r(x_0)} f,
\]
which yields (recalling that we chose $r>0$ small enough so that $f$ is bounded on $B_{4r}(x_0)$)
\[
|u(x)-u(y)|\leq d(x, y)\sup_{B_r(x_0)} f. 
\]
This immediately implies that 
\beq\label{sub char}
|\nabla^- u|(x_0)\leq |\nabla u|(x_0)\leq f(x_0).
\eeq
Hence, we can conclude that $u$ is a Monge subsolution of \eqref{eikonal eq}, since $x_0$ is arbitrarily taken. 
\end{proof}

In the proof of (ii) above, the local uniform continuity of $u$ and $f$ (especially near $x_0$) enables us to adopt 
the comparison principle. The uniform continuity can be removed if the space $(\X, d)$ has some compactness a priori.

Let us summarize our equivalence results for these three notions of solutions to the associated Dirichlet problem \eqref{eikonal eq}\eqref{bdry cond}. Here the boundary condition \eqref{bdry cond} is interpreted as the following condition:
\beq\label{new bdry regularity}
\text{$\exists$ a modulus $\omega_b$ such that }\quad |u(x)-g(y)|\leq \omega_b(d(x, y))\quad \text{for all $x\in \Omega$ and $y\in \pO$.}
\eeq
\begin{theorem}{Equivalence of solutions to Dirichlet problem}{}
Let  $(\X, d)$ be a complete length space.
%Assume that $(\X,d)$ also satisfies \eqref{eq ast}.
Suppose that $\Omega\subsetneq \X$ is a bounded domain. % with respect to the metric $\tilde{d}$. 
Assume that $f$ is uniformly continuous in $\Omega$ and satisfies $\inf_{\Omega} f>0$. Let $g\in C(\partial \Omega)$.
%Assume in addition that and
%uniformly continuous on $\partial\Omega$, and
%there exists a modulus of continuity $\omega_b$ such that 
Then the notions of c-solution, s-solution and Monge solution $u\in \lipl(\Omega)$ of \eqref{eikonal eq} that satisfy \eqref{new bdry regularity} %for some modulus $\omega_b$ 
are all equivalent. 
%If $u$ is a solution to \eqref{eikonal eq} in the sense of any of the  
%then $u$ is the unique solution of \eqref{eikonal eq} satisfying~\eqref{new bdry regularity}. 
\end{theorem}

\begin{proof}
    In view of (i) and (iii) of Proposition \ref{prop:eikonal super}, we see that c-supersolutions and Monge supersolutions of \eqref{eikonal eq} are s-supersolutions. By Proposition \ref{prop:eikonal sub1} and Proposition \ref{prop:eikonal sub2}, we also see that all c-subsolutions and Monge subsolutions are s-subsolutions. It follows that all c-solutions and Monge solutions of \eqref{eikonal eq} are s-solutions. Under the condition \ref{new bdry regularity} and the uniform continuity of $f$, the comparison principle, Theorem \ref{thm:comparison s-sol}, guarantees the uniqueness of s-solutions. We thus deduce that all of these solutions coincide. 
\end{proof}

We can also show local equivalence of these notions. 

\begin{theorem}{Local equivalence between solutions of eikonal equation}{equiv Monge}
Let $(\X,  d)$ be a complete length space and $\Omega\subset \X$ be an 
open set. 
Assume that $f$ is locally uniformly continuous and $f>0$ in $\Omega$.  
Let $u\in C(\Omega)$. Then the following statements are equivalent: 
\begin{enumerate}
\item[(a)] $u$ is a local c-solution of \eqref{eikonal eq};
\item[(b)] $u$ is a locally uniformly continuous s-solution of \eqref{eikonal eq};
\item[(c)] $u$ is a Monge solution of \eqref{eikonal eq}. 
\end{enumerate}
In addition, if any of (a)--(c) holds, then $u$ is locally Lipschitz with 
\beq\label{regular0}
|\nabla u|(x)=|\nabla^- u|(x)=f(x)\quad \text{for all  $x\in \Omega$.}
\eeq
\end{theorem}

To prove this result, besides Propositions \ref{prop:eikonal super}, \ref{prop:eikonal sub1} and \ref{prop:eikonal sub2}, we additionally need the following result. 
\begin{proposition}{Relation between Monge and local c-solutions}{eikonal solution}
Let $(\X,  d)$ be a complete length space and $\Omega$ be an open set in $\X$. 
Assume that $f$ is locally uniformly continuous 
and $f>0$ in $\Omega$.  If $u$ is a Monge solution of \eqref{eikonal eq}, then $u$ is a local c-solution of \eqref{eikonal eq}.
\end{proposition}
\begin{proof}
Suppose that $u$ is a Monge solution of \eqref{eikonal eq} in $\Omega$. By 
Proposition~\ref{prop:eikonal sub1}, we know that $u$ must be a c-subsolution. 
It suffices to show that $u$ is a local c-supersolution of \eqref{eikonal eq}. 

 For any $x_0\in \Omega$, take $r>0$ small such that $u$ is Lipschitz in $\overline{B_r(x_0)}$, that is, there exists $L>0$ such that
\beq\label{lip local solution}
|u(x)-u(y)|\leq Ld(x, y) \quad\text{for any $x, y\in \overline{B_r(x_0)}$.}
\eeq
Letting  $g(y)=u(y)$ for all $y\in \partial B_r(x_0)$, by \eqref{eq optimal control} with 
$\Omega=B_r(x_0)$ we have the unique c-solution $U$ in $B_r(x_0)$ given by
\begin{equation}\label{local formula}
U(x):=\inf\bigg\{ \int_{0}^{t_r^+} f(\gamma(s))\, ds+ u\left(\gamma(t_r^+)\right)\, :\, \gamma\in \A_{x}(\R, \X) \text{ with }0<T^+_{B_r(x_0)}[\gamma]<\infty\bigg\}.
\end{equation}
It follows from Proposition \ref{prop:eikonal sub1} and Proposition \ref{prop:eikonal super}(ii) that $U$ is a 
Monge solution of the eikonal equation in $B_r(x_0)$. 
Note also that, by Proposition \ref{prop:bdry regularity}(1), 
\[
U(x)-u(y)\leq d(x, y )\max\left\{L,\ \sup_{B_r(x_0)} f\right\} \quad \text{for any $x\in B_r(x_0)$ and $y\in \partial B_r(x_0)$.}
\]
We then can adopt the comparison principle,  Theorem \ref{thm:comparison monge}, to get 
$U\leq u$ in $B_r(x_0)$.
In view of \eqref{local formula}, it follows that for any $x\in B_r(x_0)$ and any $\vep>0$ small, there 
exists a curve $\gamma_\vep\in A_x(\R, \X)$ such that
\beq\label{local solution1}
u(x)\geq U(x)\geq \int_{0}^{t_r^+} f(\gamma_\vep(s))\, ds+ u\left(\gamma_\vep(t_r^+)\right)-\vep,
\eeq
where $t_r^+=T^+_{B_r(x_0)}[\gamma_\vep]$ denotes the exit time of $\gamma_\vep$ from $B_r(x_0)$. 
On the other hand, since $u$ is a c-subsolution, we can use Proposition \ref{prop:c-sub} to get, for any $0\leq t\leq t_r^+$,
\beq\label{local solution2}
u(\gamma_\vep(t))\leq u\left(\gamma_\vep(t_r^+)\right)+\int_t^{t_r^+} f(\gamma_\vep(s))\, ds. 
\eeq
Combining \eqref{local solution1} and \eqref{local solution2}, we deduce that for any $0\leq t\leq t_r^+$,  
\[
u(x)\geq u(\gamma_\vep(t))+\int_0^t f(\gamma_\vep(s))\, ds-\vep.
\]
Setting 
\[
\gamma(t)=\begin{cases}
\gamma_\vep(t) & \text{if $t\geq 0$,}\\
\gamma_\vep(-t) & \text{if $t<0$,}
\end{cases}
\quad\text{and }\quad w(t)=u(x)-\int_0^t f(\gamma(s))\, ds
\]
for $t\in \R$, we easily see that $(\gamma, w)$ satisfies the conditions for local c-supersolutions in 
Definition \ref{def:local c}. Indeed, $w$ is of class
$C^1$ in $(-t_r^+, t_r^+)\setminus \{0\}$ with $w'=f\circ \gamma$ in $(-t_r^+, t_r^+)\setminus \{0\}$. 
If there is $\phi\in C^1(-t_r^+, t_r^+)$ such that $w-\phi$ achieves a minimum at some 
$t_0\in (-t_r^+, t_r^+)$, then $t_0\neq 0$ since $f>0$ in $\Omega$. It then follows that $\phi'(t_0)=w'(t_0)$, which yields 
\[
|\phi^\prime|(t_0)=|w^\prime|(t_0)=f(\gamma(t_0)).  
\]
Hence, $u$ is a local c-supersolution and therefore a local c-solution. 
\end{proof}

We are now in a position to show Theorem \ref{thm:equiv Monge}.
\begin{proof}[Proof of Theorem \ref{thm:equiv Monge}]
The proof consists of the results in Propositions \ref{prop:eikonal super}, \ref{prop:eikonal sub1}, \ref{prop:eikonal sub2} 
and \ref{prop:eikonal solution}. In addition, 
combining \eqref{sub char} and the definition of Monge supersolutions at any $x_0\in \Omega$, we have \eqref{regular0} if any of (a), (b) and (c) holds. 
\end{proof}

The local uniform continuity of $f$ and $u$  
in Theorem~\ref{thm:equiv Monge} can be dropped if the space $(\X, d)$ is assumed to proper, 
that is, any closed bounded subset of $\X$ is compact.

\begin{corollary}{Local equivalence in a proper space}{equiv Monge}
Let $(\X,  d)$ be a proper complete 
geodesic space and $\Omega$ be a bounded open set in $\X$. Assume that $f\in C(\Omega)$ and $f>0$ in $\Omega$. 
Let $u\in C(\Omega)$. Then the following statements are equivalent: 
\begin{enumerate}
\item[(a)] $u$ is a c-solution of \eqref{eikonal eq};
\item[(b)] $u$ is a s-solution of \eqref{eikonal eq};
\item[(c)] $u$ is a Monge solution of \eqref{eikonal eq}. 
\end{enumerate}
In addition, if any of (a)--(c) holds, then $u\in \lipl(\Omega)$ and satisfies \eqref{regular0}. 
\end{corollary}

\subsection{General Hamilton-Jacobi equations}
We next turn to a more general class of Hamilton-Jacobi equations. In this case, c-solutions are no longer defined. We can still show the equivalence between Monge solutions and s-solutions under a monotonicity assumption on the Hamiltonian. Our discussion below is also taken from \cite{LShZ}.  
\begin{theorem}{Equivalence of Monge and $s$-solutions of general equations}{equiv general}
Let $(\X,  d)$ be a complete length space and $\Omega\subset \X$ be an open set. Let 
$H: \Omega\times \R\times [0, \infty)\to \R$ be continuous and satisfy the following conditions:
\begin{enumerate}
\item[(1)] $(x, r)\mapsto H(x, r, p)$ is locally uniformly continuous in the sense that for any $x_0\in \Omega$ 
and $r_0\in \R$, there exist $\delta>0$ small and a modulus of continuity $\omega$ such that 
\[
|H(x_1, r_1, p)-H(x_2, r_2, p)|\leq \omega\left((1+p)(d(x_1, x_2)+|r_1-r_2|)\right)
\]
for all $x_1, x_2\in B_\delta(x_0)$, $r_1, r_2\in [r_0-\delta, r_0+\delta]$ and $p\geq 0$. 
\item[(2)]  For any $x_0\in \Omega$ and $r_0\in \R$, there exist $\delta>0$ and  $\lambda_0>0$ such that 
\[
p\mapsto H(x, r, p)-\lambda_0 p
\] 
is increasing in $[0, \infty)$ for every $(x, r)\in B_\delta(x_0)\times [r_0-\delta, r_0+\delta]$. 
\item[(3)] $p\mapsto H(x, r, p)$ is coercive in the sense that 
\beq\label{coercivity2}
\inf_{(x, r)\in \Omega\times [-R, R]} H(x, r, p)\to \infty \quad \text{as $p\to \infty$ for any $R>0$. }
\eeq
\end{enumerate}
 Then $u$ is a Monge solution of \eqref{stationary eq} if and only if $u$ is a locally uniformly continuous 
 s-solution of \eqref{stationary eq}. In addition, such $u$ is locally Lipschitz in $\Omega$. 
\end{theorem}

\begin{proof}
Let $u$ be either a Monge solution or a locally uniformly continuous s-solution. 
We first claim that 
\beq\label{hamiltonian low bound}
H(x, u(x), 0)\leq 0\quad \text{for any $x\in \Omega$.}
 \eeq
(We here do not assume existence of solutions, but if a solution exists, then \eqref{hamiltonian low bound} holds.)
Thanks to the condition (2), this is clearly true when $u$ is a Monge solution.  It thus suffices to 
show \eqref{hamiltonian low bound} for a locally uniformly continuous $s$-solution $u$. 
Fix $x_0\in \Omega$ arbitrarily. Let 
\[
\psi_1(x)={1\over \vep}d(x, x_0)^2
\]
for $\vep>0$ small. Then, due to the local boundedness of $u$, there exist $\delta>0$ and  
$y_\vep\in B_\delta(x_0)\subset \Omega$  such that 
\[
(u-\psi_1)(y_\vep)\geq \sup_{B_\delta(x_0)} (u-\psi_1)-\vep^2
\]
and $y_\vep\to x_0$ as $\vep\to 0$.
By Ekeland's variational principle (Corollary \ref{cor:ekeland} with Remark \ref{rmk ekeland}), there is 
a point $x_\vep\in B_\vep(y_\vep)$ such that 
\[
(u-\psi_1)(x_\vep)\geq (u-\psi_1)(y_\vep) 
\]
and $u-\psi_1-\psi_2$ attains a local maximum in $B_\delta(x_0)$ at $x_\vep\in B_\vep(y_\vep)$, where 
\[
\psi_2(x)=\vep d(x_\vep, x). 
\] 
It is clear that $x_\vep\to x_0$ as $\vep\to 0$. 
Since $u$ is an $s$-subsolution, we have 
\[
H_\vep\left(x_\vep, u(x_\vep), {2\over \vep}d(x_\vep, x_0)\right)=\inf_{|r|\leq \vep}H\left(x_\vep, u(x_\vep), {2\over \vep}d(x_\vep, x_0)+r\right)\leq 0,
\]
which, by the condition (2) together with \eqref{h-extension}, yields 
\[
H\left(x_\vep, u(x_\vep), 0\right) \leq 0.
\]
Letting $\vep\to 0$, by the continuity of $H$, we deduce \eqref{hamiltonian low bound} at $x=x_0$. 
We have completed the proof of the claim. 

By the coercivity condition (3) we can define a 
function $h: \Omega\to [0, \infty)$ to be
\beq\label{implicit}
h(x):=\inf\{p\geq 0: H(x, u(x), p)> 0\}, 
\eeq 
and, thanks to the continuity of $H$ and the condition (2), we see that 
for each $x\in\Omega$, $h(x)\geq 0$ is the unique value satisfying
\begin{equation}\label{eq:Hamil-Eik}
H(x, u(x), h(x))=0. 
\end{equation}

We next claim that $h$ is locally uniformly continuous in $\Omega$. To see this, fix $x_0\in \Omega$ and an 
arbitrarily small $\delta>0$. We take $x, y\in B_{\delta_1}(x_0)$ with $\delta_1>0$ sufficiently small 
such that $u(x), u(y)\in [u(x_0)-\delta, u(x_0)+\delta]$. Then by the condition (1) we have 
\beq\label{local uniform ham}
H(y, u(y), h(y))-H(x, u(x), h(y))\leq \omega\left((1+h(y))(d(x, y)+|u(x)-u(y)|)\right).
\eeq
%for any $p\geq 0$. 
Denote $W(x, y):=\omega\left((1+h(y))(d(x, y)+|u(x)-u(y)|)\right)$ for simplicity.
Since $H$ satisfies the condition (2), 
we can use \eqref{local uniform ham} to get, for any $x, y\in B_{\delta_1}(x_0)$,
\[
H\left(x, u(x), h(y)+{1\over \lambda_0} W(x, y) \right)
\geq H(x, u(x), h(y))+W(x, y)\geq H(y, u(y), h(y))=0,
\]
which, by \eqref{implicit} and \eqref{eq:Hamil-Eik}, yields 
\[
h(x)\leq h(y)+{1\over \lambda_0} W(x, y).
\]
We can analogously show that 
\[
h(x)\geq h(y)-{1\over \lambda_0} W(x, y).
\]
Therefore, we see that 
\[
0\leq h(x)\leq h(x_0)+{1\over \lambda_0}\omega((1+h(x_0))(\delta+\delta_1))\quad \text{for all $x\in B_{\delta_1}(x_0)$};
\]
in other words, $h$ is locally bounded in $\Omega$. The local uniform continuity of $h$ in $\Omega$ also follows.

From~\eqref{eq:Hamil-Eik} it is clear that $u$ is a Monge solution of~\eqref{stationary eq}  
if and only if $|\nabla^-u|(x)= h(x)$, that is, $u$ is a Monge solution of \eqref{eikonal eq} with $f=h$.
Note however that the function $h$ depends on $u$ implicitly in general.

We now show that $u$ is an s-solution of \eqref{stationary eq} if and only if $u$ is an s-solution of 
\eqref{eikonal eq} with $f=h$. 
To see this, suppose that there exist $x_0\in \Omega$, $\psi_1\in \ul{\mathcal{C}}(\Omega)$ and 
$\psi_2\in \text{Lip}_{loc}(\Omega)$ such that $u-\psi_1-\psi_2$ attains a maximum in $\Omega$ at $x_0$. 
Then by the monotonicity of $p\to H(x, r, p)$, the viscosity inequality \eqref{s-sub eq} holds
at $x_0$ if and only if 
\[
H\left(x_0, u(x_0), (|\nabla\psi_1|(x_0)-|\nabla \psi_2|^\ast(x_0))\vee 0\right)\leq 0,
\]
which, due to \eqref{implicit}, amounts to saying that
\[
|\nabla\psi_1|(x_0)-|\nabla \psi_2|^\ast(x_0)\leq h(x_0),
\]
that is, $u$ is an $s$-subsolution of \eqref{eikonal eq} with $f=h$. Analogous results for
supersolutions can be similarly proved. 

Noticing that Monge solutions and locally uniformly continuous s-solutions of \eqref{eikonal eq} have 
been proved to be (locally) equivalent in Theorem~\ref{thm:equiv Monge}, we immediately obtain the 
equivalence of both notions for \eqref{stationary eq} and local Lipschitz continuity. 
\end{proof}

As in Corollary \ref{cor:equiv Monge}, if $(\X, d)$ is additionally assumed to be proper, then in 
Theorem~\ref{thm:equiv general} we can drop the assumption on the local uniform continuity of s-solutions. 
Moreover, when $(\X, d)$ is proper, in the proof above we only need to show the continuity of $h$ as 
in \eqref{implicit}, since continuity implies local uniform continuity.  
Thus the condition (1) can be removed and (2) can be weakened by merely assuming that for every 
$(x, r)\in \Omega\times \R$, the mapping $p\mapsto H(x, r, p)$ is strictly increasing 
in $(0, \infty)$. 
Below we state the result without proofs.

\begin{theorem}{Equivalence of Monge and $s$-solutions in a proper space}{equiv general proper}
Let $(\X,  d)$ be a complete proper geodesic space and $\Omega\subset \X$ be an open set.  Let 
$H: \Omega\times \R\times [0, \infty)\to \R$ be continuous and 
be coercive as in \eqref{coercivity2}. 
Assume that, for every $(x, r)\in \Omega\times \R$, 
\begin{equation}\label{strict p-monotoncity}
H(x, r, p_1)<H(x, r, p_2)\quad \text{for all  $0<p_1<p_2<\infty$.}
\end{equation}
 Then $u$ is a Monge solution of \eqref{stationary eq} if and only if $u$ is an
 s-solution of \eqref{stationary eq}. In addition, such $u$ is locally Lipschitz in $\Omega$. 
\end{theorem}

The monotonicity assumptions on $p\to H(x, r, p)$ in Theorem~\ref{thm:equiv general} and 
Theorem~\ref{thm:equiv general proper}  enables us to apply an implicit function argument. 
Although it is not clear to us if one can weaken the requirement, the examples below show that the equivalence result fails to hold in general if $H$ is not monotone in $p$.

\begin{example}{}{}
Let $\Omega=\X=\R$ with the standard Euclidean metric. Let 
\[
H(p)=1-|p-2|+\max\{p-3, 0\}^2, \quad p\geq 0.
\]
One can easily verify that this Hamiltonian satisfies all assumptions in Theorem \ref{thm:equiv general} except for the monotonicity. 

It is not difficult to see that the function $u$ given by $u(x)=-3|x|$ for $x\in \R$
is a Monge solution of
\beq\label{example hj}
H(|\nabla u|)=0 \quad \text{in $\R$},
\eeq
since $|\nabla^- u|=3$ in $\R$.
 It is however not a conventional viscosity solution or an s-solution, though it is an s-supersolution. 
\end{example}

\begin{example}{}{}
Let $\Omega=\X=\R$ again. Set 
\[
H(p)=1-|p|+\max\{p-3, 0\}^2, \quad p\geq 0,
\]
which again satisfies all assumptions in Theorem \ref{thm:equiv general} but the monotonicity.
This time we have 
\[
u(x)=|x|, \quad x\in \R
\]
as a viscosity solution or s-solution of \eqref{example hj}. But it is not a Monge solution, since $|\nabla^- u|(0)=0$ and $H(0)\neq 0$. 
\end{example}

\subsection{Exercises}

\begin{problem}\label{prob:monge-equiv}
Let $\Omega$ be a domain in a length space $(\X, d)$. Let $x\in \Omega$. Prove the following statements. 
\begin{enumerate}
\item Under the condition $a\geq 0$, $|\nabla^- u|(x)\leq a$ holds if and only if 
\[
\limsup_{y\to x} \frac{u(x)-u(y)}{d(x, y)}\leq a.
\]
\item Under the condition $a>0$, $|\nabla^- u|(x)\geq a$ holds if and only if 
\[
\limsup_{y\to x} \frac{u(x)-u(y)}{d(x, y)}\geq a.
\]
\end{enumerate}

\end{problem}
\begin{problem}\label{prob:monge-eikonal}
    Show that both $u_1(x)=1-|x|$ and $u_2(x)=|x|-1$ satisfy $|\nabla u|(x)=1$ for all $x\in (-1,1)\subset\R$. Show also that $|\nabla^- u_1|(x)=1$ for all $x\in (-1, 1)$ but $u_2$ does not satisfy this equation.
\end{problem}

%\clearpage
%\section{Applications to infinity eigenvalue problem}

\clearpage

\section{Solutions to Exercises}

{\bf Section 2}
\begin{enumerate}
\item[2.1] It is easily seen that $d_E\leq d_T$. On the other hand, by Cauchy-Schwarz inequality, we obtain  %for $a_j=|x_j-y_j|$, $j=1, 2, \ldots n$,
%\[
%\left(\sum_{j=1}^n a_j\right)^2\leq \sum_{j=1}^n a_j^2 \sum_{j=1}^{n} 1.
%\]
%Hence,
\[
\sum_{j=1}^n |x_j-y_j|\leq  \left(\sum_{j=1}^{n} 1\right)^{1\over 2}\left(\sum_{j=1}^n (x_j-y_j)^2\right)^{1\over 2}
\]
for all $x,y\in \mathbb{R}^n$. We thus have $d_T\leq \sqrt{n}d_E$.

\item[2.2] 
(1) We only verify the triangle inequality. Let us first consider the case when $L_f(x, y)=\infty$. In this case, we have $\Gamma_f(x, y)=\emptyset$, which implies that either $\Gamma_f(x, z)=\emptyset$ or $\Gamma_f(y, z)=\emptyset$, for otherwise we can concatenate curves from each set to get an admissible curve joining $x$ and $y$. It thus follows that $L_f(x, y)= L_f(x, z)+L_f(y, z)$. If $L_f(x,z)$ or $L_f(y,z)$ is infinite, then we trivially have $L_f(x, y)\leq L_f(x, z)+L_f(y, z)$.

If on the other hand $L_f(x, y)<\infty$, $L_f(x,z)<\infty$, \emph{and} $L_f(y,z)<\infty$, then
then for any $\varepsilon>0$, there exist $\gamma_1\in \Gamma_f(x, z)$ and $\gamma_2\in \Gamma_f(z, y)$ such that 
\[
L_f(x, z)\geq I_f(\gamma_1)-\varepsilon, \quad L_f(z, y)\geq I_f(\gamma_2)-\vep. 
\]
By connecting the curve $\gamma_1$ and $\gamma_2$ to build a curve joining $x$ and $y$, we can easily see that 
\[
L_f(x, y)\leq I_f(\gamma_1)+I_f(\gamma_2)\leq L_f(x, z)+L_f(z, y)+2\varepsilon.
\]
We conclude the proof by sending $\varepsilon\to 0$.

(2) Suppose that $x_j \in \X$ is a Cauchy sequence with respect to the metric $L_f$. Thanks to the lower bound of $f$, it follows that $L_f(x, y)\geq  \alpha d(x, y)$ for all $x, y\in \Omega.$ Hence, $\{x_j\}$ is also a Cauchy sequence with respect to the metric $d$. Since $\X$ is a closed set in the complete space $\X$, there exists $x_0\in \X$ such that $d(x_j, x_0)\to 0$ as $j\to \infty$.  \\
By definition of Cauchy sequences, for any $\varepsilon>0$, we can take $x_{j_1}$ such that $L_f(x_{j_1}, x_i)<\vep/2$ for all $i\geq j_1$. Similarly, we can choose $x_{j_2}$ with $j_2>j_1$ satisfying $L_f(x_{j_2}, x_i)<\varepsilon/4$ for all $i\geq j_2$. We repeat this process to obtain a sequence $x_{j_k}$ such that $L_f(x_{j_k}, x_{j_{k+1}})<2^{-k}\varepsilon$. By definition of $L_f$, we can find curves $\gamma_k$ in $\X$ joining $x_{j_k}$ and $x_{j_k+1}$ such that 
\[
\int_{\gamma_k} f\, ds< 2^{-k}\varepsilon. 
\]
Concatenating these curves in order, we build a curve $\gamma$ connecting $x_{j_1}$ to $x_0$ satisfying 
\[
\int_\gamma f\, ds\leq \sum_{k\geq 1} 2^{-k}\varepsilon=\varepsilon,
\]
which yields $L_f(x_{j_1}, x_0)\leq \varepsilon$. In view of the arbitrariness of $\varepsilon>0$, we have actually proved the convergence of $x_j$ to $x_0$ in the metric $L_f$. 

\item[2.3] Assume $s_\gamma(t)$ is not continuous. Since $s_\gamma$ is increasing, there exists $\hat{t}\in (a,b)$ such that
\[
\lim_{t\rightarrow\hat{t}^+} s_\gamma(t)=\alpha \quad\quad\quad \lim_{t\rightarrow\hat{t}^-} s_\gamma(t)=\beta,
\]
and $\alpha-\beta=\delta>0$. Hence, for $a',b'$ satisfying $a<a'<\hat{t}<b'<b$ and sufficiently close to $\hat{t}$, we get that 
\[
\ell_\gamma([a',b'])=s_\gamma(b')-s_\gamma(a')>\frac{3\delta}{4}.
\]
In particular, any sub-curve of $\gamma$ containing $a',b'$ has length at least $\frac{3\delta}{4}$. 

We can choose a partition $P$ of $[a,b]$ such that 
\begin{itemize}
    \item[(i)] $\sum_{i=1}^n d(\gamma(t_i),\gamma(t_{i-1}))>\ell(\gamma)-\frac{\delta}{3};$
    
\item[(ii)] $d(\gamma(t_i),\gamma(t_{i-1}))\le \frac{\delta}{3}$ for all $i=1,\cdots, n$;

\item[(iii)] $\hat{t}\in (t_k,t_{k-1})$ for some $k\in\{1,\cdots, n\}$.
\end{itemize}

Such partition satisfying the above conditions always exists due to definition and the fact that $\gamma$ is continuous. For the sub-curve of $\gamma$ restricted to $[t_{k-1},t_k]$, since the length is at least $\frac{2\delta}{3}$, we can find a partition $\hat{P}$ of $[t_{k-1},t_k]$ such that $\sum_j d(\gamma(t_j-1), \gamma(t_j))>\frac{2\delta}{3}$. Define $\tilde{P}=P\cup \hat{P}$, then 
\[
\sum_{i\in P} d(\gamma(t_i), \gamma(t_{i-1}))=\sum_{i\in P, i\neq k}d(\gamma(t_i), \gamma(t_{i-1}))+\sum_{i\in \hat{P}}d(\gamma(t_i), \gamma(t_{i-1}))> \ell(\gamma)-\frac{2\delta}{3}+\frac{2\delta}{3},
\]
which gives a contradiction. 

\item[2.4] The implication ``$\Rightarrow$" follows from the Ekeland variational principle. We focus on ``$\Leftarrow$'' below. Let $\{x_m\}$ be a Cauchy sequence in $\X$, and define $f:\X\to \mathbb{R}$ as
\[
f(x)=\liminf_{m\to \infty} d(x,x_m), \quad x\in \X.
\]
It follows that $f$ is continuous, bounded below and $\inf_\X f=0$. Fix $\varepsilon\in (0,1)$. By assumptions, there exists a point $x\in \X$ such that $f(x)\le \varepsilon$ and for all $y\neq x$,
\[
f(x)<f(y)+\varepsilon d(x,y).
\]
By the definition of $f$ and the fact that $\{x_m\}$ is a Cauchy sequence, for any $\eta>0$ we can choose $y=x_{j}$ such that $d(x, x_{j})\leq \vep+\eta$ and $f(x_{j})<\eta$. We thus get $f(x)\leq \vep^2$.
%\[
%%d(x, x_{m_j})
%f(x)< d(x_{m_k}, x_{m_j})+\vep d(x, x_{m_k})
%\]
Repeating this argument, we get $f(x)\leq \vep^k$ for any $k=1, 2, \ldots$, which yields $f(x)=0$. In other words, there exists a subsequence of $x_m$ converging to $x$. Then the whole Cauchy sequence also converges to $x$.
\end{enumerate}

\begin{comment}
    \item[2.3] 

\[
L_f(x, 0)=\int_0^x f(s)\, ds=2\sqrt{|x|}
\]
for all $x\in (-1, 1)$ and therefore
\[
\frac{L_f(x, 0)}{d(x, 0)}=\frac{2}{\sqrt{|x|}}\to \infty\quad \text{as $x\to 0$.}
\]
Note however that $(-1,1)$ is bounded with respect to the metric $L_f$, and the topology generated by $L_f$ 
agrees with the Euclidean topology on $(-1,1)$; moreover, $[-1,1]$ is complete and compact in both metrics. 
\item[2.5]

It is not difficult to see that $(\X, d)$ is a complete geodesics space. For $x\in \X$ with coordinates $(x_1, x_2)$ in $\mathbb{R}^2$, let 
\[
f(x_1, x_2)=\begin{dcases}
1/x_1 &\text{if $x_2> 0$;}\\
1 & \text{if $x_2=0$.}
\end{dcases}
\]
Denoting $O=(0, 0)$, by direct calculations, for any $j\geq 2$ we have
\[
L_f(O, Q_j)=j \int_0^{1\over j}  \, ds+\int_0^{1\over j} ds=1+{1\over j}. 
\]
(Here we choose the optimal integration path from $Q_j$ to $P_j$ and then to $O$.) 
We thus observe that as $j\to \infty$, $d(O, Q_j)\to 0$ but $L_f(O, Q_j)\to 1$. 
\end{comment}

\medskip

{\bf Section 3}
\begin{enumerate}
\item[3.1] Let $v_\vep=(u_\vep)'$. Then $v=v_\vep$ satisfies 
\[
v^2-\vep v'=1 \quad \text{in $(-1, 1)$.}
\]
By the evenness of $u_\vep$, we have $v_\vep(0)=0$, which is used as the initial condition for the ODE above. Solving the ODE, we obtain $v_\vep(x)=-\tanh{(x/\vep)}$.
Integrating it and using $u_\vep(\pm1)=0$, we get 
\[
u_\vep(x)=\vep \log{\cosh(1/\vep)\over \cosh(x/\vep)}=\vep \left(\log \cosh(1/\vep)-\log \cosh(x/\vep)\right).
\]
Since
\[
C_1e^{1-|x|\over \vep}\leq {\cosh(1/\vep)\over \cosh(x/\vep)}\leq C_2e^{1-|x|\over \vep}
\]
for $\vep>0$ small, letting $\vep\to 0$, we have $u_\vep(x)\to 1-|x|$ for all $x\in [-1, 1]$. 

\item[3.2] Let us prove ``$\Rightarrow$''.  Suppose that there exist a test function $\psi\in C^1(\Omega)$ and $x\in \Omega$ such that $U-\psi$ attains a local maximum at $x_0$. Then letting $\psi_C=\psi+C$ for some $C\in \R$, we have $U-\psi_C$ attains a maximum at $x_0$ with $U(x_0)=\psi_C(x_0)<0$. It follows that $u(x)-\phi(x)$ attains a local maximum at $x_0$, where $\phi(x)=-\log(-\psi_C(x))$. Then by the definition of viscosity subsolution, we have $|\nabla \phi(x_0)|\leq f(x_0)$, which, in terms of $\psi$, can be expressed as 
\[
-{|\nabla \psi(x_0)|\over (\psi(x_0)+C)}\leq f(x_0).
\]
It follows that $|\nabla \psi(x_0)|+f(x_0) U(x_0)\leq 0$. One can use a symmetric argument to show that $U$ is a viscosity supersolution of $|\nabla U|+f(x) U=0$. The reverse implication ``$\Leftarrow$'' can also be similarly proved.

\item[3.3] It is not difficult to show that $w\in \usc(\Omega)$ by using the upper semicontinuity of $u$ and $v$. Suppose that $w-\psi$ attains a local maximum at $x\in \Omega$ for some $\psi\in C^1(\Omega)$. Then $u-\psi$ and $v-\psi$ both attain a local maximum at $x$. It follows from the definition of subsolutions that $H(x, u(x), \nabla\psi)\leq 0$. 

In general, $\min\{u(x), v(x)\}$ may not be a viscosity subsolution. Consider the equation $1=|u'|$ in $\R$. Both $u(x)=x$ and $v(x)=-x$ are viscosity solutions, but $\min\{u(x), v(x)\}=-|x|$ is not a viscosity subsolution. The subsolution property fails to hold at $x=0$.

\item[3.4] One can write inequalities associated to DPP with test functions and apply Taylor expansion to complete the verification. Details are omitted.
\end{enumerate}

\medskip

{\bf Section 4}

\begin{enumerate}

\item[4.1] By taking $w(x)=u(x)-\int_a^x f(y)\, dy$, we only need to show that $w$ is a subsolution of $w'=0$ in $I$ if and only if $w$ is nonincreasing in $I$. Let us show ``$\Leftarrow$'' first. Suppose that there exists a test function $\psi\in C^1(I)$ such that $w-\psi$ attains a local maximum at $x\in I$. Since $w$ is nonincreasing, we have $\psi(y)\geq \psi(x)$ for all $y\leq x$ close to $x$. It follows immediately that $\psi'(x)\leq 0$. Let us now prove ``$\Rightarrow$''. Assume by contradiction that there exists $a, b\in I$ such that $w(a)<w(b)$. Let $k={w(b)-w(a)\over b-a}$. It is clear that $k>0$. Fix $c\in I$ with $c>b$ and define $\psi\in C^1(I)$ by
\[
\psi(x)=\begin{cases}
    w(a)+k(x-a) & \text{if $x<b$}\\
     w(a)+k(x-a)+C(x-b)^2 &\text{if $x\geq b$}.
\end{cases}
\]
We can take $C>0$ large so that $\psi(c)>w(c)$. Since 
\[
(w-\psi)(a)=(w-\psi)(b)=0>(w-\psi)(c).
\]
Then there must exist a local maximizer $x$ of $w-\psi$ in $(a, c)$, where $\psi'(x)\geq k>0$. This contradicts the definition of subsolution, which states $\psi'(x)\leq 0$.

\item[4.2] We can use the same argument as in Problem 4.1 to show that $w$ is a supersolution of $w'=0$ if and only if $w$ is nondecreasing in $I$. By changing the orientation of the interval, we see that, if $w-\psi$ attaining a local minimum at $x\in I$ for $\psi\in C^1(I)$ always yields $\psi'(x)\leq 0$, then $w$ is nonincreasing in $I$. We apply this result to $w(x)=u(x)+\int_a^x f(y)\, dy$. Whenever there exist $\psi\in C^1(I)$ and $x_0\in I$ such that $w-\psi$ attains a local minimum, $x\mapsto u(x)-(\psi(x)-\int_0^x f(y)\, dy)$ attains a local minimum at $x_0$. Then $\psi'(x_0)-f(x_0)\leq 0$ due to the condition that $u$ is nonincreasing. Therefore, since $u$ is a supersolution of $|u'|=f$, we get $f(x_0)-\psi'(x_0)\geq f(x_0)$, which yields $\psi'(x_0)\leq 0$. Applying the previous result, we conclude that $u(x)+\int_a^x f(y)\, dy$ is nonincreasing in $I$.

\item[4.3] By the concavity of $w(x)=u(x)-Kx^2$, for any $x_0\in I$, there exists $a\in \R$ such that $w(x)\leq w(x_0)+a(x-x_0)$ holds for all $x$ near $x_0$. It follows that 
\[
u(x)-u(x_0)\leq (a+Kx+Kx_0)(x-x_0)
\]
for all $x$ near $x_0$. Then one can check that $|\nabla^+ u|(x_0)\leq |a+2Kx_0|$ but $|\nabla^- u|(x_0)\geq |a+2Kx_0|$, which by the fact that $|\nabla u|=\max\{|\nabla^+ u|, |\nabla^- u|\}$ yields 
\[
|\nabla u|(x_0)=|\nabla^- u|(x_0)=|a+2Kx_0|.
\]

\end{enumerate}

\medskip

{\bf Section 5}
\begin{enumerate}
\item[5.1] (1) The implication ``$\Rightarrow$'' is trivial. Let us show the reverse implication ``$\Leftarrow$''. For any sequence $y_j\to x$ such that 
\[
\limsup_{y\to x} \frac{\max\{u(x)-u(y), 0\}}{d(x, y)}=\lim_{y_j\to x} \frac{\max\{u(x)-u(y_j), 0\}}{d(x, y_j)}=\max\left\{\lim_{y_j\to x} \frac{u(x)-u(y_j)}{d(x, y_j)}, 0\right\}
\]
(the second equality is due to the continuity of $a\mapsto \max\{a, b\}$),
we have
\[
\lim_{y\to x} \frac{u(x)-u(y_j)}{d(x, y_j)}\leq \limsup_{y\to x} \frac{u(x)-u(y)}{d(x, y)}\leq a,
\]
It follows immediately that 
\[
|\nabla^- u|(x)= \limsup_{y\to x} \frac{\max\{u(x)-u(y), 0\}}{d(x, y)}\leq a.
\]
(2) This time ``$\Leftarrow$'' is trivial. We prove ``$\Rightarrow$''. Since $|\nabla^- u|(x)\geq a>0$, there exists a sequence $y_j\to x$ such that 
\[
\lim_{y_j\to x} \frac{\max\{u(x)-u(y_j), 0\}}{d(x, y_j)}=a>0.
\]
It follows that 
\[
\limsup_{y_j\to x} \frac{u(x)-u(y_j)}{d(x, y_j)}=a,
\]
which implies that 
\[
\limsup_{y\to x} \frac{u(x)-u(y)}{d(x, y)}\geq a.
\]
Note that in general without the condition $a>0$, we cannot obtain ``$\Rightarrow$''. One simple example is that $u(x)=|x|$ for $x\in \R$. It is clear that $|\nabla^- u|(0)=0$ but 
\[
\limsup_{y\to 0} \frac{u(0)-u(y)}{|y|}=-1.
\]
\item[5.2] One can directly check the slopes of $u_1$ and $u_2$ to prove the statements. Note that $|\nabla^- u|(0)=0$.
\end{enumerate}

\clearpage

\bibliographystyle{abbrv}

\end{document}